\documentclass[12pt, reqno]{amsart}

\usepackage[margin=1in]{geometry}

\usepackage{paralist}
\usepackage{amsmath,xcolor, amscd, amsthm, amssymb}
\usepackage[hyperfootnotes=false, colorlinks, linkcolor={blue}, citecolor={magenta}, filecolor={blue}, urlcolor={blue}, pdfpagelabels]{hyperref}

\usepackage{bm} 
\usepackage{chngcntr}
\counterwithin{equation}{section}

\def\C{{\mathbf C}}
\def\R{{\mathbf R}}
\def\Z{{\mathbf Z}}
\def\Q{{\mathbf Q}}
\def\A{{\mathbf A}}
\def\H{{\mathbf H}}
\def\O{{\mathbb O}}

\def\ZZ{\widehat{\mathbf Z}}

\newtheorem{theorem}{Theorem}[subsection]
\newtheorem{conjecture}[theorem]{Conjecture}
\newtheorem{thm}[theorem]{Theorem}
\newtheorem{lemma}[theorem]{Lemma}

\newtheorem{proposition}[theorem]{Proposition}
\newtheorem{prop}[theorem]{Proposition}
\newtheorem{corollary}[theorem]{Corollary}

\newtheorem{claim}[theorem]{Claim}

\theoremstyle{definition}
\newtheorem{definition}[theorem]{Definition}

\theoremstyle{remark}
\newtheorem{remark}[theorem]{Remark}

\newcommand{\mm}[4]{\left(\begin{smallmatrix} #1 & #2\\ #3 & #4\end{smallmatrix}\right)}
\newcommand{\mb}[4]{\left(\begin{array}{cc} #1 & #2\\ #3 & #4\end{array}\right)}

\def\Lie{{\mathrm {Lie}}}

\DeclareMathOperator{\tr}{tr}
\DeclareMathOperator{\ind}{Ind}
\DeclareMathOperator{\SO}{SO}
\DeclareMathOperator{\Spin}{Spin}

\DeclareMathOperator{\Sp}{Sp}

\DeclareMathOperator{\PGSp}{PGSp}

\DeclareMathOperator{\SU}{SU}

\DeclareMathOperator{\SL}{SL}
\DeclareMathOperator{\GL}{GL}
\DeclareMathOperator{\PGL}{PGL}

\DeclareMathOperator{\diag}{diag}

\DeclareMathOperator{\Span}{Span}

\def\g{{\mathfrak g}}
\def\h{{\mathfrak h}}
\def\k{{\mathfrak k}}
\def\p{{\mathfrak p}}

\def\so{{\mathfrak {so}}}
\def\sl{{\mathfrak {sl}}}

\def\su{{\mathfrak {su}}}

\def\H{{\mathcal{H}}}

\def\linspan{\textnormal{-span}}

\begin{document}
\title{The quaternionic Maass Spezialschar on split $\SO(8)$}
\author[J. Johnson-Leung]{Jennifer Johnson-Leung}
\address{Department of Mathematics and Statistical Science\\ University of Idaho\\ Moscow, ID~USA}
\email{jenfns@uidaho.edu}
\author[F. McGlade]{Finn McGlade}
\address{Department of Mathematics\\ The University of Oklahoma, Norman, OK USA}
\email{finn.mcglade@ou.edu}
\author[I. Negrini]{Isabella Negrini}
\address{Department of Mathematics\\ University of Toronto\\ Toronto, ON Canada}
\email{isabella.negrini@utoronto.ca}

\author[A. Pollack]{Aaron Pollack}
\address{Department of Mathematics\\ The University of California, San Diego\\ La Jolla, CA USA}
\email{apollack@ucsd.edu}

\author[M. Roy]{Manami Roy}
\address{Department of Mathematics\\ Lafayette College\\ Easton, PA USA}
\email{royma@lafayette.edu}

\begin{abstract} The classical Maass Spezialschar is a Hecke-stable subspace of the space of holomorphic Siegel modular forms of genus two and level one cut out by certain linear relations among Fourier coefficients.  We define an analogous quaternionic Maass Spezialschar, which consists of the quaternionic modular forms of level one on split $\SO(8)$ whose Fourier coefficients satisfy certain linear relations.  We characterize this space in terms of a theta lift from the space of holomorphic Siegel modular forms on $\Sp(4)$, and in terms of periods.  We also give a conjecture for the Dirichlet series of the standard $L$-function of quaternionic modular eigenforms on $\SO(8)$ and verify our conjecture on the quaternionic Maass Spezialschar.
\end{abstract}


\maketitle

\setcounter{tocdepth}{1}
\tableofcontents

\section{Introduction}
The classical Saito-Kurokawa subspace of Siegel modular forms on $\PGSp(4)$ has many characterizations.  One way to define the Saito-Kurokawa subspace is in terms of theta lifts from half-integral weight modular forms on $\mathrm{Mp}(2)$  (the metaplectic cover of $\SL(2)$) to $\SO(5) \simeq \PGSp(4)$.  The image of this theta lift can be characterized in terms of Fourier coefficients, yielding the classical identification of the Saito-Kurokawa subspace with the Maass Spezialschar as recalled in Section~\ref{sec:classical}.  With this definition, it becomes an interesting question to ask whether (semi)-classical characterizations, such as the Maass relations, apply to generalizations of the Saito-Kurokawa lifting.

In the present work, we give an affirmative answer to this question for the theta lift from $\Sp(4)$ to $\SO(8)$, and present a generalization of the theory of the Maass Spezialschar of Siegel modular forms to the setting of quaternionic modular forms on $\SO(8)$.  In the process, we develop a notion of Fourier-Jacobi coefficients for quaternionic modular forms on groups of type $\SO(4,n+2)$.  We now summarize the results of this paper.

\subsection{The Fourier-Jacobi coefficient}
Fix $n\in \Z_{\geq 2}$ and let $V$ be a non-degenerate quadratic space over $\Q$ of signature $(4,n+2)$. Assume $V$ contains an isotropic subspace of dimension $4$ and set $G = \SO(V)$.  There is a special class of automorphic forms on $G$, called quaternionic modular forms, corresponding to minimal $K$-type vectors in the quaternionic discrete series on $G(\R)$ \cite{grossWallachII}. Concretely, if $\ell \in \Z_{\geq 1}$ then a quaternionic modular form of weight $\ell$ is a vector-valued automorphic function $\varphi\colon G(\Q)\backslash G(\A)\to \mathbf{V}_{\ell}$ where $\mathbf{V}_{\ell}\simeq \mathrm{Sym}^{2\ell}\C^2$ is the $(2\ell+1)$-dimensional irreducible representation of $\SU(2)$. \\
\indent In analogy with the case of holomorphic modular forms, a weight $\ell$ quaternionic modular form $\varphi$ satisfies a specific ``Cauchy-Riemann type" equation $D_{\ell}\varphi\equiv 0$ (see Definition~\ref{defn:QMFs1}). This differential equation leads to a semi-classical theory of Fourier coefficients which we now recall. Let $P = M_P N_P$ denote the Heisenberg parabolic subgroup of $G$, defined as the stabilizer in $G$ of a fixed isotropic two-plane $U\subseteq V$. The Levi factor $M_P\simeq \GL(U)\times \SO(V_{2,n})$ where $V_{2,n}\subseteq V$ is a non-degenerate quadratic subspace of signature $(2,n)$. The characters of $N_P(\Q)\backslash N_P(\A)$ are indexed by ordered pairs $[T_1,T_2]$ consisting of two elements $T_1,T_2\in V_{2,n}$. Given a quaternionic modular form $\varphi$, the results of \cite{wallach} and \cite{pollackQDS} give a family of complex numbers $a_{\varphi}([T_1,T_2]) \in \C$ for every $[T_1,T_2] \in V_{2,n}^{\oplus 2}$ which we refer to as the \textit{quaternionic Fourier coefficients of $\varphi$} (see Definition \ref{Definition-of-FCs}).

 Recall that a holomorphic Siegel modular form $F$ of genus $2$ and level one has a Fourier-Jacobi expansion of the form 
 $$
 F\left(\begin{pmatrix} \tau &z \\ z &\tau'\end{pmatrix}\right)=\sum_{m=0}^{\infty}\phi_m(\tau,z)e^{2\pi i m\tau'}
 $$
 (see \cite[Theorem 6.1]{EichZag}). Here $\tau, \tau', z\in \C$ satisfy $\mathrm{Im}(\tau),\ \mathrm{Im}(\tau')>0$ and $\mathrm{Im}(\tau)\mathrm{Im}(\tau')>\mathrm{Im}(z)^2$, and $\phi_m$ is a Jacobi form of index $m$. 
 In our first main result we study an analogue of the ``first" Fourier-Jacobi coefficient $\phi_1(\tau,z)$ in the setting of quaternionic modular forms on $G$. To be more precise, fix a line $\Q b_1\subseteq U$ and let $Q= M_Q N_Q$ be the parabolic subgroup in $G$ stabilizing $\Q b_1$. Given a quaternionic modular form $\varphi$ on $G$, we define \textit{the first Fourier-Jacobi coefficient $\mathrm{FJ}_\varphi(h)$ of $\varphi$} as
  \begin{equation}
 \label{1st-FJ-coefficients} 
 \mathrm{FJ}_\varphi(h) = \overline{\mathcal{L}\left(\int_{[N_Q]}{\chi_{y_0}^{-1}(n) \varphi(nh)\,dn}\right)}
 \end{equation}
 for a certain linear functional $\mathcal{L}: \mathbf{V}_{\ell} \rightarrow \C$, given in Corollary~\ref{Main Corollary}.  Here $\chi_{y_0}$ is a specific non-degenerate character of $N_Q(\Q)\backslash N_Q(\A)$. The formula \eqref{1st-FJ-coefficients} defines $ \mathrm{FJ}_\varphi$ as an automorphic function on the stabilizer $H=\mathrm{Stab}_{M_Q}\{\chi_{y_0}\}$. Moreover, for our particular choice of $\chi_{y_0}$, we have $H(\R)\simeq \SO(2, n+1)$, and so it makes sense to ask whether $\mathrm{FJ}_\varphi$ is (the automorphic function associated to) a holomorphic modular form on $H$. Our first main theorem settles this question in the affirmative.
\begin{theorem}\label{thm:introFJgeneral} Let the notation be as above, so that $\varphi$ is a weight $\ell$ quaternionic modular form on $G$. Then $\mathrm{FJ}_{\varphi}$ is the automorphic function corresponding to a weight $\ell$ holomorphic modular form on $H$.  Moreover, the classical Fourier coefficients of $\mathrm{FJ}_{\varphi}$ are finite sums of the complex conjugates of the quaternionic Fourier coefficients of $\varphi$.\end{theorem}

\subsection{$D_4$ modular forms}
The remaining theorems in this paper are specialized to the case of $n=2$. In this case $G\simeq \SO(8)$ is split and $H\simeq \PGSp(4)$ is also split. Moreover, the quadratic space $V_{2,n}$ can be identified with the space of $2$-by-$2$ rational matrices $M_2(\Q)$. For a form $\varphi$ of level one, the Fourier coefficients $a_{\varphi}([T_1,T_2])$ can be nonzero only if $[T_1, T_2]$ is a pair consisting of $2$-by-$2$ matrices $T_1$ and $T_2$ with integer entries. \\
\indent In \cite{weissmanD4}, Weissman develops a theory of Fourier coefficients for quaternionic modular forms on the two-fold simply connected covering group $\Spin(8)\to G$. Theorem~\ref{thm:introFJgeneral}, while stated for special orthogonal groups, applies equally well to the group $\Spin(8)$.  Specializing Theorem~\ref{thm:introFJgeneral} to this case and applying a triality automorphism on $\Spin(8)$ (as discussed in Appendix~\ref{sec:triality}) we obtain the following application. 

\begin{theorem}\label{thm:introSp4} Suppose $\varphi$ is a quaternionic modular form on $\Spin(8)$ of level one and weight $\ell$ with Fourier coefficients $a_\varphi([T_1,T_2])$.  For integers $a,b,c$, define 
\[b_\varphi\left([a,b,c]\right) = \overline{a_\varphi\left(\left[\mm{a}{0}{b}{1},\mm{0}{-1}{c}{0} \right]\right)}.\]
Then the $q$-series 
\[f(Z)=\sum_{a,b,c \in \Z}{b_\varphi\left([a,b,c]\right) e^{2\pi i \tr\left(\mm{a}{b/2}{b/2}{c} Z\right)}}\]
is the Fourier expansion of a Siegel modular form on $\Sp(4)$ of weight $\ell$ and level one. Moreover, if $\varphi$ is cuspidal then so is $f$.\end{theorem}

\subsection{The quaternionic Maass Spezialschar}
Recall that now $G = \SO(V)$ with $V$ of dimension $8$.  Suppose $f$ is a weight $\ell$ cuspidal holomorphic Siegel modular form on $\Sp(4)$.  For test data $\phi \in S(V(\A)^2)$ for the Weil representation restricted to the dual pair $G \times \Sp(4)$, one can consider the theta lift $\theta_\phi(f)$, which is an automorphic form on $G$.  In \cite{pollackCuspidal}, the fourth-named author chooses special test data $\phi_{\ell} \in S(V(\A)^2) \otimes \mathbf{V}_{\ell}$ for the Weil representation.  Set $\theta^*(f) = \theta_{\phi_{\ell}}(f)$.  It turns out that $\theta^*(f)$ is a cuspidal quaternionic modular form on $G$ of weight $\ell$, and its Fourier coefficients are computed in \cite[Theorem 4.1.1]{pollackCuspidal} in terms of the classical Fourier coefficients of $f$.

For an integer $\ell \geq 1$, we write $\mathcal{A}_{0,\Z}(G,\ell)$ for the space of cuspidal quaternionic modular forms on $G = \SO(8)$ of weight $\ell$ and level one (see Section~\ref{subsec:Definition-QMFS}). We define the \emph{weight $\ell$ quaternionic Saito-Kurokawa subspace} $\mathrm{SK}_{\ell}$ as the subspace of $\mathcal{A}_{0,\Z}(G,\ell)$ spanned by $\theta^*(f)$ as $f$ ranges over the space of cuspidal holomorphic Siegel modular forms on $\Sp(4)$ of weight $\ell$ and level one.  

Our main application of Theorem~\ref{thm:introSp4} is to characterize the space $\mathrm{SK}_{\ell}$ in terms of Fourier coefficients.  More precisely, we define a subspace $\mathrm{MS}_{\ell}$ consisting of forms $\varphi\in \mathcal{A}_{0,\Z}(G,\ell)$ for which the Fourier coefficients $a_{\varphi}([T_1,T_2])$ satisfy a particular system of linear relations (see Definition~\ref{definition-Quaternionic-Spezialschar}).  We call $\mathrm{MS}_{\ell}$ the quaternionic Maass Spezialschar of weight $\ell$. The definition of $\mathrm{MS}_{\ell}$ is analogous to the classical definition of the Maass Spezialschar on $\PGSp(4)$.

It follows easily from \cite{pollackCuspidal} that, with our definition of $\mathrm{MS}_{\ell}$, there is an inclusion $\mathrm{SK}_{\ell} \subseteq \mathrm{MS}_{\ell}$.  Using Theorem~\ref{thm:introSp4}, we are able to prove that this inclusion is an equality.  Specifically, given an element $\varphi \in \mathrm{MS}_{\ell}$, if we want to prove that $\varphi \in \mathrm{SK}_{\ell}$, then we must produce a Siegel modular form $f$ on $\Sp(4)$ of level one such that $\varphi = \theta^*(f)$.  This $f$ is exactly given by the $q$-series of Theorem~\ref{thm:introSp4}.  One arrives at one of the main theorems of this paper.

\begin{theorem}\label{thm:intoSKMS} Suppose $\ell \geq 16$ is even.  Then the inclusion $\mathrm{SK}_{\ell} \subseteq \mathrm{MS}_{\ell}$ is an equality. \end{theorem}

Let us remark explicitly that the composite map $f \mapsto \theta^*(f) \mapsto \mathrm{FJ}_{\theta^*(f)}$, sending the space of Siegel modular forms of level one  and genus $2$ to itself is not the identity map.  In order to recover $f$ from $\varphi = \theta^*(f)$, we instead proceed as follows.  Let $\varphi_1$ denote the pullback of $\varphi$ to $\Spin(8)$ and set $\varphi_2 = \sigma(\varphi_1)$ for $\sigma$ an order three automorphism of $\Spin(8)$ coming from triality.  Then we find $f$ as a Fourier-Jacobi coefficient of $\varphi_2$.

\subsection{Periods of quaternionic modular forms}
Theorem~\ref{thm:intoSKMS} characterizes the Saito-Kurokawa subspace in terms of Fourier coefficients.  As we describe below, this subspace can also be characterized using periods of automorphic forms.

Let $L \subseteq V$ be a fixed choice of a split, unimodular lattice.  For $v_1, v_2 \in L$ spanning a positive-definite two-plane, define $H_{v_1,v_2} \subseteq G$ as the pointwise stabilizer of the vectors $v_1$ and $v_2$. Given a quaternionic modular form $\varphi$ on $G$ of weight $\ell$ and level one, we define 
\[\mathcal{P}_{v_1,v_2}(\varphi) = \int_{H_{v_1,v_2}(\Z)\backslash H_{v_1,v_2}(\R)}{ \varphi(h) \,dh},\]
the period of $\varphi$ over $H_{v_1,v_2}(\Z)\backslash H_{v_1,v_2}(\R)$.  We define an integer $D(v_1,v_2)$ as 
$$D(v_1,v_2)=(v_1,v_2)^2 - (v_1,v_1)(v_2,v_2).$$
Finally, let $k_0 \in O(L)(\Z)$ have determinant $-1$.  Say that $\varphi$ is in the plus subspace if $\varphi(g) = \varphi(k_0 g k_0^{-1})$.  (The right-hand side of this expression is independent of the choice of $k_0$.)  One has that $\mathrm{SK}_{\ell}$ is contained in the plus subspace.

\begin{theorem}\label{thm:introPeriod} Suppose $\ell \geq 22$ is even, and $\varphi \in \mathcal{A}_{0,\Z}(G,\ell)$ is a Hecke eigenform in the plus subspace.  Then $\varphi \in \mathrm{SK}_{\ell}$ if and only if there exists $v_1, v_2 \in L$ such that $\mathcal{P}_{v_1,v_2}(\varphi) \neq 0$ and $D(v_1,v_2)$ is odd and square-free.\end{theorem}

To put Theorem~\ref{thm:introPeriod} into context, let $L(\varphi)$ be a ``lift" of $\varphi$ from the special orthogonal group $\SO(V)$ to the orthogonal group $O(V)$ as in Proposition~\ref{proposition-restriction-to-SO(V)}, and let $\Pi_{L(\varphi)}$ be the automorphic representation generated by $L(\varphi)$. In Theorem~\ref{thm:TFAE} we verify that $\varphi$ is a quaternionic Saito-Kurokawa lift if and only if the representation $\Pi_{L(\varphi)}$ has a nonzero theta lift to $\Sp(4)$.  It is standard that the representation $\Pi_{L(\varphi)}$ has a nonzero theta lift to $\Sp(4)$ if and only if the representation $\Pi_{L(\varphi)}$ is distinguished by a subgroup of the form $H_{v_1, v_2}$, for some $v_1, v_2$.  This follows from \cite[section 4]{howePS} together with the non-existence of singular cusp forms \cite{liJS-Singular}.   In other words, by collecting results from the literature on the theta correspondence, one sees that $\varphi$ is a Saito-Kurokawa lift if and only if there is some automorphic form $\varphi'$ in the representation $\Pi_{L(\varphi)}$ that is distinguished by an algebraic subgroup of the form $H_{v_1, v_2}$.  Theorem \ref{thm:introPeriod} refines this fact by showing that it is enough to check if the function $\varphi$ itself is distinguished.  

Theorem~\ref{thm:introPeriod} can be considered a non-holomorphic analogue of (a consequence of) the result \cite[Proposition 2.1]{shintani1975} of Shintani, which considers the theta lift from holomorphic cusp forms on $\PGL(2) \simeq \SO(2,1)$ to the metaplectic group $\mathrm{Mp}(2)$, and is closely related to the more general results of Funke-Millson, especially  \cite[Theorem 1.1 and Example 1.3]{FM06}.  In the course of proving Theorem~\ref{thm:introPeriod}, we compute the Petersson inner product of a cuspidal quaternionic modular form $\varphi$ of weight $\ell$ with a quaternionic modular form $Q_{S,\ell} = \theta^{*}(P_{S,\ell})$, where $P_{S,\ell}$ is a holomorphic Siegel modular Poincar\'e series of weight $\ell$.  One aspect of the proof of Theorem~\ref{thm:introPeriod} that we highlight is the relatively straightforward way in which we compute the inner product $\langle \varphi, Q_{S,\ell}\rangle$, at least modulo technical considerations concerning the convergence of certain archimedean integrals.

Combining Theorem~\ref{thm:intoSKMS} with Theorem~\ref{thm:introPeriod} gives a way to detect if cuspidal quaternionic modular eigenforms have non-vanishing periods in terms of relations between their Fourier coefficients.  We remark that one should also be able to characterize the Hecke eigenforms in the quaternionic Maass Spezialschar via a pole at $s=2$ in their standard $L$-functions.  See \cite{GinzburgJiangSoudry09, GanTakeda11, Yamana14, Evdokimov1980, Oda1981}.

The remainder of the paper is organized as follows. In Section~\ref{sec:classical}, we recall some well-known facts about the classical Saito-Kurokawa lift and Maass Spezialschar.  In Section~\ref{sec:notation}, we define various notations that we use throughout the paper. In Section~\ref{sec:QMFs}, we define quaternionic modular forms on $G$ and review some results about the Fourier expansion of quaternionic modular forms. In Section~\ref{sec:QSKL}, we define 
the quaternionic Saito-Kurokawa lifting and quaternionic Maass Spezialschar, and show that the quaternionic Saito-Kurokawa subspace $\mathrm{SK}_{\ell}$ is contained in the quaternionic Maass Spezialschar $\mathrm{MS}_{\ell}$. In Section~\ref{sec:dirichlet}, we discuss a conjectural Dirichlet series for the $L$-functions of irreducible, cuspidal, quaternionic automorphic representations of $G$ and show that this conjecture is satisfied by the Saito-Kurokawa lifts. In Section~\ref{sec:FJ}, we show that the Fourier-Jacobi coefficient of a quaternionic modular form on $G$ is a Siegel modular form on $\Sp(4)$ of the correct weight (see Corollary~\ref{Main Corollary}), and using this result we obtain $\mathrm{MS}_{\ell}\subseteq \mathrm{SK}_{\ell}$. In Section~\ref{sec:Hecke}, we collect results from the literature on theta lifts that we use in Section~\ref{sec:periods}.  In Section~\ref{sec:periods}, we prove Theorem~\ref{thm:introPeriod}.  The paper ends with three appendices: Appendix~\ref{sec:triality} contains results about the triality on $\Spin(8)$ and Appendices~\ref{sec:orbits} and \ref{sec:finiteness_appendix} contain technical results used in the proof of Theorem~\ref{thm:introPeriod}.

\subsection{Acknowledgments}
This project began at the Research Innovations and Diverse Collaborations workshop at the University of Oregon, which was supported by NSF CAREER grant DMS-1751281 (E. Eischen) and NSA MSP conference grant H98230-21-1-0029 (E. Eischen).  The authors thank the organizers of this workshop for their generous support in establishing this collaboration. Work was also done at a SQuaRE hosted by the American Institute for Mathematics. The authors thank AIM for providing a supportive and mathematically rich environment.  We thank the anonymous referees for their careful reading of our manuscript and helpful suggestions for improvement.

JJL was partially supported by the Renfrew Faculty Fellowship from the University of Idaho College of Science.
IN was partially supported by a FRQNT Scholarship by Fonds de Recherche du Québec and by NSF Grant No. DMS-1928930 while the author was in residence at MSRI in Berkeley, California, during the Spring 2023 semester.
AP has been supported by the Simons Foundation via Collaboration Grant number 585147 and by the NSF via grant numbers 2101888 and 2144021. 
MR was partially supported by the AMS Simons Travel Grant program.
\section{The classical Saito-Kurokawa lift and Maass Spezialschar}\label{sec:classical}

In this introductory section, we recall the classical Saito-Kurokawa lift to Siegel modular forms and its image, the Maass Spezialschar. We describe characterizations of this subspace via linear relations on Fourier coefficients, the Saito-Kurokawa theta lift, and an inverse map via the first Fourier-Jacobi coefficient. As explained in the introduction, these characterizations will be generalized to the quaternionic Maass Spezialschar in $\SO(8)$ in subsequent sections.

Let $\mathcal{M}_k^2(\Sp(4,\Z))$ denote the space of Siegel modular forms of even weight $k$ and level one.
Recall that $F \in \mathcal{M}_k^2(\Sp(4,\Z))$ has a Fourier expansion of the form
	
	\[
		F(Z)=\sum_{T\in S_2(\Z)^\vee}a_F(T)e^{2\pi i \operatorname{Tr}(ZT)},
	\]
	where
	\begin{equation*}
	\begin{aligned}
		S_2(\Z)^{\vee}&=\left\{T\in S_2(\R):\operatorname{Tr}(XT)\in\Z \text{ for all }X\in S_2(\Z)\right\}%\\&=\left\{T\in S_2(\Q):t_{ii}\in\Z,\;t_{ij}\in\frac{1}{2}\Z\text{ for }i\neq j\right\},
		\end{aligned}
			\end{equation*}
 and $Z\in \mathcal{H}_2$, the Siegel upper half space of degree 2. Here $S_2(A)$ denotes symmetric $2\times 2$ matrices with coefficients in a ring $A$.  The Fourier coefficients $a_F(T)$ are $0$ if $T$ is not positive semi-definite. If $T \in S_2(\Z)^\vee$ is positive semi-definite, we can write 
$$
T=\begin{pmatrix}
    a & b/2\\ b/2 & c
\end{pmatrix}, \:\:a,b,c\in\Z, \:\:\:a,c\geq 0,\:\: b^2\leq 4ac.
$$
 The \emph{Maass Spezialschar} is the subspace $\mathrm{MS}_k(\Sp(4,\Z))$ of $S_k^2(\Sp(4,\Z))$ given by all cusp forms $F$ such that 
    $$
a_F\begin{pmatrix} a &b/2 \\ b/2 &c \end{pmatrix}=\sum_{d|\gcd(a,b,c)}d^{k-1}a_F\begin{pmatrix} ac/d^2 &b/(2d) \\ b/(2d) &1\end{pmatrix}.
    $$
This space was first studied by Maass \cite{Maass1, Maass2} building on experimental evidence found by Resnikoff and Salda\~na \cite{RS}.

The Maass Spezialschar is isomorphic to the \emph{Kohnen plus space}  $S_{k-1/2}^+$, which for $k$ even is the space of 
 cusp forms of weight $k-\frac{1}{2}$ on $\Gamma_0(4)$ having a Fourier development of the form 
 $$  
 g(z) = \sum_{n\geq 1} c(n) q^n, \mbox{ with } 
 \qquad c(n)=0 \mbox{ unless } n \equiv 0 \mbox{ or } 3 
 \pmod{4}.
 $$ 

 \begin{proposition}[Eichler, Zagier]
 \label{EichZagProp}
     There is a Hecke-equivariant linear isomorphism from $S_{k-1/2}^+$ to $\mathrm{MS}_k(\Sp(4,\Z))$ sending
     $$
     g(z) = \sum_{n\geq 1} c(n) q^n \in S_{k-1/2}^+
     $$   
    to
    $$
    \sum_{T>0}a_F(T)e^{2\pi i \operatorname{Tr}(ZT)},
    $$
    where
    $$
    a_F\begin{pmatrix} a &b/2 \\ b/2 &c \end{pmatrix}=\sum_{d|\gcd(a,b,c)}d^{k-1}c\Big(\frac{4ac-b^2}{d^2}   \Big).
    $$
 \end{proposition}

The map from Proposition~\ref{EichZagProp} is called the \emph{Saito-Kurokawa lift}. It has the kernel function 
$$
\Omega_k(Z,\tau)=\sum_{N>0}N^{3/2-k}w_{N,k}(Z)e^{2\pi iN\tau}, \:\:\:\:\text{with }Z\in \mathcal{H}_2 ,\:\tau \in \mathcal{H} \text{ (the upper half-plane)},
$$
where 
$$
w_{N,k}(Z)=\sum_{\substack{(a,b,c,d,e)\in\Z^5\\4bd-c^2-4ae=N}}\frac{1}{(a(\tau \tau^{'}-z^2)+b\tau+cz+d\tau^{'}+e)^k},\:\:\:\:\:\:\:Z=\begin{pmatrix} \tau & z\\ z &\tau^{'}\end{pmatrix}.
$$
More precisely, for $k\geq 8$ even, $\Omega_{k}(Z,\tau)$ belongs to $\mathcal{M}_k^2(\Sp(4,\Z))$ as a function of $Z$, and to $ S_{k-1/2}(\Gamma_0(4))$ as a function of $\tau$.  As explained in Section 3 of \cite{ZaFrench}, the Saito-Kurokawa lift $\mathcal{SK}:S_{k-1/2}(\Gamma_0(4)) \rightarrow \mathcal{M}_k^2(\Sp(4,\Z))$ is given via the Petersson inner product by
$$
\mathcal{SK}(h)(Z)=\int_{\Gamma_0(4)\backslash\mathcal{H}} h(u+iv){\Omega_k(Z,u-iv)}v^{k-1/2}\frac{dudv}{v^2}. 
$$
The function $\Omega_k$ is the holomorphic projection of a theta series attached to $\SO(2,3)$ \cite{ZaFrench}, and the Saito-Kurokawa lift is the theta lift from $\mathrm{Mp}(2)$ to $\SO(2,3) \simeq \PGSp(4)$.

Let $F\in S_k^2(\Sp(4,\Z))$. From the point of view of general theta correspondences, the fact that $F=\mathcal{SK}(h)$ for some $h\in S_{k-1/2}(\Gamma_0(4))$ is equivalent to the fact that the theta lift of $F$ to $\mathrm{Mp}(2)$ is nonzero (see \cite[Theorem 2.2]{PS}).
 The space $S_{k-1/2}^+$ is related to the space $S_{2k-2}(\SL(2,\Z))$ of weight $2k-2$ cusp forms for $\SL(2,\Z)$ by the Shimura lifting $\mathcal{S}:S_{k-1/2}^+\rightarrow S_{2k-2}(\SL(2,\Z))$  (see \cite{Shimura1973HalfIntegral}), which can be given as
 $$
 \mathcal{S}:\sum_{n\geq 1}c(n)q^n \mapsto \sum_{n\geq 1}(\sum_{d|n} d^{k-1}c(n^2/d^2))q^n.
 $$
Moreover, the isomorphism between $S_{2k-2}(\SL(2,\Z))$ and $S_{k-1/2}^+$ is Hecke equivariant \cite{Kohnen}. \\
 \indent 
For $k$ even the Maass Spezialschar $\mathrm{MS}_k(\Sp(4,\Z))$ is also isomorphic to the space $J_{k,1}$ of Jacobi forms of weight $k$ and index $1$. We can see this in two ways. The first way follows from the theorem below \cite[Theorem 5.4]{EichZag}, which states that $J_{k,1}$ is isomorphic to the Kohnen plus space. But by Proposition \ref{EichZagProp} this space is isomorphic to the Maass Spezialschar, hence the Maass Spezialschar is isomorphic to $J_{k,1}$. 

\begin{theorem}
\label{Jacobi}
    For $k$ even, the assignment sending
    $$
    g(z) = \sum_{n\geq 1} c(n) q^n \in S_{k-1/2}^+
    $$
    to
    $$
    \sum_{\substack{n,r\in\Z\\ 4n \geq r^2}}c(4n-r^2)q^n\zeta^r \in J_{k,1}
    $$
    gives a Hecke-equivariant isomorphism between $S_{k-1/2}^+$ and $J_{k,1}$.
\end{theorem}

Alternatively, one can associate to a Siegel modular form in the Maass space its first Fourier-Jacobi coefficient, and this gives an isomorphism between $\mathrm{MS}_k(\Sp(4,\Z))$ and $J_{k,1}$. The inverse of the Saito-Kurokawa lift is given by composing this assignment with the one of Theorem \ref{Jacobi} (see \cite{EichZag} for more details).

\section{Preliminaries}\label{sec:notation}
In this section, we set up explicit notation for working with quaternionic modular forms on special orthogonal groups.

\subsection{The underlying quadratic space}
\label{The underlying quadratic space}
Let $n \in \Z_{\geq 2}$ and let $(V,q)$ denote a non-degenerate quadratic space over $\Q$ of signature $(4,n+2)$. Write $(x,y) = q(x+y)-q(x)-q(y)$ for the bilinear form associated to $q$ and assume that $V$ contains an isotropic subspace of dimension $4$. We fix a two-dimensional isotropic subspace $U\subseteq V$. Since $(V,q)$ is non-degenerate we may fix a second isotropic $2$-plane $U^{\vee}$ such that $(\cdot, \cdot)$ defines a perfect pairing between $U$ and $U^{\vee}$. We write $V_{2,n}$ for the orthogonal complement of $U+ U^{\vee}$ in $V$ so that $(V_{2,n},q)$ is a quadratic space of signature $(2,n)$ and 
\begin{equation}
\label{decomposition of V stabilized by MP}
V=U\oplus V_{2,n}\oplus U^{\vee}.
\end{equation} 

Let $\{b_1,b_2\}$ denote a fixed basis of $U$ and write $\{b_{-1},b_{-2}\}$ for the corresponding dual basis of $U^{\vee}$ so that $(b_i,b_{-j}) = \delta_{i,j}$ for $i,j=1,2$. Then $V$ admits an orthogonal decomposition 
\begin{equation}
\label{decomposition of V stabilized by MQ}
V=\Q b_1\oplus V_{3,n+1}\oplus \Q b_{-1}
\end{equation} 
where $V_{3,n+1}=(\Q b_1+\Q b_{-1})^{\perp}$ is a quadratic space of signature $(3,n+1)$. 
Our assumption that $V$ contains an isotropic $4$-plane implies that there exist vectors $b_3,b_4,b_{-3},b_{-4}\in V_{2,n}$ such that $\{b_1,b_2,b_3,b_4\}$ and $\{b_{-1},b_{-2},b_{-3},b_{-4}\}$ give dual bases for a pair of isotropic $4$-planes which are in perfect duality under $(\cdot, \cdot)$.

Let $y_0=b_3+b_{-3}$ and $y_1=b_4+b_{-4}$ and write $V_2^+$ for the $\Q$-rational subspace of $V_{2,n}$ spanned by $y_0$ and $y_1$. Let $V_n^-(\R)$ denote the orthogonal complement of $V_2^+(\R):=V_2^+\otimes_{\Q}\R$ in $V_{2,n}(\R):=V_{2,n}\otimes_{\Q} \R$ so that 
\begin{equation}
\label{orthogonalization V2,n}
V_{2,n}(\R)=V_2^+(\R)\oplus V_n^-(\R).
\end{equation} 

For $j=1,2$ we set $u_j = (b_j + b_{-j})/\sqrt{2} \in V(\R)$ and $u_{-j} =  (b_j - b_{-j})/\sqrt{2}\in V(\R)$.  Thus $\{u_1,u_2, u_{-1},u_{-2}\}$ is an orthonormal basis of $V_{2,n}(\R)^{\perp}$. Similarly we set $v_1 = (b_3+b_{-3})/\sqrt{2}$, $v_2 = (b_4+b_{-4})/\sqrt{2}$, $v_{-1} = (b_3-b_{-3})/\sqrt{2}$ and $v_{-2} = (b_4-b_{-4})/\sqrt{2}$.

For much of the paper, we will be working in the case that $V$ is split over $\Q$ of dimension~$8$.  In this case, we assume that $V$ is endowed with an integral structure that is unimodular for the pairing $(\cdot ,\cdot)$.  More specifically, we endow $V_{2,2}$ with the integral structure $V_{2,2}(\Z)=\Z\linspan\{b_3,b_4,b_{-4},b_{-3}\}$, and set
\begin{equation}
\label{eqn-defn-V(Z)}
V(\Z) = \Z b_1 \oplus \Z b_2 \oplus V_{2,2}(\Z) \oplus \Z b_{-1} \oplus \Z b_{-2}.
\end{equation}

\subsection{Octonions}\label{subsec:octonions}
 When $\dim(V) = 8$, we will identify $V$ with the Zorn model of the split octonion algebra over $\Q$ (see \cite[\S 1.8]{springerVeldkamp}) 
\[
\O:=\Big\{\begin{pmatrix} a & v \\\phi & d\end{pmatrix} : \: a,d\in \Q; \: v\in V_3, \: \phi \in    V_3^\vee     \Big\}.
\]
Here $V_3$ is the $3$ dimensional representation of $\SL(3)$ with basis $\{e_1,e_2,e_3\}$, and $V_3^{\vee}$ is the dual of $V_3$. We write $\{e_1^{\ast},e_2^{\ast},e_3^{\ast}\}$ for the basis of $V_3^{\vee}$ dual to $\{e_1,e_2,e_3\}$. The quadratic norm on $\O$ is denoted $n_{\O}$, and we write $c$ for the standard involution on $\O$. We have the \emph{trace map} $
tr_{\O}:\O \rightarrow \Q, \quad tr_{\O}(x)\cdot 1=x+c(x)$, as well as the trilinear form
\[
( \cdot,\cdot ,\cdot ):\O\times\O\times\O \rightarrow \Q, \quad (x_1,x_2 ,x_3 ):=tr_{\O}(x_1\cdot(x_2\cdot x_3)).
\]
In a slight abuse of notation, for $j=1,2,3$, we set $e_j:=\left(\begin{smallmatrix} 0 & e_j \\0 & 0\end{smallmatrix}\right)$, $e_j^{\ast}:=\left(\begin{smallmatrix} 0 & 0 \\e_j^{\ast} & 0\end{smallmatrix}\right)$,
$\epsilon_1:=\left(\begin{smallmatrix} 1 & 0 \\0 & 0\end{smallmatrix}\right)$, and $\epsilon_2:=\left(\begin{smallmatrix} 0 & 0 \\0 & 1\end{smallmatrix}\right)$. Then $\{e_j, e_j^{\ast}, \epsilon_1, \epsilon_2\colon j=1,2,3\}$ gives a basis for $\O$. We view $\O$ as a quadratic space with respect to the quadratic form $q(x) = - n_{\O}(x)$. Then, in the case of $n=2$, $V$ and $\O$ are isomorphic as quadratic spaces via the identification of lists
\begin{equation}
\label{idenitification-of-lists}
(b_1,b_2,b_3,b_4,b_{-4},b_{-3},b_{-2},b_{-1}) = (e_1,e_3^*,\epsilon_2,e_2^*,e_2,-\epsilon_1,e_3,e_1^*).
\end{equation}

\subsection{Algebraic groups, their Lie algebras, and parabolic subgroups}\label{subsec:groupnotation}

Now let $n$ be an arbitrary positive integer and let $G$ denote the special orthogonal group of $(V,q)$, which we assume acts on the left of $V$.  Let $Q$ denote the parabolic subgroup of $G$ defined as the stabilizer in $G$ of the line $\Q b_{1}$. The \textit{Heisenberg parabolic subgroup} $P$ is defined as the subgroup of $G$ stabilizing the space $U= \Span\{b_1,b_2\}$.  Define $M_Q$ to be the Levi subgroup of $Q$ that stabilizes $\Q b_{-1}$ and define the Levi factor $M_P\leq P$ as the subgroup stabilizing the
subspace $U^\vee = \Span\{b_{-1},b_{-2}\}$. Through its action on the decomposition (\ref{decomposition of V stabilized by MP}) the subgroup $M_P$ is identified with the product $\GL(U)\times \SO(V_{2,n})$. Similarly the action of $M_Q$ on the decomposition (\ref{decomposition of V stabilized by MQ}) yields an identification $M_Q\simeq \GL(1)\times \SO(V_{3,n+1})$. Let $N_P$ and $N_Q$ denote the unipotent radicals of $P$ and $Q$, respectively. We set $Z = [N_P,N_P]$, which is one-dimensional and the center of $N_P$.

The Lie algebra of $G$ can be identified with $\wedge^2 V$, where $\wedge^2V$ acts on $V$ via
$$
(v \wedge w) \cdot x = (w,x) v- (v,x) w. \qquad (v,w,x\in V)
$$
We note that $\mathrm{Lie}(Z)=\Q\linspan\{b_1\wedge b_2\}$. As an $M_P$ representation, the Lie algebra of the quotient $N_P^{\mathrm{ab}}=N_P/Z$ is identified as
\begin{equation}
\label{map-to-NPab}
U\otimes_{\Q} V_{2,n}\xrightarrow{\sim} \mathrm{Lie}(N_P^{\mathrm{ab}}), \quad b_1\otimes w_1+b_2\otimes w_2\mapsto b_1\wedge w_1+ b_2\wedge w_2+\mathrm{Lie}(Z). 
\end{equation}

Let $\A$ denote the adele ring of $\Q$ and write $\psi: \Q\backslash \A \rightarrow \C^\times$ for the standard additive character of $\Q\backslash\A$ with $\psi_{\infty}(x)=e^{2\pi i x}$.  For $y,y' \in V_{2,n}$, let $\varepsilon_{[y,y']}$ denote the unique automorphic character of $N_P(\A)$ satisfying $\varepsilon_{[y,y']}(\exp(b_1\wedge w+b_2\wedge w'))=\psi((y,w)+(y',w'))$ for all $w,w'\in V_{2,n}(\A)$. We have an $M_P$-equivariant identification \begin{equation}
\label{characters-of-NP}
U^{\vee}\otimes_{\Q} V_{2,n}\xrightarrow{\sim} \mathrm{Hom}(N_P(\Q)\backslash N_P(\A), \C^1), \qquad b_{-1}\otimes y+b_{-2}\otimes y'\mapsto \varepsilon_{[y,y']}
\end{equation}
Similarly, the abelian Lie algebra $\Lie(N_Q)$ is identified as an $M_Q$-module via the map 
$$
\Q b_1\otimes_{\Q}V_{3,n+1}\xrightarrow{\sim} \Lie(N_Q), \quad b_1\otimes w\mapsto b_1\wedge w.
$$
Let $\chi_y\colon N_Q(\A)\to \C^{\times}$ denote the automorphic character satisfying $\chi_y(\exp(b_1\wedge v))=\psi((y,v))$ for all $v\in V_{3,n+1}(\A)$. We have an $M_Q$-equivariant identification
\begin{equation}
\label{characters-of-NQ}
\Q b_{-1}\otimes_{\Q}V_{3,n+1}\xrightarrow{\sim}\mathrm{Hom}(N_Q(\Q)\backslash N_Q(\A), \C^1), \qquad b_{-1}\otimes y\mapsto \chi_y.
\end{equation}

When $\dim V=8$, we let $G'$ denote the spin group of $V$.  If $R$ is a $\Q$-algebra, one can extend the trilinear pairing $(\,,\,,\,)$ to $\O \otimes R$ by linearity and then
\[G'(R) = \{g=(g_1,g_2,g_3) \in \SO(\O,n_\O)(R)^3: (g_1 x_1, g_2 x_2, g_3 x_3) = (x_1, x_2,x_3) \forall x_j \in \O \otimes R\}.\]

The association $g \mapsto g_1$ defines a map $G' \rightarrow G$ when $\dim(V) = 8$.  We define $P' = M_{P'}N_{P'}$ to be the Heisenberg parabolic subgroup of $G'$ given by the inverse image of $P$ under the map $G' \rightarrow G$.  See Appendix~\ref{sec:triality} for more on $M_{P'}$ and $N_{P'}$.

\subsection{Compact subgroups}
\label{subsection-compact-subgroups}
We now define Cartan involutions and maximal compact subgroups of the groups with which we work.

In the case of $G = \SO(V)$, we proceed as follows.  Let $\iota_{2,n}: V_{2,n}(\R) \rightarrow V_{2,n}(\R)$ be the involution that is $+1$ on $V_2^{+}(\R)$ and $-1$ on $V_{n}^{-}(\R)$.  Let $\iota: V(\R) \rightarrow V(\R)$ be the involution with $\iota(b_j) = b_{-j}$ for $j \in \{1,2,-1,-2\}$ and $\iota$ restricted to $V_{2,n}(\R)$ is $\iota_{2,n}$.  Note that $\iota \in O(V)(\R)$.  If $V^+$ is the $\iota =1$ subspace, and $V^{-}$ is the $\iota = -1$ subspace, then $V(\R) = V^+ \oplus V^{-}$, and we define $K \subseteq G(\R)$ as the subgroup that preserves the decomposition $V( \R) = V^{+} \oplus V^{-}$.  Note that $K = S(O(V^{+}) \times O(V^{-}))$ and has identity component $K^0 = \SO(V^{+}) \times \SO(V^{-})$.  The corresponding Cartan involution $\Theta_\iota$ on $\mathrm{Lie}(G(\R)) \simeq \wedge^2 V(\R)$ is defined as $\Theta_\iota(u \wedge v) = \iota(u) \wedge \iota(v)$. 

To define quaternionic modular forms on $G$, we need a distinguished map 
\begin{equation}
\label{map-to-distinguished-su2}
K^0 \rightarrow \SU(2)/\mu_2.
\end{equation} To see that we have such a map, we make a Lie algebra argument.  Namely, recall that $V^+ = \Span\{u_1,u_2, v_1,v_2\}$, and following \cite[Ch. 8]{pollackAWSNotes}, set
\begin{itemize}
	\item $e^+ = \frac{1}{2}(u_1-iu_2) \wedge (v_1 - iv_2)$
	\item $h^+ = i (u_1 \wedge u_2 + v_1 \wedge v_2) = \frac{1}{2} (u_1-iu_2) \wedge (u_1+iu_2) + \frac{1}{2}(v_1-iv_2) \wedge (v_1+iv_2)$
	\item $f^+ = -\frac{1}{2} (u_1+iu_2) \wedge (v_1 + iv_2)$.
\end{itemize}
We define a distinguished subalgebra of $\g$ by $\sl_2^{\mathrm{dist}}=\C\linspan\{e^+,f^+,h^+\}$.  One obtains another $\SL(2)$ triple in $\g$ by replacing $v_2$ with $-v_2$ in the above formulas: that is, we let $\sl_2'$ be the subalgebra of $\g$ with basis
\begin{itemize}
	\item $e'^+ = \frac{1}{2}(u_1-iu_2) \wedge (v_1 + iv_2)$
	\item $h'^+ = i (u_1 \wedge u_2 - v_1 \wedge v_2)$
	\item $f'^+ = -\frac{1}{2} (u_1+iu_2) \wedge (v_1 - iv_2)$.
\end{itemize}
Then $\so(V^+)=\sl_2^{\mathrm{dist}}\oplus \sl_2'$.

The adjoint action of $K^0$ on the $\sl_2^{\mathrm{dist}}$ defines a three-dimensional representation $\mathbf{V}$ of $K^0$.  We choose a basis $x,y$ of $\C^2 = V_2$ so that $\mathrm{Sym}^2(V_2)$ is identified with $\sl_2^{\mathrm{dist}}$ via $e^+ = -x^2$, $h^+ = 2xy$, $f^+ = y^2$.  This gives the map $K^0 \rightarrow \SU(2)/\mu_2$ of \eqref{map-to-distinguished-su2}.  For an integer $\ell\geq 1$, write $\mathbf{V}_{\ell}$ for the highest weight quotient of the $\ell^{th}$ symmetric power of $\mathbf{V}$. So, $\mathbf{V}_{\ell}$ is the pullback of $\mathrm{Sym}^{2\ell}(V_2)$ along \eqref{map-to-distinguished-su2}, and $\mathbf{V}_{\ell}$ has a basis $x^{2\ell}, x^{2\ell-1}y, \ldots, xy^{2\ell-1}, y^{2\ell}$.

Let $\{\cdot , \cdot \}_{K^0}\colon \mathbf{V}_{\ell}\times \mathbf{V}_{\ell}\to \C$ denote the unique $K^0$-invariant bilinear form on $\mathbf{V}_{\ell}$ normalized so that 
\[
\{x^{\ell+v}y^{\ell-v}, x^{\ell-w}y^{\ell+w}\}_{K^0}=(-1)^{\ell+v}\delta_{v,w}(\ell+v)!(\ell-v)!
\]
where $v,w\in\{-\ell, \cdots, \ell\}$ and $\delta_{v,w}$ is the Kronecker delta symbol.

In the case when $n=2$, recall that $(g_1,g_2,g_3)\mapsto g_1$ gives a map $G'\to G$. We define a maximal compact subgroup  $K'$ of $G'(\R)$ as the preimage of $K^0$ in $G'(\R)$ under this map.  In this case, the Cartan involution on $\wedge^2 \mathbb{O}$ is given by $\Theta_\iota$, where $\iota(b_j) = b_{-j}$ and $\Theta_\iota(u\wedge v ) = \iota(u) \wedge \iota(v)$.

We have an action of $S_3$ on $G'$ by triality operators, which is transparent in the coordinates $g = (g_1, g_2, g_3)$ on $G'$.  For the convenience of the reader, in Appendix~\ref{sec:triality} we make explicit this action on $\wedge^2 \mathbb{O}$ via the isomorphisms
\[
\mathrm{Lie}(G') \simeq \mathrm{Lie}(G) \simeq \wedge^2 \mathbb{O},
\] 
the first map being induced via $(g_1, g_2, g_3) \mapsto g_1$.  In particular, in Theorem~\ref{thm:trialityQMF} we show that $K'$ is stable under the action of $S_3$.  By combining Theorem~\ref{thm:spin8O} with a statement in the last paragraph of \cite[Section 2.3]{6authorG2paper}, one sees that the $S_3$-action is trivial on $\sl_2^{\mathrm{dist}}$.

\section{Holomorphic and Quaternionic Modular Forms}\label{sec:QMFs}
In this section, we review facts about holomorphic modular forms on groups of type $\SO(2,n+1)$ and quaternionic modular forms on groups of type $\SO(4,n+2)$.  We also prove a few facts about the exact functions that appear in the Fourier expansion of quaternionic modular forms on special orthogonal groups. On account of certain complications resulting from the disconnectedness of the real group $\SO(4,n+2)$, these facts were not worked out in \cite{pollackQDS}.
\subsection{Holomorphic modular forms on $\mathrm{SO}(2,n+1)$}
\label{subsec:HMFS}
Recall that $n\geq 2$. Fix $V_{2,n+1}$ to be a $\Q$-vector space equipped with a non-degenerate quadratic form $q_{2,n+1}$ of signature $(2,n+1)$. Write $(x,y)_{2,n+1}=q_{2,n+1}(x+y)-q_{2,n+1}(x)-q_{2,n+1}(y)$ for the symmetric bilinear pairing on $V_{2,n+1}$ associated to $q_{2,n+1}$. In this subsection, we discuss holomorphic modular forms on $H := \SO(V_{2,n+1})$. We will later consider $V_{2,n+1}$ as a subspace of $V_{4,n+2}$, and our notation is chosen to be compatible with this embedding.

 According to the Hasse principle, there exists a nonzero isotropic vector $b_2\in V_{2,n+1}$. Since $q_{2,n+1}$ is non-degenerate, we may fix a second isotropic vector $b_{-2}\in V_{2,n+1}$ satisfying $(b_2,b_{-2})_{2,n+1}=1$. \\
\indent Let $V_{1,n}:=(\Q b_2+\Q b_{-2})^{\perp}$ be the orthogonal complement of $\Q b_2+\Q b_{-2}$ in $V_{2,n+1}$. Then
 \begin{equation}\label{orthogonalization-V(1,n)}
 V_{2,n+1}=\Q b_2\oplus V_{1,n}\oplus \Q b_{-2}.
 \end{equation}
 The Siegel parabolic subgroup $R\leq H$ is the subgroup of $H$ defined as the stabilizer of the line $\Q b_2\subseteq V_{2,n+1}$. Let $M_R\leq R$ denote the Levi subgroup of $R$ stabilizing the line spanned by $b_{-2}$. The unipotent radical $N_R\trianglelefteq R$ is identified with $V_{1,n}$ as an $\SO(V_{1,n})$ module via the map 
\begin{equation}
\label{identification-of-N_R}
\kappa\colon V_{1,n}\xrightarrow{\sim} N_R, \qquad v\mapsto \exp(b_2\wedge v).
 \end{equation}
\indent
Assume there exists $y_1\in V_{1,n}(\Q)$ such that $(y_1,y_1)_{2,n+1} = 2$.  The symmetric space for the group $H(\R)$ is
\[\mathfrak{h} = \{Z = X+ iY \in V_{1,n} \otimes \C: (Y,-y_1)_{2,n+1} > 0 \text{ and } (Y,Y)_{2,n+1} > 0\}.\]
We have an identification of $\mathfrak{h}$ with the set of isotropic elements in $V_{2,n+1} \otimes \C$ via the map
\[
\mathfrak{h}\to V_{2,n+1}\otimes_{\R}\C, \qquad Z \mapsto r(Z):= -q_{2,n+1}(Z) b_2 + Z + b_{-2}.\]
This identification yields an action of the identity component $H(\R)^0$ on $\mathfrak{h}$ as follows: If $g \in H(\R)^0$, then there exists a unique nonzero complex number $j_H(g,Z)$ and a unique element $gZ \in \mathfrak{h}$ so that $g r(Z) = j_H(g,Z) r(gZ)$.    Observe that $j_H(g,Z) = (g r(Z), b_2)_{2,n+1}$. 

We can now define classical holomorphic modular forms on $\mathfrak{h}$.
\begin{definition} Suppose $\ell \in \Z$, and $\Gamma \subseteq H(\R)$ is a congruence subgroup, i.e., $\Gamma = H(\Q) \cap U$ for an open compact subgroup $U$ of $H(\A_f)$.  Then $f: \mathfrak{h} \rightarrow \C$ is a \emph{holomorphic modular form} of weight $\ell$ and level $\Gamma$ if
	\begin{enumerate}
		\item $f$ is holomorphic;
		\item $f(\gamma Z) = j_H(\gamma, Z)^{\ell} f(Z)$ for all $Z \in \mathfrak{h}$ and $\gamma \in \Gamma \cap H(\R)^{0}$;
		\item the function $\varphi_f(g):=j_H(g,-iy_1)^{-\ell} f(g \cdot (-iy_1)): H(\R)^0 \rightarrow \C$ is of moderate growth.
	\end{enumerate}
\end{definition}

If $T \in V_{1,n}$, say that $T$ is \textit{positive definite} if $(T,-y_1)_{2,n+1} > 0$ and $(T,T)_{2,n+1} > 0$. We say that $T$ is \textit{positive semi-definite} if $(T,-y_1)_{2,n+1} \geq 0$ and $(T,T)_{2,n+1} \geq 0$.  Write $T > 0$ (resp. $T \geq 0$) if $T$ is positive definite (resp. semi-definite).  

The bilinear pairing $(\, , \,)_{2,n+1}$ restricts to give a bilinear form on $V_{1,n}$, and we use \eqref{identification-of-N_R} to define a lattice $\Lambda\leq V_{1,n}$ as the dual of  $\kappa^{-1}(\Gamma \cap N_R)$ relative to $(\, , \,)_{2,n+1}$. If $f\colon \mathfrak{h}\to \C$ is a holomorphic modular form of level $\Gamma$ then $f(n\cdot Z)=f(Z)$ for $n\in N_R\cap \Gamma$. Therefore $f$ admits a Fourier expansion 
\[
f(Z) = \sum_{T \in \Lambda \colon T\geq 0} a_f(T) e^{2\pi i (T,Z)_{2,n+1}}.
\]
The condition $T\geq 0$ arises as a consequence of the moderate growth assumption. If $f$ is cuspidal, so that $|\varphi_f(g)|$ is bounded, then $a_f(T) \neq 0$ implies that $T>0$.  For more on holomorphic modular forms on $\mathfrak{h}$ one may consult \cite[Chapters 3 and 4]{MR1903920}.

We define $K_{H} = \mathrm{Stab}(\mathrm{Span}\{b_2+b_{-2},y_1\}) \subseteq H(\R)$ and $K_H^0 = K_H \cap H(\R)^0$. Then $K_H$ is a maximal compact subgroup of $H(\R)$, and one can verify that $K_H^0$ is the stabilizer of $-iy_1$ in $H(\R)^0$.  Observe that $j_H(\cdot, -iy_1): K_H^0 \rightarrow \C^\times$ is a character.

Suppose that $\varphi: \Gamma \backslash H(\R)^0 \rightarrow \C$ is a function satisfying $\varphi(hk) = j_H(k,-iy_1)^{\ell}\varphi(h)$ for all $h \in H(\R)^0$ and $k \in K_H^0$.  Then the function $f_\varphi: \mathfrak{h} \rightarrow \C$,
\[f_\varphi(Z) = j_H(g,-iy_1)^{\ell} \varphi(g), \text{ if } g \cdot (-iy_1) = Z\]
is well-defined.  If $f_\varphi(Z)$ is holomorphic, and $\varphi$ is of moderate growth, we say that $\varphi$ is the automorphic function associated to a holomorphic modular form of weight $\ell$.

\subsection{Two Definitions of Quaternionic Modular Forms}
\label{subsec:Definition-QMFS}
We present two definitions of quaternionic modular forms on $G$. The first uses differential operators, while the second uses representation theory. We begin with a general definition of automorphic forms on $G$, using the notation given in Section~\ref{sec:notation}. 
\begin{definition}
\label{defn-automorphic-forms-on-$G$}
We define $\mathcal{A}(G)$, the space of automorphic forms on $G$, to consist of the smooth functions $\varphi$ on $G(\A)$ such that the following conditions are satisfied: 
\begin{compactenum}
    \item $\varphi$ is left $G(\Q)$ invariant; 
    \item $\varphi$ is right invariant under some open compact subgroup $U$ of $G(\A_f)$; 
    \item $\varphi$ is annihilated by an ideal of finite codimension in $Z(\g)$, the center of the universal enveloping algebra of $\g$;
    \item for every $g_f \in G(\A_f)$, the function $g \mapsto \varphi(g_f g)$ is of moderate growth on $G(\R)$;
    \item $\varphi$ is $K^0$-finite.
\end{compactenum}
We write $\mathcal{A}_0(G)$ to denote the subspace of $\mathcal{A}(G)$ consisting of cusp forms. Then both $\mathcal{A}(G)$ and $\mathcal{A}_0(G)$ are naturally $(\g,K^0)\times G(\A_f)$-modules (see for example \cite[Definition 6.5]{MR4738301}). 
\end{definition}

To define quaternionic modular forms, we recall (see, for example, \cite[Introduction]{pollackQDS}) a certain differential operator on $\mathbf{V}_{\ell}$-valued functions on the disconnected Lie group $G(\R)$.

\begin{definition}\label{defn:DiffOpDl} Let $G' = G$ or $G' = G_1$.  Suppose $\ell \geq 1$ is an integer and $\phi: G'(\R) \rightarrow \mathbf{V}_{\ell}$ is a smooth function on $G'(\R)$, satisfying $\phi(gk) = k^{-1} \cdot \phi(g)$ for all $g \in G'(\R)$ and $k \in K^0$.  Let $\g = \k \oplus \p$ be the corresponding Cartan decomposition on the Lie algebra of $G'(\R)$, so that $\p$ is the orthogonal complement to $\k$ via the Killing form.  The $K^0$-representation $\mathbf{V}_{\ell} \otimes \p^\vee$ decomposes into two irreducible pieces, 
\[
\mathbf{V}_{\ell} \otimes \p^\vee = V'_{-} \oplus V'_{+},\]
with $\dim(V'_{-}) = \ell \cdot \dim(\p)$ and $\dim(V'_{+}) = (\ell+1) \cdot \dim(\p)$.  Let $\mathrm{pr}'_{-}: \mathbf{V}_{\ell} \otimes \p^\vee \rightarrow V'_{-}$ denote the $K^0$-equivariant projection onto the smaller piece.  One sets
\[
D_{\ell}\phi = \sum_{\alpha} \mathrm{pr}'_{-}(X_\alpha \phi \otimes X_{\alpha}^\vee).
\]
Here $\{X_\alpha\}_\alpha$ is a basis of $\p$, $\{X_\alpha^\vee\}_\alpha$ denotes the dual basis of $\p^\vee$, and $X_\alpha \phi$ denotes the differential of the right-regular representation.  The function $D_\ell \phi$ is independent of the choice of basis $X_\alpha$.   One says that $\phi$ is \textbf{quaternionic} if $D_\ell \phi$ is identically $0$.
\end{definition}

Note that in Definition \ref{defn:DiffOpDl}, $g \in G'(\R)$ need not live in the identity component.

\begin{definition}\label{defn:QMFs1} For an integer $\ell \geq 1$, say that $\varphi\in \mathcal{A}(G)\otimes_{\C}\mathbf{V}_{\ell}$ is a quaternionic modular form of weight $\ell$ if the following two properties are satisfied:
\begin{enumerate}
\item $\varphi(gk) = k^{-1} \varphi(g)$ for all $g \in G(\A)$ and $k \in K^0$. Here, since $\varphi(g)$ is a vector in the representation $\mathbf{V}_{\ell}$ of Section~\ref{subsection-compact-subgroups}, we  write $k^{-1}\varphi(g)$ to denote the image of $\varphi(g)$ under the action of the element $k^{-1}$.
\item For every $g_f \in G(\A_f)$, the function $\varphi_{g_f}: G(\R) \rightarrow \mathbf{V}_{\ell}$ defined as $\varphi_{g_f}(g) = \varphi(g_fg)$ is quaternionic in the sense of Definition \ref{defn:DiffOpDl}.
\end{enumerate}
Let $\mathcal{A}(G,\ell)$ denote the space of quaternionic modular forms on $G$ of weight $\ell$, and write $\mathcal{A}_{0}(G,\ell)$ for the subspace of $\mathcal{A}(G,\ell)$ consisting of cusp forms. When $G$ is split $\SO(8)$, we write $\mathcal{A}_{0,\Z}(G,\ell)$ for the space of cuspidal quaternionic modular forms that are right $G(\widehat{\Z})$ invariant. Here $G(\ZZ)$ denotes the stabilizer of $V(\ZZ)$ in $G(\A_f)$ (see \eqref{eqn-defn-V(Z)}).  If $\varphi \in \mathcal{A}_{0,\Z}(G,\ell)$, we say that $\varphi$ is of level one.
\end{definition}

For an integer $\ell \geq 1$, let $\pi_{\ell}^0$ be the irreducible representation of $G(\R)^0$ denoted by the symbol $\pi'_{2\ell+2}$ in \cite[\S 8]{grossWallachII}. Then $\pi_{\ell}^0$ contains the subspace $\mathbf{V}_{\ell}$ as its minimal $K^0$-type with multiplicity one. For concreteness, we remark that in the case when $G$ is split $\SO(8)$, $\pi_{\ell}^0$ is a discrete series representation for $\ell\geq 4$ (see loc. cit.).  Let $\beta$ denote the highest weight of the representation $\wedge^2 V$ of $\g$ and for an integer $r \geq 1$ let $V_{r\beta}$ denote the irreducible representation of $\g$ with highest weight $r \beta$, so that $V_{r\beta}$ is contained in $\mathrm{Sym}^{r}(\wedge^2 V)$ with multiplicity one. We remark that $\pi_{\ell}^0$ has the same infinitesimal character as that of $V_{(\ell-4)\beta}$.

Write $\pi_{\ell,\mathrm{fin}}^0$ for the subspace of $K^0$-finite vectors in $\pi_{\ell}^0$ endowed with its natural structure as a $(\g,K^0)$-module (see for example \cite[Proposition 4.4.3]{MR4738301}). 
\begin{definition}\label{defn:QMF2} 
Define $\mathcal{A}^{rep}(G,\ell)$ as the space of $(\g,K^0)$-module homomorphisms from $\pi_{\ell,\mathrm{fin}}^0$ to $\mathcal{A}(G)$. Similarly, define $\mathcal{A}_{0}^{rep}(G,\ell)$ as the space of $(\g,K^0)$-module homomorphisms from $\pi_{\ell,\mathrm{fin}}^0$ to $\mathcal{A}_0(G)$. In the case when $G$ is split $\SO(8)$, we define 
\[
\mathcal{A}^{rep}_{\Z}(G,\ell)=\mathrm{Hom}_{(\g,K^0)}(\pi_{\ell,\mathrm{fin}}^0, \mathcal{A}(G)^{G(\ZZ)})
\]
where here $\mathcal{A}(G)^{G(\ZZ)}$ denotes the space of automorphic forms on $G$ that are $G(\ZZ)$-invariant.  Similarly we set
\[
\mathcal{A}_{0,\Z}^{rep}(G,\ell)=\mathrm{Hom}_{(\g,K^0)}(\pi_{\ell,\mathrm{fin}}^0, \mathcal{A}_0(G)^{G(\ZZ)}).\]
\end{definition}
\begin{remark}
As far as the authors are aware, a definition of $\mathcal{A}^{\mathrm{rep}}_{\mathbf Z}(G,\ell)$ of the form given in Definition~\ref{defn:QMF2} has not previously appeared in the literature. The primary reason is that $G(\mathbf R)$ is disconnected, whereas many of the foundational references on quaternionic modular forms (for example, \cite{ganGrossSavin, grossWallachII, pollackQDS}) treat groups whose real points are connected. We refer the reader to Definition~\ref{defn:QDSOrthog1} for a more detailed discussion of the issues arising in the disconnected setting.
\end{remark}
There is a canonical injection $\mathcal{A}^{rep}(G,\ell) \rightarrow \mathcal{A}(G,\ell)$ given as follows. Suppose $\alpha: \pi_{\ell}^0 \rightarrow \mathcal{A}(G)$ is in $\mathcal{A}^{rep}(G,\ell)$.  Let $\{v_j\}_j$ be a basis of $\mathbf{V}_{\ell}$ and $\{v_j^\vee\}$ the dual basis of $\mathbf{V}_{\ell}^\vee \simeq \mathbf{V}_{\ell}$.  Then $\sum_{j} \alpha(v_j) \otimes v_j^\vee$ is an element of $\mathcal{A}(G,\ell)$. \\
\indent In the case when $n=2$, recall that $G'$ denotes $\Spin(V)$. In Appendix~\ref{sec:triality}, we consider spaces of quaternionic modular forms on $G'$. We define $\mathcal{A}(G')$ and $\mathcal{A}_0(G')$ exactly as in Definition~\ref{defn-automorphic-forms-on-$G$}, except that $G$ is replaced with $G'$. Likewise the spaces of quaternionic modular forms on $G'$, $\mathcal{A}(G',\ell)$ and $\mathcal{A}_{0}(G',\ell)$, are defined as in Definition~\ref{defn:QMFs1}, except $G$ is replaced with $G'$, and $K^0$ is replaced with $K'$. Pullback along the map $G'\to G$ gives a map
\[ 
\mathcal{A}(G,\ell)\to \mathcal{A}(G',\ell),
\]
sending $\mathcal{A}_0(G,\ell)$ to $\mathcal{A}_0(G',\ell)$. 
\subsection{The Explicit Fourier Expansion of Quaternionic Modular Forms}
Recalling the notation from Section~\ref{subsec:groupnotation}, let $dz$ denote the standard right $Z(\A)$-invariant measure on $Z(\Q)\backslash Z(\A)\simeq \Q\backslash \A$. Suppose $\varphi$ is a quaternionic modular form on $G(\A)$. The \emph{constant term} of $\varphi$ along $Z$ is defined as the function on $G(\A)$ given by
\begin{equation}
\label{defn-constant-term}
\varphi_Z(g)=\int_{Z(\Q)\backslash Z(\A)}\varphi(zg)dz. 
\end{equation}
More generally, if $\mathcal{G}$ is any algebraic group over $\Q$, $h\colon \mathcal{G}(\Q)\backslash \mathcal{G}(\A)\to \C$ a continuous function, and $\mathcal{N}$ a unipotent subgroup of $\mathcal{G}$ equipped with a character $\chi\colon \mathcal{N}(\Q)\backslash \mathcal{N}(\A)\to \C^{\times}$, we define the $\chi$-th \emph{Fourier coefficient} of $h$ along $\mathcal{N}$ as 
\begin{equation}
\label{defn-general-FC}
    h_{\mathcal{N}, \chi}(g)=\int_{\mathcal{N}(\Q)\backslash \mathcal{N}(\A)}\chi^{-1}(n)h(ng)dn. 
\end{equation}
In this notation, the Fourier expansion of $\varphi_Z$ along $Z(\A)N_P(\Q)\backslash N_P(\A)$ takes the form 
\begin{equation}
\label{FE-general-phiZ}
\varphi_Z(g)=\sum_{[T_1,T_2]\in V_{2,n}\times V_{2,n}}\varphi_{N_P, \varepsilon_{[T_1,T_2]}}(g). 
\end{equation}
The main result of \cite{pollackQDS} fully explicates \eqref{FE-general-phiZ} in the case when $\varphi$ is a quaternionic modular form. We use a version of this result from \cite{pollackAWSNotes}. To state the theorem, we write $B$ to denote the natural $\GL(U)\times \SO(V_{2,n})$-equivariant pairing between $U^{\vee}\otimes_{\Q} V_{2,n}$ and $U\otimes_{\Q} V_{2,n}$. So if $T_1, T_1',T_2,T_2'\in V_{2,n}$ then
\begin{equation}
\label{Bilinear-form-B}
B(b_{-1}\otimes T_1+b_{-2}\otimes T_2, b_{1}\otimes T_1'+b_{2}\otimes T_2')=(T_1,T_1')+(T_2,T_2'). \quad
\end{equation}
Since $M_P$ preserves the decomposition $V=U\oplus V_{2,n}\oplus U^{\vee}$, the action of $M_P(\R)$ on $U^{\vee}\otimes V_{2,n}$ and $U\otimes V_{2,n}$ preserves the pairing $B$. 
 Recall the elements $v_1=y_0/\sqrt{2}$ and $v_2=y_1/\sqrt{2}$ defined in Section~\ref{The underlying quadratic space}. 
\begin{definition}
Let $[T_1,T_2]\in V_{2,n}\times V_{2,n}$  and define $\beta_{[T_1,T_2]}\colon M_P(\R)\to \C$ by 
\begin{equation}
\label{definition of beta}
\beta_{[T_1,T_2]}(r):=\sqrt{2} iB(r^{-1}\cdot (b_{-1}\otimes T_1+b_{-2}\otimes T_2), b_{1}\otimes(v_1+iv_2)+b_{2}\otimes i(v_1+ iv_2)).
\end{equation}
The ordered pair $[T_1, T_2]$ is said to be \textit{positive semi-definite} if $\beta_{[T_1,T_2]}(r)\neq 0$ for all $r\in M_P(\R)^0$. We write $[T_1,T_2]\succeq 0$ to mean that the pair $[T_1,T_2]$ is positive semi-definite. We write $[T_1,T_2]\succ 0$ if $[T_1,T_2]\succeq 0$ and $(T_1,T_1)(T_2,T_2)-(T_1,T_2)^2>0$. 
\end{definition}

Suppose $r \in M_P(\R) \simeq \GL(U)(\R) \times \SO(V_{2,n})(\R)$.  We express $r$ as an ordered pair $(m,h)$ with $m\in \GL(U)(\R)$ and $h\in \SO(V_{2,n})(\R)$. 
\begin{definition}\label{def:WhitOrth} \cite[section 8.2]{pollackAWSNotes} Fix an integer $\ell \geq 1$, and suppose $[T_1, T_2] \in V_{2,n}(\R)^{\oplus 2}$.  We define a function 
	\[
	\mathcal{W}_{[T_1,T_2]}: G(\R) \rightarrow \mathbf{V}_{\ell}
	\]
as follows.  First suppose $[T_1, T_2]$ is positive semi-definite. If $r = (m,h) \in M_P(\R)^{0}$, then 
\begin{equation}
	\label{K-Bessel-Magic}
	\mathcal{W}_{[T_1,T_2]}(r)=\det(m)^{\ell}|\det(m)|\sum_{-\ell\leq v\leq \ell}\left(\frac{\beta_{[T_1,T_2]}(r)}{|\beta_{[T_1,T_2]}(r)|}\right)^vK_v(|\beta_{[T_1, T_2]}(r)|)\frac{x^{\ell+v}y^{\ell-v}}{(\ell+v)!(\ell-v)!}.
\end{equation}
Here $K_{v}\colon \R_{>0}\to \R$ denotes the modified K-Bessel function $K_v(x)=\frac{1}{2}\displaystyle{\int_0^{\infty}}t^{v-1}e^{-x(t+t^{-1})}dt$.  If $r$ is in the non-identity component of the group $M_P(\R) \cap G(\R)^{0}$, then $\beta_{[T_1, T_2]}(r)$ is still nonzero and one defines $\mathcal{W}_{[T_1,T_2]}(r)$ again as in \eqref{K-Bessel-Magic}. If $g\in G(\R)^0$, write $g = nmk$ with $n \in N_P(\R)$, $m \in M_P(\R) \cap G(\R)^0$ and $k \in K^0$.   Suppose  $\log(n) = b_1 \wedge w_1 + b_2 \wedge w_2 + z b_1 \wedge b_2$ for some $w_1, w_2 \in V_{2,n}(\R)$ and $z \in \R$.  Then
\[
\mathcal{W}_{[T_1,T_2]}(g) := (e^{i(T_1,w_1) + i (T_2,w_2)}) k^{-1} \cdot \mathcal{W}_{[T_1,T_2]}(m).
\]
This is independent of the choice of $m$ and $k$.  If $[T_1,T_2]$ is not positive semi-definite, one defines $\mathcal{W}_{[T_1,T_2]}$ to be $0$ on $G(\R)^0$.  

Finally, let $\iota  \in M_P(\R)$ be $\iota = \left(\mm{0}{1}{1}{0}, 1\right).$ If $g \in G(\R)$ is not in the identity component, then $\iota g \in G(\R)^0$ and one defines
\[
\mathcal{W}_{[T_1,T_2]}(g) = (-1)^{\ell} \mathcal{W}_{[T_2, T_1]}(\iota g).
\]
\end{definition}

\begin{thm}\label{Thm 1.2.1 Aaron Paper}\emph{\cite[Theorem 8.2.2]{pollackAWSNotes}}
Fix $\ell\in \Z_{\geq 1}$ and suppose $[T_1,T_2]\in V_{2,n}(\R)^{\oplus2}$.  The function $\mathcal{W}_{[T_1,T_2]}$ is smooth, of moderate growth, and satisfies the differential equation $D_{\ell} \mathcal{W}_{[T_1,T_2]} = 0$.  Moreover, suppose $F: G(\R)^0 \rightarrow \mathbf{V}_{\ell}$ is smooth, of moderate growth, and satisfies the following properties:
\begin{enumerate}
\item If $g\in G(\R)^0$ and $k\in K^0$ then $F(gk)=k^{-1}F(g)$.
\item If $g\in G(\R)^0$ and $n \in N_P(\R)$ satisfies $\log(n) = b_1 \wedge w_1 + b_2 \wedge w_2 + z b_1 \wedge b_2$, then
\[F(ng)=e^{i(T_1,w_1) + i (T_2,w_2)}F(g).\]
\item If $g\in G(\R)^0$ then $D_{\ell}F(g)=0$.  
\end{enumerate}
Then there is $\lambda \in \C$ so that $F = \lambda \mathcal{W}_{[T_1,T_2]}$ on $G(\R)^0$.
\end{thm}
\begin{remark}
\label{rmk-FE-QMF-SPIN8}
\begin{asparaenum}[(1)]
    \item If $\ell\geq n+2$ and $(T_1,T_1)(T_2,T_2)-(T_1,T_2)^2\neq 0$ then, with the exception of the explicit formula (\ref{K-Bessel-Magic}), Theorem \ref{Thm 1.2.1 Aaron Paper} follows from work of Wallach \cite[Theorem 16]{wallach}.
    \item If $n=2$ and $G'=\Spin(V)$, then the multiplicity at most one statement of Theorem~\ref{Thm 1.2.1 Aaron Paper} remains valid when $G$ is replaced with $G'$, and $N_P$ is replaced with $N'$. Indeed, $\mathbf{V}_{\ell}$ appears in $\g$, and so the central $\mu_2$ in the kernel of the map $G'\to G$ acts trivially on $\mathbf{V}_{\ell}$. Hence, any function $W\colon G'(\R)\to \mathbf{V}_{\ell}$ such that $W(gk)=k^{-1}W(g)$ for all $g\in G'(\R)$ and $k\in K'$, factors across $G(\R)$. 
\end{asparaenum}

\end{remark}
As a corollary to Theorem \ref{Thm 1.2.1 Aaron Paper} we deduce the following.
\begin{corollary}
\label{cor:FE of phiZ}
Suppose $\ell\in \Z_{\geq 1}$ and let $\varphi\colon G(\A)\to \mathbf{V}_{\ell}$ be a weight $\ell$ quaternionic modular form on $G(\A)$. Then there exists a family of locally constant functions  
$$
\{a_{[T_1,T_2]}(\varphi,\cdot)\colon G(\A_f)\to \C\colon \hbox{$[T_1,T_2]\in V_{2,n}\times V_{2,n}$ such that $[T_1,T_2]\succeq 0$}\}
$$ such that the Fourier expansion of $\varphi_Z$ along $Z(\A)N_P(\Q)\backslash N_P(\A)$ takes the form
\begin{equation}
\label{FE-QMF} 
\varphi_{Z}(g_fg_{\infty})=\varphi_{N_P}(g_fg_{\infty})+\sum_{[T_1,T_2]\in V_{2,n}\times V_{2,n} \colon [T_1,T_2]\succeq 0}a_{[T_1,T_2]}(\varphi,g_f)\mathcal{W}_{[2\pi T_1,2\pi T_2]}(g_{\infty}).  
\end{equation}
\end{corollary}
\begin{remark}
\label{FE-QMFS-on-Spin82} The scaling by $2\pi$ is a consequence of the definition of the character $\varepsilon_{[T_1,T_2]}$ and property (2) of Theorem \ref{Thm 1.2.1 Aaron Paper}. If $n=2$ then the statement of Corollary~\ref{cor:FE of phiZ} applies equally well to quaternionic modular forms on $G'=\Spin(V)$. Indeed, the map $G' \rightarrow G$ induces an isomorphism $N'\simeq N$, and as such, if $\varphi\in\mathcal{A}(G',\ell)$ and $Z'$ denotes the center of $N'$, then 
\begin{equation}
\label{FE-QMF-spin8} 
\varphi_{Z'}(g_fg_{\infty})=\varphi_{N'}(g_fg_{\infty})+\sum_{[T_1,T_2]\in V_{2,n}\times V_{2,n} \colon [T_1,T_2]\succeq 0}a_{[T_1,T_2]}(\varphi,g_f)\mathcal{W}'_{[2\pi T_1,2\pi T_2]}(g_{\infty}).  
\end{equation}
Here $\mathcal{W}'_{[T_1,T_2]}$ is the pullback of $\mathcal{W}_{[T_1,T_2]}$ along $G'(\R)\to G(\R)$ (see Remark~\ref{rmk-FE-QMF-SPIN8}(2)). 
\end{remark}
\begin{definition}
\label{Definition-of-FCs}
    With hypotheses as in Corollary~\ref{cor:FE of phiZ}, let $\lambda=[T_1,T_2]$ denote an element in $V_{2,n}\times V_{2,n}$ and suppose $[T_1,T_2]\succeq 0$. Define the \textit{$\lambda$-th Fourier coefficient of $\varphi$} to be the complex number $a_{\varphi}(\lambda):=a_{[T_1,T_2]}(\varphi,1)$. Similarly, if $n=2$ and $\varphi\in \mathcal{A}(G',\ell)$, then with notation as in \eqref{FE-QMF-spin8}, we define $a_{\varphi}(\lambda):=a_{[T_1,T_2]}(\varphi,1)$. 
\end{definition}
\subsection{The Vanishing of $\beta_{[T_1,T_2]}$}
In this subsection we develop some properties of the functions $\beta_{[T_1, T_2]}$ defined by (\ref{definition of beta}). We write $\SO(V_{2,n})(\R)^{0}$ to denote the identity component of $\SO(V_{2,n})(\R)$. The reader may verify the following lemma. 
\begin{lemma}
    \label{Identity-component-SO(V2n)}
Suppose $h\in \SO(V_{2,n})(\R)$ and let $\{w_1,w_2\}$ denote a set of vectors in $V_{2,n}\otimes_{\Q}\R$ satisfying $(w_i,w_j)=\delta_{ij}$ for $i,j=1,2$. Then $h\in \SO(V_{2,n})(\R)^0$ if and only if 
\begin{equation}
    \label{defn-chi}
    \chi_{w_1,w_2}(h):=\det\begin{pmatrix} (w_1,hw_1) &(w_1,hw_2) \\ (w_2,hw_1) &(w_2, hw_2) \end{pmatrix}>0.
\end{equation}
\end{lemma}
The proposition below is an analogue of \cite[Proposition 10.0.1]{pollackQDS}. Write $\GL(U)(\R)^0$ to denote the identity component of $\GL(U)(\R)$; then $M_P(\R)^0=\GL(U)(\R)^0\times \SO(V_{2,n})(\R)^0$.
\begin{proposition}
\label{Properties-of-beta}
Suppose $[T_1, T_2]\in V_{2,n}\times V_{2,n}$.
\begin{compactenum}[(1)] 
\item If $\R\linspan\{T_1,T_2\}$ is an indefinite or negative definite two-plane in $V_{2,n}(\R)$ then there exists $r\in M_P(\R)^0$ such that $\beta_{[T_1,T_2]}(r)=0$. 
\item If $\R\linspan\{T_1, T_2\}$ is a positive definite two-plane in $V_{2,n}(\R)$ then exactly one of $\beta_{[T_1,T_2]}$ and $\beta_{[T_2,T_1]}$ has a zero on $M_P(\R)^{0}$. 
\item If $|\beta_{[T_1,T_2]}(r)|$ is bounded away from zero on the set $\{(m,h)\in M_P(\R)^0\colon \det(m)=1\}$ then $(T_1,T_1)(T_2,T_2)-(T_1,T_2)^2>0$. In particular $T_1$ and $T_2$ span a two-plane in $V_{2,n}(\R)$. 
\end{compactenum}
\end{proposition}
\begin{proof}
In the proof below we adopt the temporary notation $W=\R\linspan\{T_1,T_2\}$. Recall the fixed positive definite two-plane $V^+_2(\R)=\R\linspan\{v_1, v_2\}$ defined in Section~\ref{The underlying quadratic space}. 
\begin{asparaenum}[(1)]
    \item If $W$ is a negative definite two-plane then there exists a positive definite two-plane $P^+$ such that $P^+$ is orthogonal to $W$. Hence, there exists $h\in \mathrm{O}(V_{2,n})(\R)$ such that $hv_1$ and $hv_2$ give a basis for $P^+$. If necessary, we pre-compose $h$ with any isometry of $V_{2,n}(\R)$ which acts as the identity on $V_{2}^+(\R)$ and acts by an orthogonal transformation of determinant $-1$ on $V_2^+(\R)^{\perp}$. Post-composing with $k=(k_1, k_2)$ if necessary, where $k_1 \in O(W)$, $k_2 \in O(W^\perp)$, we may assume $h\in \SO(V_{2,n})(\R)^0$.  Since $P^+$ and $W$ are orthogonal, the element $r=(1, h)\in M_P(\R)$ satisfies $\beta_{[T_1, T_2]}(r)=0$. \\
    \indent Next suppose $W$ is a two-plane of indefinite signature $(1,1)$. Fix a basis $\{e,f\}$ for $W$ such that $(e,e)=(f,f)=0$ and $(e,f)=1$. Then there exists $m\in \mathrm{GL}(U)(\R)^0$ such that $$m^{-1}(b_{-1}\otimes T_1+b_{-2}\otimes T_2)=b_{-1}\otimes e+b_{-2}\otimes f.$$
    Let $u_+\in W^{\perp}$ be such that $(u_+,u_+)=2$. Write $P^+$ to denote the positive definite two-plane spanned by the vectors $p_1=\frac{1}{\sqrt{2}}(e+f)$ and $p_2= u_++\frac{1}{\sqrt{2}}(e-f)$.
    Reasoning as above, there exists $h\in \SO(V_{2,n})(\R)$ such that $hv_1=p_1$ and $hv_2=p_2$. It follows that if $r=(m,h)$ then $\beta_{[T_1,T_2]}(r)=0$. It remains to show that we can choose $r\in M_P(\R)^0$. Thus suppose $r\notin M_P(\R)^0$, which implies $h\notin \SO(V_{2,n})(\R)^0$. Define $p_1'=\frac{1}{\sqrt{2}}(e+f)$, $p_2'= -u_++\frac{1}{\sqrt{2}}(e-f)$ and let $h'\in \SO(V_{2,n})(\R)$ be such that $h'p_1=p_1'$ and $h'p_2=p_2'$. Then $\chi_{p_1,p_2}(h')=-1$ and Lemma \ref{Identity-component-SO(V2n)} implies $r':=(m,h'h)\in M_P(\R)^0$. Moreover, $\beta_{[T_1,T_2]}(r')=0$ as required.
    \item Assume $W$ is a positive definite two-plane in $V_{2,n}(\R)$. Then there exists $r\in M_P(\R)$ such that
    $
    r^{-1}(b_{-1}\otimes T_1+b_{-2}\otimes T_2)=b_{-1}\otimes v_1+b_{-2}\otimes v_2,
    $
and we may assume $r=(m,h)$ with $m\in \GL(U)(\R)^0$. Define $T_1',T_2'\in V_{2,n}(\R)$ by $$m^{-1}(b_{-1}\otimes T_1+b_{-2}\otimes T_2)=b_{-1}\otimes T_1'+b_{-2}\otimes T_2'.$$ If $\chi_{T_1',T_2'}(h)>0$ then $r\in M_P(\R)^0$ and $\beta_{[T_1,T_2]}(r)=\sqrt{2} i (hv_1+ihv_2, hv_1+ihv_2)=0$. Otherwise, $\chi_{T_1',T_2'}(h)<0$ and Lemma \ref{Identity-component-SO(V2n)} implies $\chi_{T_1',T_2'}(h')$ is negative for all $h'\in M_P(\R)^0$. Thus, if $r=(m,h')\in M_P(\R)^0$ then 
    $$
    |\beta_{[T_1,T_2]}(r)|^2=(T_1',h'v_1)^2+(T_2',h'v_2)^2+(T_1',h'v_2)^2+(T_2',h'v_1)^2-2\chi_{T_1',T_2'}(h')>0.
    $$
Hence $\beta_{[T_1,T_2]}$ vanishes on exactly one of the two components of $\GL(U)(\R)^0\times \SO(V_{2,n})(\R)$. 
    \item Assume the quantity $|\beta_{[T_1,T_2]}(r)|$ is bounded away from zero on the subset of pairs $(m,h) \in M_P(\R)^0$ with $\det(m) = 1$. Then by part (1), either $\dim(W)<2$, or $W$ is a positive definite $2$-plane. If $W$ is a positive definite two-plane then the Gram matrix of $W$ has positive determinant, and therefore
    $(T_1,T_1)(T_2,T_2)-(T_1,T_2)(T_2,T_1)>0$. The case $\dim (W)<2$ does not occur when $\beta_{[T_1,T_2]}$ is bounded away from zero. Indeed if $T_1$ and $T_2$ are collinear and $\alpha\in \R_{>0}$, then there exists $m_{\alpha}\in \SL(U)(\R)$ such that $m_{\alpha}$ acts by multiplication by $\alpha$ on $T_1$ and $T_2$. It follows that $\beta_{[T_1,T_2]}(m_{\alpha},1)\to 0$ as $\alpha\to 0^+$.
\end{asparaenum}
\end{proof}

\begin{corollary}
\label{FE-CUSPIDAL-QMF} 
    Suppose $\ell\in \Z_{\geq 1}$ and $\varphi\colon  G(\A)\to \mathbf{V}_{\ell}$ is a weight $\ell$ quaternionic modular form. If $\varphi$ is a cusp form then the Fourier expansion (\ref{FE-QMF}) takes the form 
  \begin{equation}
  \label{equation-FE-CUSPIDAL-QMF}
\varphi_{Z}(g_fg_{\infty})=\sum_{[T_1,T_2]\in V_{2,n}\times V_{2,n} \colon [T_1,T_2]\succ 0}a_{[T_1,T_2]}(\varphi,g_f)\mathcal{W}_{[2\pi T_1,2\pi T_2]}(g_{\infty}).  
\end{equation}
\end{corollary}
\begin{proof} Suppose $\varphi$ is cuspidal and assume $[T_1,T_2]$ does not satisfy $[T_1,T_2]\succ 0$. Fix $g_f\in G(\A_f)$. On the one hand, the form $\varphi$ is bounded because it is cuspidal. Therefore, the function 
\[
	g_{\infty}\mapsto \int_{N_P(\Q)\backslash N_P(\A)}\varepsilon_{[T_1,T_2]}^{-1}(n)\varphi(ng_fg_{\infty})\,dn
\]
	is bounded. As such, either $a_{[T_1,T_2]}(\varphi,g_f)\equiv 0$, or $\mathcal{W}_{[2\pi T_1,2\pi T_2]}(g_{\infty})$ is bounded. On the other hand, the function $K_{v}(x)$ diverges as $x$ tends to $0$ from the right. Applying Proposition \ref{Properties-of-beta}(3), it follows that if $\{T_1,T_2\}$ spans a one-plane in $V_{2,n}(\R)$, then $\mathcal{W}_{[2\pi T_1,2\pi T_2]}(m_{\infty})$ is unbounded on $M_P(\R)^0$. We conclude that $a_{[T_1,T_2]}(\varphi,g_f)\equiv 0$ as required.
\end{proof}

\section{The quaternionic Saito-Kurokawa lifting and quaternionic Maass Spezialschar}\label{sec:QSKL}

In this section we continue with the notation of Section~\ref{The underlying quadratic space}, specialized to the case $n=2$. So $V$ is a quadratic space of signature $(4,4)$ and $G=\SO(V)$. We let $V(\Z)$ and $V_{2,2}(\Z)$ be as in \eqref{eqn-defn-V(Z)}. Then $V_{2,2}(\Z)\simeq M_{2}(\Z)$, the space of $2$-by-$2$ integer matrices equipped with the determinant quadratic form. The identification $V_{2,2}(\Z)\simeq M_{2}(\Z)$ is given by
\[m_{11}b_3 - m_{21}b_4 +m_{12}b_{-4} + m_{22} b_{-3} \mapsto \mb{m_{11}}{m_{12}}{m_{21}}{m_{22}}.\]
As in Section~\ref{subsec:Definition-QMFS}, we write 
$
G(\widehat{\Z})
$ for the stabilizer of $V(\widehat{\Z})$ inside $G(\A_f)$; automorphic forms invariant by $G(\widehat{\Z})$ are said to be of level one.  Fourier coefficients of elements in $\mathcal{A}_{0, \Z}(G,\ell)$ are indexed by pairs $[T_1,T_2]\in V_{2,2}(\Z)^{\oplus 2}$ satisfying $[T_1,T_2] \succ 0$.

\subsection{The Quaternionic Saito-Kurokawa Lift}
\label{subsection-the-quaternionic-Saito-Kurokawa-lift}
In this subsection we recall the statement of \cite[Theorem 4.1.1]{pollackCuspidal}, in the case $n=2$. 
Before we give the statement of (loc. cit.), suppose $\lambda=[T_1,T_2]\in V^{\oplus 2}$ and define a $2$-by-$2$ matrix with entries in $\Q$ via the formula
\begin{equation}
\label{Definition-S(T1,T2)}
S(\lambda)=\frac{1}{2}\begin{pmatrix} (T_1,T_1) &(T_1,T_2) \\ (T_2,T_1) &(T_2,T_2)\end{pmatrix}. 
\end{equation}
    \begin{theorem}\emph{\cite[Theorem 4.1.1]{pollackCuspidal}}
    \label{pollack4.1.1} 
    Suppose $\ell\geq 16$ is even and let 
    $$F(Z)=\sum_{T>0}a_F(T)\exp(2\pi i \mathrm{tr}(TZ))$$ be a weight $\ell$ cuspidal Siegel modular form on $\Sp(4)$ of level one. There exists a unique element $\theta^{\ast}(F)\in \mathcal{A}_{0,\Z}(G, \ell)$, such that if $\lambda=[T_1,T_2]\in V_{2,2}(\Z)^{\oplus 2}$ then the $\lambda$ Fourier coefficient of $\theta^{\ast}(F)$ is
\begin{equation}\label{fouriercoefficients}
		a_{\theta^{\ast}(F)}(\lambda)=\sum_{\substack{r\in \GL(2,\Z)\backslash M_2(\Z)^{\det\neq 0}\\\lambda r^{-1}\in V_{2,2}(\Z)^{\oplus 2}}}|\det(r)|^{\ell-1}\overline{a_F(^{t}r^{-1}S(\lambda)r^{-1})}.
\end{equation}
Here $\lambda r^{-1}$ is the matrix product of the $1\times 2$ row vector $\lambda= [T_1,T_2]$ with the matrix $r^{-1}$.
    \end{theorem}
    \begin{remark}
The construction of $\theta^{\ast}(F)$ is a special case of the theta lifting from $\Sp(4)$ to  $\SO(V)$. Theorem \ref{pollack4.1.1} is analogous to the computations of Oda \cite{Oda77} who computed the Fourier coefficients of the theta lifts of holomorphic modular forms on $\SL(2)$ and $\mathrm{Mp}(2)$ to holomorphic modular forms on $\SO(2,n)$. 
    \end{remark}
\subsection{Poincar\'e Lifts on $G$}
\label{subsection-Poincare-Lifts-on-$G$}
Let $T$ denote a half-integral, $2$-by-$2$, positive definite matrix. Write $\mathcal{H}_2$ for the Siegel upper half space of degree $2$, and define $S_{\ell}(\Sp(4,\Z))$ to be the space of weight $\ell$ and level one holomorphic cuspidal modular forms on $\mathcal{H}_2$. Write $j(\cdot, \cdot )\colon \Sp(4,\R)\times \mathcal{H}_2\to \C$ for the standard factor of automorphy and recall the classical Poincar\'e series 
\begin{equation}
\label{holomorphic-poincare-series}
P_{T,\ell}(Z)=\sum_{\gamma\in \left\{\left(\begin{smallmatrix} I_2 &\ast \\ 0 &I_2 \end{smallmatrix}\right)\right\}\backslash \Sp(4,\Z)}j(\gamma, Z)^{-\ell}e^{2\pi i \mathrm{Tr}(T\gamma(Z))}
\end{equation}
Here $I_2$ is the $2\times 2$ identity matrix.
If $\ell$ is sufficiently large then the sum (\ref{holomorphic-poincare-series}) converges absolutely to an element in $S_{\ell}(\Sp(4,\Z))$. 

We now set-up notation to study the theta lift $\theta^{\ast}(P_{T,\ell})$. Recall that in Section~\ref{subsection-compact-subgroups} we introduced a distinguished map $K^0\to \SU(2)/\mu_2$ together with an identification $\Lie(\SU(2)/\mu_2)\otimes_{\R}\C\simeq \mathrm{Sym}^2(\C^2)$. We thus have a $K^0$-equivariant projection
\begin{equation}\label{defn-prK}
\mathrm{pr}_K\colon \Lie(\SO(V))\to \Lie(K^0)\to \Lie(\SU(2)/\mu_2)\to \mathrm{Sym}^2(\C^2).
\end{equation}
Given $X\in \mathrm{Sym}^2(\C^2)$, let $X^\ell$ denote the element of $\mathbf{V}_\ell$ obtained by raising $X$ to the $\ell$-th power.  The $\C$-linear pairing $\{ \,,\,\}_{K^0}$ on $\mathbf{V}_{\ell}$ can be normalized to be positive-definite when restricted to the (real) Lie algebra $\su_2^{dist}$ of the distinguished $\SU(2)$.  For $u \in \su_2^{dist}$, we write $\|u\|^2=\{u, u\}_{K^0}$. If $[v_1,v_2]\in V(\R)\times V(\R)$ is such that $v_1$ and $v_2$ span a positive definite two-plane then $\mathrm{pr}_K(v_1\wedge v_2)\neq 0$ \cite[Lemma 3.2.2]{pollackCuspidal}, thus we may define
\begin{equation}\label{eqn:Bv1v2Def}
B_{[v_1,v_2]}\colon \SO(V)(\R)\to \mathbf{V}_{\ell}, \quad g\mapsto \frac{\mathrm{pr}_K(\mathrm{Ad}(g^{-1})\cdot v_1\wedge v_2)^\ell}{\|\mathrm{pr}_K(\mathrm{Ad}(g^{-1})\cdot v_1\wedge v_2)\|^{2\ell+1}}.
\end{equation}
\begin{prop}\emph{\cite[Corollary 3.3.7]{pollackCuspidal}}
\label{Proposition-Explicit-Poincare-Lift}
Suppose $\ell\geq 16$ is even.
Let $T$ denote a half-integral, $2$-by-$2$, positive definite matrix.
With notation as in \eqref{Definition-S(T1,T2)}, suppose $[v_1,v_2]\in V(\Z)^{\oplus 2}$ satisfies $S([v_1,v_2])=T$. Write $\Gamma=\SO(V)(\Q)\cap G(\widehat{\Z})$ and consider the identification 
\[
    \Gamma\backslash\SO(V)(\R)\xrightarrow{\sim} \SO(V)(\Q)\backslash \SO(V)(\A)/G(\widehat{\Z}), \quad \Gamma g_{\infty}\mapsto \SO(V)(\Q)g_{\infty}G(\widehat{\Z}).
\]
Then as a function on $\Gamma\backslash \SO(V)(\R)$, the theta lift $\theta^{\ast}(P_{T,\ell})$ is proportional to the function 
\begin{equation}
\label{definition-poincare-lift}
Q_{T,\ell}\colon\Gamma\backslash \SO(V)(\R)\to \mathbf{V}_\ell, \qquad \Gamma g_{\infty}\mapsto \sum_{\substack{[v_1,v_2]\in V(\Z)\times V(\Z)\\S([v_1,v_2])=T}}B_{[v_1,v_2]}(g_{\infty}).
\end{equation}
\end{prop}
\subsection{The Quaternionic Maass Spezialschar and Saito-Kurokawa Subspace}
\label{subsection-the-quaternionic-Maass-Spezialschar}
By Theorem \ref{pollack4.1.1} we have a lifting 
$$
\theta^{\ast}\colon S_{\ell}(\Sp(4,\Z))\to \mathcal{A}_{0,\Z}(G,\ell), \qquad F\mapsto \theta^{\ast}(F).
$$
The \textit{quaternionic Saito-Kurokawa subspace} $\mathrm{SK}_{\ell}\subseteq \mathcal{A}_{0,\Z}(G,\ell)$ is defined as the image of $\theta^{\ast}$.

\begin{definition}
    \label{defn-strong-primitivity}
    Suppose $\lambda\in M_2(\Z)^{\oplus 2}$. We say that $\lambda$ is \textit{strongly primitive} if 
    $$
    \{r\in \GL(2,\Q)\cap M_2(\Z) \colon \lambda r^{-1}\in M_2(\Z)^{\oplus 2}\}=\GL(2,\Z).
    $$
\end{definition}
\begin{remark}
\label{remark-section-5.3}
Note that if $\lambda=[T_1,T_2]\in M_2(\Z)^{\oplus 2}$ is strongly primitive and $F\in S_{\ell}(\Sp(4,\Z))$, then by Theorem \ref{pollack4.1.1}, $a_{\theta^{\ast}(F)}(\lambda)=\overline{a_F(S(\lambda))}$. This property is reminiscent of the classical Saito-Kurokawa subspace. Indeed, if $T$ is a primitive binary quadratic form and $F\in S_{\ell}(\Sp(4,\Z))$ is the theta lift of an anti-holomorphic form $\overline{f}$ on $\mathrm{Mp}(2)$ then $a_F(T)=\overline{a_{f}(\det(T))}$.
\end{remark}
As a corollary to Theorem \ref{pollack4.1.1} we deduce the following.
\begin{corollary}
\label{corollary-coarser-invariants}
 Suppose $\lambda_1, \lambda_2\in M_2(\Z)^{\oplus 2}$ are strongly primitive elements satisfying $S(\lambda_1)=S(\lambda_2)$. If $\varphi\in \mathrm{SK}_{\ell}$ then $a_{\varphi}(\lambda_1)=a_{\varphi}(\lambda_2)$.
\end{corollary}
\begin{lemma} 
\label{primitive-lemma}
Suppose $T=[a,b,c]$ is a half-integral symmetric matrix with $a,b,c\in \Z$. Then there exists a strongly primitive element $\lambda=[T_1,T_2]\in M_2(\Z)^{\oplus 2}$ such that $T=S(\lambda)$.
\end{lemma}
\begin{proof}
We let $T_1=\begin{pmatrix} 1&0 \\ -b &a \end{pmatrix}$ and $T_2=\begin{pmatrix} 0 &1 \\ -c &0 \end{pmatrix}$ so that if $\lambda=[T_1,T_2]$ then $S(\lambda)=T$. Suppose $r \in \GL(2,\Q)\cap M_2(\Z)$ satisfies $\lambda r^{-1}\in M_2(\Z)^{\oplus 2}$ and write $r^{-1}=\begin{pmatrix} w &x \\ y &z\end{pmatrix}$ with $w,x,y,z\in \Q$. Then 
$$\lambda r^{-1}=\left[
\begin{pmatrix} w & y \\ \ast & \ast
\end{pmatrix},
\begin{pmatrix} x & z \\ \ast & \ast
\end{pmatrix}\right]\in M_2(\Z)^{\oplus 2}. 
$$
So $x,y,z,w\in \Z$ which implies $\det(r^{-1})=\det(r)^{-1}\in \Z$. So $r\in \GL(2,\Z)$. \end{proof}
\begin{definition}
\label{definition-Quaternionic-Spezialschar}
    Define the \textit{quaternionic Maass Spezialschar} $\mathrm{MS}_{\ell}$ as the subspace of $\mathcal{A}_{0,\Z}(G, \ell)$ consisting of forms $\varphi$ satisfying conditions (1) and (2) below. 
    \begin{compactenum}
        \item[(1)] If $\lambda_1,\lambda_2\in M_2(\Z)^{\oplus 2}$ are strongly primitive and $S(\lambda_1)=S(\lambda_2)$ then 
        $$
        a_{\varphi}(\lambda_1)=a_{\varphi}(\lambda_2).
        $$
    \end{compactenum}
    Let $\lambda\in M_2(\Z)^{\oplus 2}$. By Lemma  \ref{primitive-lemma} there exists a strongly primitive element $\Breve{\lambda}\in M_2(\Z)^{\oplus 2}$ such that $S(\lambda)=S(\Breve{\lambda})$. If $\varphi\in \mathcal{A}_{0,\Z}(G, \ell)$ satisfies condition (1) then
    $$
    a_{\varphi}^{\mathrm{prim}}(\lambda):=a_{\varphi}(\Breve{\lambda})
    $$
    is well-defined independent of the choice of $\Breve{\lambda}$.
    \begin{compactenum}
        \item[(2)] If $\lambda=[T_1,T_2]\in M_2(\Z)^{\oplus 2}$ satisfies $[T_1,T_2]\succ 0$ (see (\ref{Definition-S(T1,T2)}))  then 
        \begin{equation}
        \label{eqn-quat-maass-relation}
        a_{\varphi}(\lambda)=
        \sum_{\substack{r\in \GL(2,\Z)\backslash (\GL(2,\Q)\cap M_2(\Z))\\ \lambda r^{-1}\in M_2(\Z)^{\oplus 2}}}|\det(r)|^{\ell-1}a_{\varphi}^{\mathrm{prim}}(\lambda r^{-1}).
        \end{equation}
    \end{compactenum}
\end{definition}

\begin{remark} Note that the two conditions of Definition \ref{definition-Quaternionic-Spezialschar} can be replaced by the single condition 
\[a_{\varphi}(\lambda)=
        \sum_{\substack{r\in \GL(2,\Z)\backslash (\GL(2,\Q)\cap M_2(\Z))\\ \lambda r^{-1}\in M_2(\Z)^{\oplus 2}}}|\det(r)|^{\ell-1}a_{\varphi}(\Breve{(\lambda r^{-1})}),\]
where if $\mu \in M_2(\Z)^{\oplus 2}$ has $S(\mu) = [a,b,c]$ then $\Breve{\mu} = \left[\begin{pmatrix} 1&0 \\ -b &a \end{pmatrix}, \begin{pmatrix} 0 &1 \\ -c &0 \end{pmatrix}\right]$.
\end{remark}

\begin{lemma}
\label{Lemma-SK-in-MS} If $\ell\geq 16$ is even, then we have an inclusion $\mathrm{SK}_{\ell}\subseteq \mathrm{MS}_{\ell}$. 
\end{lemma}
\begin{proof}
    Suppose $\varphi=\theta^{\ast}(F)\in \mathrm{SK}_{\ell}$. The fact that $\varphi$ satisfies property (1) of Definition \ref{definition-Quaternionic-Spezialschar} is the content of Corollary \ref{corollary-coarser-invariants}. It remains to show that $\varphi$ satisfies condition (2). Suppose $\lambda=[T_1, T_2]\in M_2(\Z)^{\oplus 2}$ is such that $T_1, T_2\succ 0$. Let $r\in \GL(2,\Q)\cap M_2(\Z)$ be such that $\lambda r^{-1}\in M_2(\Z)^{\oplus 2}$. By Theorem \ref{pollack4.1.1}, it suffices to show that 
    \begin{equation}
        \label{Proof-SK-inside-MS}
        a_{\varphi}^{\mathrm{prim}}(\lambda r^{-1})=\overline{a_F(^{t}r^{-1}S(\lambda)r^{-1})}.
    \end{equation}
    Since $S(\lambda r^{-1})=\ ^{t}r^{-1}S(\lambda)r^{-1}$, equality (\ref{Proof-SK-inside-MS}) follows from Theorem \ref{pollack4.1.1} (see Remark \ref{remark-section-5.3}). 
\end{proof}
\section{Dirichlet Series for \texorpdfstring{$L$}{L}-functions of Quaternionic Modular Forms}\label{sec:dirichlet}
In this section, we present a conjectural Dirichlet series for the $L$-functions of irreducible, cuspidal, quaternionic automorphic representations $\Pi$ of split $\mathrm{SO}(8)$ and show that this conjecture is satisfied by the Saito-Kurokawa lifts.  By a \emph{Dirichlet series} for an $L$-function of an automorphic representation $\Pi$, we mean a sum of the form $\sum_{n}{a_n n^{-s}}$ where the $a_n$ are given explicitly in terms of Fourier coefficients of an automorphic form in $\Pi$. 

\subsection{Conjecture on Dirichlet series}
We recall the formula of Andrianov \cite{And74} for the standard $L$-function of Siegel modular forms of full level on $\Sp(4)$. Let $f$ be a cuspidal Siegel modular Hecke eigenform of weight $\ell$ and level one with Fourier expansion $f(Z)=\sum a_f(T)\exp(2\pi i \mathrm{tr}(TZ))$, where the sum is taken over half-integral, positive definite, symmetric matrices $T$. Let $\pi$ be the automorphic representation generated by the automorphic function corresponding to $f$. For a fixed index $T$, let $S$ denote a finite set of primes containing those that divide $8\det(T)$ and let $\chi_T(\cdot)=(\cdot,-\det(T))_2$ be the automorphic character induced by the Hilbert symbol. Then
one has
\begin{equation}
\label{eq:dirichlet_sp4}
a(T) \frac{L^S(\pi,\mathrm{Std},s)}{\zeta^S(2s)L^S(\chi_T,s+1)} = \sum_{g \in M_2^S(\Z)/\GL(2,\Z)}{\frac{a(g^t T g)}{|\det(g)|^{s+ \ell-1}}}.
\end{equation}

Here $L^S(\pi,\mathrm{Std},s)$ is the partial standard $L$-function of $\pi$, $\zeta^S$ is the partial Riemann zeta-function, $L^S(\chi_T,s)$ is the partial Dirichlet $L$-function associated to $\chi_T$, and $M_2^S(\Z)$ denotes the $2 \times 2$ integer matrices with nonzero determinant, only divisible by primes away from $S$. See \cite[section 5.4]{PollackUG} for a modern reference and derivation.

To state the conjecture for Dirichlet series on split $\mathrm{SO}(8)$, let $\varphi$ be a cuspidal quaternionic modular Hecke eigenform on $\mathrm{SO}(8)$ of level one.  Let $\Pi$ denote the automorphic representation generated by $\varphi$; see Lemma \ref{lem:eigenHecke}.  Let $[T_1,T_2] \in M_2(\Z)^{\oplus 2}$ be strongly primitive with $T = S([T_1,T_2])$, and suppose $S$ is a finite set of primes containing all those that divide $8 \det(T)$. Following Definition \ref{Definition-of-FCs}, we write $a_\varphi([T_1,T_2])$ for the $[T_1,T_2]$-th Fourier coefficient of $\varphi$.  Finally, let $\zeta_{E}(s)$ denote the Dedekind zeta function of the quadratic field $E = \Q(\sqrt{-\det(T)})$ determined by the discriminant of $T$.

\begin{conjecture}\label{conj:dirichlet} Let the notation be as above, with $L^S(\Pi,\mathrm{Std},s)$ the partial standard $L$-function of $\Pi$.  Then for $\mathrm{Re}(s)$ sufficiently large, 
\[ a_\varphi([T_1,T_2])\frac{L^S(\Pi,\mathrm{Std},s)}{\zeta^S(2s) \zeta_E^S(s+1)} = \sum_{g \in M_2^S(\Z)/\GL(2,\Z)}\frac{a_\varphi([T_1,T_2] \cdot g)}{|\det(g)|^{s+\ell-1}}.\]
\end{conjecture}

We call the right hand side the Dirichlet series for $\varphi$ at $[T_1,T_2]$ and denote it by $D_\varphi(T_1,T_2)(s)$. The conjecture was found by analyzing a hypothetical Rankin-Selberg integral for cusp forms on $G$ similar to that used in the proof of Andrianov's formula.  We hope to prove that this integral does in fact represent the standard $L$-function in a future work. At this point, we are able to provide some evidence for this conjecture as a consequence of Lemma \ref{Lemma-SK-in-MS} by considering the behavior of the $L$-function of the theta lift and the Maass relations.
\subsection{Evidence from the Maass Spezialschar}
Let $\varphi=\theta^*(F)$ be the Saito-Kurokawa lift of a cuspidal Siegel modular eigenform $F$ of weight $\ell$ and level one, let $\Pi$ be the automorphic representation of $\SO(8)$ generated by $\varphi$, and let $\pi$ be the automorphic representation of $\Sp(4)$ generated by the automorphic form corresponding to $F$. By properties of the theta lifting \cite{Rallis1982}, we have that $L(\Pi,\mathrm{Std},s) = L(\pi,\mathrm{Std},s) \zeta(s-1)\zeta(s) \zeta(s+1)$.  Proposition \ref{prop:mass_dirichlet} shows that the Dirichlet series associated to any element of the Maass Spezialschar $\mathrm{MS}_\ell$ factors in a similar way. Before we can prove such a proposition, we need a lemma about the action of $M_2(\Z)$ on $M_2(\Z)^{\oplus 2}$.

We use the following lemma regarding strong primitivity.
\begin{lemma}\label{lem:superstrong} Suppose $\lambda \in M_2(\Z)^{\oplus 2}$ is strongly primitive and $r \in \GL(2,\Q)$ satisfies $\lambda \cdot r \in M_2(\Z)^{\oplus 2}$.  Then $r \in M_2(\Z)$.\end{lemma}
\begin{proof} We can write $r = k_1 d k_2$ with $k_j \in \GL(2,\Z)$ and $d = \mathrm{diag}(m_1/n_1, m_2/n_2)$ with $\mathrm{gcd}(m_1,n_1) = 1$, $\mathrm{gcd}(m_2, n_2) = 1$.  Suppose $\lambda \cdot k_1 = (\mu_1, \mu_2)$, so that $\mu_j \in M_2(\Z)$.  Then $\lambda \cdot r \in M_2(\Z)$ implies $\dfrac{m_1}{n_1}\mu_1 \in M_2(\Z)$ and $\dfrac{m_2}{n_2} \mu_2  \in M_2(\Z)$.  But from the fact that $\mathrm{gcd}(m_i,n_i) = 1$, we obtain $\dfrac{1}{n_1}\mu_1 \in M_2(\Z)$ and $\dfrac{1}{n_2}\mu_2 \in M_2(\Z)$.

Set $g = k_2^{-1} \mathrm{diag}(n_1, n_2) k_1^{-1}$.  Then $\lambda \cdot g^{-1} \in M_2(\Z)^{\oplus 2}$ by the previous line.  Hence $g \in \GL(2,\Z)$ so $n_1, n_2 = \pm 1$.  Consequently, $r \in M_2(\Z)$ as desired.
\end{proof}

\begin{lemma}\label{lem:prim_Saction}
     Let $S$ be our fixed finite set of primes and suppose $\lambda \in M_2(\Z)^{\oplus 2}$ is strongly primitive. Suppose $g\in M_2^S(\Z)$. If $r\in\GL(2,\Q)\cap M_2(\Z)$ with $\lambda\cdot g r^{-1}\in M_2(\Z)^{\oplus 2}$, then $r,\, gr^{-1}\in M_2^S(\Z)$.
\end{lemma}
\begin{proof} 
By Lemma \ref{lem:superstrong}, $gr^{-1} \in M_2(\Z)$. Then since $g = (gr^{-1}) r$ and because $\det(g)$ is a unit at primes in $S$, so are $\det(gr^{-1})$ and $\det(r)$.
\end{proof}

\begin{prop}\label{prop:mass_dirichlet}
Let $\varphi\in\mathcal{A}_{0,\Z}(G,\ell)$ be an element of the quaternionic Maass Spezialschar $\mathrm{MS}_\ell$ so that the Fourier coefficients $a_\varphi$ satisfy the conditions of Definition \ref{definition-Quaternionic-Spezialschar}. Then for a fixed $\lambda=[T_1,T_2]\in M_2(\Z)^{\oplus2}$, the Dirichlet series factors as
\begin{equation}\label{eq:maass_dirichlet}
D_\varphi(T_1,T_2)(s)=\sum_{r\in\GL(2,\Z)\backslash M_2^S(\Z)}|\det(r)|^{-s}\sum_{g\in M_2^S(\Z)/\GL(2,\Z)}\frac{a_\varphi^{\mathrm{prim}}(\lambda\cdot g)}{|\det(g)|^{s+\ell-1}}.
\end{equation}
\end{prop}
\begin{proof}
To prove this proposition, we calculate
\begin{align*}
D_\varphi(T_1,&T_2)(s)=\sum_{g \in M_2^S(\Z)/\GL(2,\Z)}\frac{a_\varphi(\lambda \cdot g)}{|\det(g)|^{s+\ell-1}}\\
&=\sum_{g \in M_2^S(\Z)/\GL(2,\Z)}|\det(g)|^{1-s-\ell}\left(\sum_{\substack{r\in \GL(2,\Z)\backslash (\GL(2,\Q)\cap M_2(\Z))\\\text{with }\lambda\cdot g r^{-1}\in M_2(\Z)^{\oplus 2}}}|\det(r)|^{\ell-1}a_{\varphi}^{\mathrm{prim}}(\lambda\cdot g r^{-1})\right)\\
&=\sum_{\substack{g\in M_2^S(\Z)/\GL(2,\Z)\\r\in \GL(2,\Z)\backslash M_2^S(\Z)\\\text{with }\lambda\cdot g r^{-1}\in M_2(\Z)^{\oplus 2}}}|\det(r)|^{-s}\frac{a_{\varphi}^{\mathrm{prim}}(\lambda\cdot g r^{-1})}{|\det(gr^{-1})|^{s+\ell-1}}\\
&=\sum_{d\in\Z_{>0}}\left(\sum_{\substack{r\in \GL(2,\Z)\backslash M_2^S(\Z)\\|\det(r)|=d}}|\det(r)|^{-s}\sum_{\substack{g\in M^S_2(\Z)/\GL(2,\Z)\\\lambda\cdot gr^{-1}\in M_2^S(\Z)^{\oplus 2}}}\frac{a_{\varphi}^{\mathrm{prim}}(\lambda\cdot g r^{-1})}{|\det(gr^{-1})|^{s+\ell-1}}\right),
\end{align*}
where the second line uses \eqref{eqn-quat-maass-relation}, the third is Lemma \ref{lem:prim_Saction}, and the last is grading by $|\det(r)|$. 

Now let $R_d=\{r_1,\dots,r_{N_d}\}$ be a fixed set of coset representatives for 
\[\{r\in\GL(2,\Z)\backslash M_2^S(\Z)\colon |\det(r)|=d\},\]
and let $\{g_j\colon j\in\Z_{>0}\}$ be a fixed set of coset representatives of $M_2^S(\Z)/\GL(2,\Z)$. Then applying Lemma \ref{lem:prim_Saction} again, the sum rearranges as
\begin{align*}
 D_\varphi(T_1,T_2)(s)&=\sum_{\substack{d\in\Z_{>0}\\h\in M_2^S(\Z)/\GL(2,\Z)}}d^{-s}\frac{a^{\mathrm{prim}}_\varphi(\lambda\cdot h)}{|\det(h)|^{s+\ell-1}}\left(\sum_{i\in\{1,\dots,N_d\},j\in\Z^+}\mathrm{char}_{h\GL(2,\Z)}(g_jr_i^{-1})\right).
\end{align*}
Here $\mathrm{char}_{h\GL(2,\Z)}$ is the characteristic function of the coset $h\GL(2,\Z)$. 

Recall that $N_d=|R_d|$ is the number of coset representatives of $\GL(2,\Z)\backslash M^S_2(\Z)$ of determinant $d$. We now show that for a given $h\in M_2^S(\Z)$,  
\[N_d=|\{(i,j)\colon r_i\in R_d,\, g_jr_i^{-1}\in h\GL(2,\Z)\}|,\]
and this will complete the proof. For a fixed element $r_i\in R_d$, let $$S(d,i)=\{g_j\colon g_j \in h\GL(2,\Z)r_i\}$$ be the set of fixed coset representatives $g_j$ in the double coset determined by $r_i$. As these cosets are disjoint, it remains to show that $N_d=\sum_{i=1}^{N_d}|S(d,i)|=|\bigcup_{i=1}^{N_d}S(d,i)|.$ As $R_d$ is a complete set of coset representatives, we have $\bigcup_{i=1}^{N_d}S(d,i)=h(M_2^S(\Z)^{|\det|=d})$, but this is a union of $N_d$ $\GL(2,\Z)$ cosets with $\GL(2,\Z)$ now acting on the right. Thus, there are exactly $N_d$ representatives $g_j$ in the set, and the Dirichlet series factors as in \eqref{eq:maass_dirichlet}.
\end{proof}
We will also need the following well-known fact.
\begin{lemma}\label{lem:detzeta}
 We have an equality of meromorphic functions, 
 $$\sum_{r\in\GL(2,\Z)\backslash M^S_2(\Z)}|\det(r)|^{-s}=\zeta^S(s)\zeta^S(s-1).$$
\end{lemma}
\begin{proof}
    For each determinant $d$, there are $\sigma(d) = \sum_{r|d}{r}$ orbits of $\GL(2,\Z)$ on the set of determinant $\pm d$ matrices in $M_2^S(\Z)$. Observing that the Dirichlet series of $\sigma$ is the convolution of 1 with the identity function, we have the result.
\end{proof}
\begin{corollary}
Let $F$ be a cuspidal Siegel modular eigenform of weight $\ell$, genus $2$, and level one. Then Conjecture \ref{conj:dirichlet} holds for $\varphi=\theta^*(F)$.
\end{corollary}
\begin{proof}
    By Lemma \ref{Lemma-SK-in-MS}, $\theta^*(F)$ is in the Maass Spezialschar and thus for any choice of strongly primitive $[T_1,T_2]\in M_2(\Z)^{\oplus 2}$ the Dirichlet series factors as \eqref{eq:maass_dirichlet} by Proposition \ref{prop:mass_dirichlet}. Now by Lemma \ref{primitive-lemma} and Theorem \ref{pollack4.1.1} we have that $a_{\theta^*(F)}^{\mathrm{prim}}(\lambda\cdot g)=\overline{a_F(S(\lambda)\cdot g)}$ for all $g\in M_2^S(\Z)$. We note that the Hecke eigenvalues of $F$ are totally real as the Hecke operators for $\Sp(4,\Z)$ are self-adjoint. Hence, by Lemma \ref{lem:detzeta} and \eqref{eq:dirichlet_sp4}
    \begin{align*}
     D_{\theta^*(F)}(T_1,T_2)(s)&=\sum_{r\in\GL(2,\Z)\backslash M_2^S(\Z)}|\det(r)|^{-s}\sum_{g\in M_2^S(\Z)/\GL(2,\Z)}\frac{\overline{a_F(S(\lambda)\cdot g)}}{|\det(g)|^{s+\ell-1}}\\
     &=\zeta^S(s)\zeta^S(s-1)\overline{a_F(S(\lambda))} \frac{L^S(\pi_F,\mathrm{Std},s)}{\zeta^S(2s)L^S(\chi_{S(\lambda)},s+1)}\\
     &=a_{\theta^*(F)}(\lambda)\frac{L^S(\Pi_{\theta^*(F)},\mathrm{Std},s)}{\zeta^S(s+1)\zeta^S(2s)L^S(\chi_{S(\lambda)},s+1)}.
    \end{align*}
The last equality follows from \cite{Rallis1982}, which implies that 
\[L^S(\Pi_{\theta^*(F)},\mathrm{Std},s) = \zeta^S(s)\zeta^S(s-1) \zeta^S(s+1) L^S(\pi_F,\mathrm{Std},s).\]
To complete the proof, we note that $\zeta^S_E(s)=\zeta^S(s)L^S(\chi_{S(\lambda)},s)$. 
\end{proof}
\section{The Fourier-Jacobi coefficient}\label{sec:FJ}
In this section, we complete the proofs of some of the main theorems of the paper, namely Theorems~\ref{thm:introFJgeneral},  \ref{thm:introSp4}, and \ref{thm:intoSKMS}.  The main technical result of this section is Proposition \ref{Archimedean-Integral}.  This proposition, in conjunction with Corollary \ref{cor-partial-FE-Xi}, is used to complete the proofs of these results.

 \subsection{Fourier Coefficients along the Siegel parabolic subgroup.}
 \label{Fourier Coefficients along NR}
 Let $\varphi\colon G(\Q)\backslash G(\A) \to \mathbf{V}_{\ell}$ denote a vector-valued automorphic form on $G(\A)$ such that $\varphi(gk)=k^{-1}\varphi(g)$ for all $k\in K^0$ and $g\in G(\A)$. The notation $k^{-1}\varphi(g)$ is explained in Definition~\ref{defn:QMFs1}.\\
 \indent Recall the inclusion $V_{2,n}\subseteq V_{3,n+1}$ as well as the fixed positive definite two-plane $V_2^+\subseteq V_{2,n}$. Recall also the orthogonal basis for $V_2^+$ defined in Section \ref{sec:notation} consisting of vectors $y_0$ and $y_1$ such that $(y_0,y_0)=(y_1,y_1)=2$ and the character $\chi_{y_0}$ defined in \eqref{characters-of-NQ}. Let $\xi^{\varphi}$ denote the Fourier coefficient of $\varphi$ along $N_Q$ corresponding to the character $\chi_{y_0}$ as in \eqref{defn-general-FC}, and let $H$ denote the stabilizer of $y_0$ in $\SO(V_{3,n+1})\leq M_Q$. Then $\xi^{\varphi}$ defines an automorphic function
 $$
 \xi^{\varphi}\colon H(\A)\to \mathbf{V}_{\ell}, \qquad \xi^{\varphi}(h)=\int_{V_{3,n+1}(\Q)\backslash V_{3,n+1}(\A)}\psi^{-1}((v,y_0))\varphi(\exp(b_1\wedge v)h)\,dv. 
 $$
 Let $V_{2,n+1}$ be the orthogonal complement of $\Q y_0$ in $V_{3,n+1}$. Through its action on $$V_{3,n+1}/ \Q y_0\simeq V_{2,n+1},$$ the group $H$ is identified with $\SO(V_{2,n+1})$.  \\
 \indent Write $R=M_R\cdot N_R$ for the parabolic subgroup in $H$ stabilizing the line $\Q b_2$. The reader is referred to Section~\ref{subsec:HMFS} for notation concerning the group $H$.
The following lemma gives a ``soft" formula for the Fourier coefficients of $\xi^{\varphi}$ along the unipotent radical $N_R\simeq V_{1,n}$. We remind the reader that $y_0=b_3+b_{-3}$.
\begin{lemma}
\label{Soft Formula}
Let $T\in V_{1,n}$ and write $\xi^{\varphi}_{T}$ to denote the Fourier coefficient of $\xi^{\varphi}$ along $N_R$ corresponding to the character 
$$
\eta_{T}\colon N_R(\A)\to \C^1, \qquad \exp(b_2\wedge v)\mapsto \psi((T,v)).
$$
If $\varphi_{N_P, \varepsilon_{[y_0,T]}}(g)=\displaystyle{\int_{N_P(\Q)\backslash N_P(\A)}}\varepsilon_{[y_0,T]}^{-1}(n)\varphi(ng)dn$ and $h\in H(\A)$ then 
$$
\xi^{\varphi}_T(h)=\displaystyle{\int_{\A}}\varphi_{N_P, \varepsilon_{[y_0,T]}}(\exp(sb_1\wedge b_{-2})h)ds
$$
\end{lemma} 
\begin{proof} Throughout the proof we adopt the convention that if $\mathcal{G}$ is an algebraic group over $\Q$ then $[\mathcal{G}]:=\mathcal{G}(\Q)\backslash \mathcal{G}(\A)$. Let $X=M_P\cap N_Q =\exp(\Q b_1 \wedge b_{-2})$ and fix $T\in V_{1,n}$.  Let $N_P^1 = \exp(b_1 \wedge V_{2,n})$ and let $V_{1,n}$ denote the orthogonal complement to the line $\Q y_0$ in $V_{2,n}$.  We have $N_Q = X N_{P}^1 Z$ and $N_R = \exp(b_2 \wedge V_{1,n})$.  Note that if $w \in N_R$, $x \in X$ and $n_1 \in N_P^1$ then 
\[
n_1 x w = n_1 \{x,w\} w x
\]
and the commutator $\{x,w\}$ satisfies $\chi_{y_0}(\{x,w\}) = 1$.  Thus
\begin{align}\label{eqnalign:FCsoftDef}
\xi_{T}^\varphi(h) &= \int_{[N_R]}\int_{[N_Q]}\varepsilon_{[0,T]}^{-1}(w)\chi_{y_0}^{-1}(u)\varphi(uwh)\,du\,dw \notag \\ 
&=\int_{[N_R]}\int_{[X]} \int_{[N_P^1]}\varepsilon_{[0,T]}^{-1}(w)\chi_{y_0}^{-1}(n_1)\varphi_Z(n_1 x wh)\,dn_1\,dx\,dw \notag \\ 
&=\int_{[X]}\int_{[N_R]} \int_{[N_P^1]}\varepsilon_{[0,T]}^{-1}(w)\chi_{y_0}^{-1}(n_1)\varphi_Z(n_1 w xh)\,dn_1\,dw\,dx.
\end{align}
For $g \in G(\A)$, we have
\begin{align}\label{eqnalign:FCpartial}
\int_{[N_R]}\int_{[N_P^1]}\varepsilon_{[0,T]}^{-1}(w)\chi_{y_0}^{-1}(n_1) \varphi_Z(n_1 wg)\,dn_1\,dw &= \sum_{\mu \in \Q}\varphi_{[y_0,\mu y_0 + T]}(g) \notag\\
&= \sum_{\delta \in X(\Q)}\varphi_{[y_0,T]}(\delta g).\end{align}
Plugging \eqref{eqnalign:FCpartial} into \eqref{eqnalign:FCsoftDef} gives the lemma.
\end{proof}
\indent The following corollary is a consequence of Lemma~\ref{Soft Formula} and Corollary~\ref{FE-CUSPIDAL-QMF}. 
\begin{corollary}
\label{cor-partial-FE-Xi}
Suppose $\varphi\colon G(\Q)\backslash G(\A)\to \mathbf{V}_{\ell}$ is a weight $\ell$ quaternionic modular form on $G$. Then $\xi^{\varphi}_{T}(h)\not\equiv 0$ only if $T\in V_{1,n}$  satisfies $(T,T)\geq 0$. Moreover, with notation as in (\ref{FE-QMF}) the Fourier expansion of $\xi^{\varphi}$ along $R$ takes the form 
\begin{equation}
\label{Partial-FE-of-Xi}
\xi^{\varphi}(h_fh_{\infty})=\sum_{T\in V_{1,n} \colon (T,T)\geq 0}a_T(\xi^{\varphi},h_f)\int_{\R}\mathcal{W}_{[2\pi y_0,2\pi T]}(\exp(s_{\infty}b_1\wedge b_{-2})h_{\infty})ds_{\infty}
\end{equation}
where $h_fh_{\infty}\in H(\A)$ and $a_T(\xi^{\varphi},h_f)=\displaystyle{\int_{\A_f}}a_{[y_0,T]}(\varphi,\exp(s_{f}b_1\wedge b_{-2})h_f)ds_f$.
\end{corollary}
\subsection{The Archimedean Component of $\xi_T^{\varphi}$}
 In this subsection we further refine (\ref{Partial-FE-of-Xi}) and explain how our refinement can be used to obtain a scalar-valued holomorphic modular form on $H(\R)$. First we recall that $M_R$ is identified with the product $\mathbb{G}_m\times \SO(V_{1,n})$ via its action on the decomposition (\ref{orthogonalization-V(1,n)}). We write elements $m\in M_R$ as pairs $m=(t,u)$ with $t\in \mathbb{G}_m$ and $u\in \SO(V_{1,n})$. The coordinate $t\in \mathbb{G}_m$ is normalized so that $(t,1)\cdot b_2=tb_2$.
\begin{proposition}
\label{Archimedean-Integral}
Fix $T\in V_{1,n}$ such that $[y_0,T]\succeq 0$. If $h_{\infty}=(t,u)\in M_R(\R)^0$ with $t\in \R_{>0}$ and $u\in \SO(V_{1,n})(\R)^0$ then 
\begin{equation}
\label{Miracle}
\int_{\R}\mathcal{W}_{[y_0,T]}(\exp(s_{\infty}b_1\wedge b_{-2})h_{\infty})ds_{\infty}=  
\frac{\pi t^{\ell}e^{-(2-t(T, u\cdot y_1))}}{2}\sum_{-\ell\leq v\leq \ell}i^v\frac{x^{\ell+v}y^{\ell-v}}{(\ell-v)!(\ell+v)!}.
\end{equation}
\end{proposition}
\begin{proof} To begin, we apply (\ref{K-Bessel-Magic}) to explicate the integrand in the left hand side of (\ref{Miracle}). Let $s\in \R_{>0}$ and recall that $h_{\infty}=(t,u)$ with $t\in \R_{>0}$ and $u\in \SO(V_{1,n})(\R)^0$. Then unraveling definitions one obtains
\begin{equation}
\label{computation-of-beta}
\beta_{[y_0,T]}(\exp(sb_1\wedge b_{-2})h_{\infty})
=-2st+i(2-t(T,u\cdot y_1)) 
\end{equation}
To simplify notation we write $w=t(T,u\cdot y_1)$. If $-\ell\leq v\leq \ell$ then (\ref{computation-of-beta}) together with a manipulation in elementary calculus gives 
\begin{align*}
 \int_{-\infty}^{\infty}&\left(\frac{\beta_{[y_0,T]}(\exp(sb_1\wedge b_{-2})h_{\infty})}{|\beta_{[y_0,T]}(\exp(sb_1\wedge b_{-2})h_{\infty})|}\right)^vK_v(|\beta_{[y_0,T]}(\exp(sb_1\wedge b_{-2})h_{\infty})|)ds \\
 &=\frac{1}{t}\sum_{k=0}^{\lfloor |v|/2\rfloor}\binom{|v|}{2k} (i\,\mathrm{sgn}(v)(2-w))^{|v|-2k}\int_{0}^{\infty}\frac{s^{2k}}{\sqrt{(s^2+|2-w|^2)^{|v|}}}K_{|v|}\left(\sqrt{s^2+|2-w|^2}\right)ds. 
\end{align*}
Hence the formula \cite[p. 693, 6.596(3)]{BigBookIntegrals} together with (\ref{K-Bessel-Magic}) implies 
\begin{align}
\label{Main-Archimedean-Integral}
\int_{\R}&\mathcal{W}_{[y_0,T]}(\exp(s_{\infty}b_1\wedge b_{-2})h_{\infty})ds_{\infty}   \nonumber\\ 
=&
t^{\ell}\sum_{-\ell\leq v\leq \ell}\left(\sum_{k=0}^{\lfloor |v|/2\rfloor}\binom{|v|}{2k} \frac{(i\,\mathrm{sgn}(v)(2-w))^{|v|-2k}2^{(2k-1)/2}\Gamma\left(\frac{2k+1}{2}\right)}{|2-w|^{|v|-(2k+1)/2}}\right.\nonumber \\ 
&\left.\times K_{|v|-(2k+1)/2}(|2-w|) \frac{x^{\ell+v}y^{\ell-v}}{(\ell-v)!(\ell+v)!}\right). 
\end{align}
\begin{claim} 
\label{claim1}
If $h_{\infty}\in M_R(\R)^0$ then $\mathrm{Im}(\beta_{[y_0,T]}(h_{\infty}))=2-w$ is positive.
\end{claim}  
\begin{proof}
To begin the proof of Claim~\ref{claim1} we note that if $h_{\infty}\in M_R(\R)^0$ satisfies $\mathrm{Im}(\beta_{[y_0,T]}(h_{\infty}))=0$ then \eqref{computation-of-beta} implies that $\beta_{[y_0,T]}(h_{\infty})=0$ which contradicts our assumption that $[y_0,T]\succeq 0$. This means that $2-t(T,u\cdot y_1)\neq 0$ for all $t\in \R_{>0}$ and $u\in \SO(V_{1,n})(\R)^0$, completing the proof of Claim~\ref{claim1}.
\end{proof}
\indent Given $X\in \R_{>0}$ and $-\ell\leq v\leq \ell$ define 
\begin{equation}
\label{sum-fun}
S_v(X):=\sum_{k=0}^{\lfloor |v|/2\rfloor}\binom{|v|}{2k} \frac{(i\mathrm{sgn}(v)X)^{|v|-2k}2^{(2k-1)/2}\Gamma\left(\frac{2k+1}{2}\right)}{X^{|v|-(2k+1)/2}}K_{|v|-(2k+1)/2}(X).
\end{equation}
Using Claim \ref{claim1}, the expression (\ref{Main-Archimedean-Integral}) simplifies to the form $t^{\ell}\sum_{-\ell\leq v\leq \ell}S_v(2-w)\frac{x^{\ell+v}y^{\ell-v}}{(\ell-v)!(\ell+v)!}$. Thus to complete the proof of Proposition~\ref{Archimedean-Integral} it remains to establish the following claim.
\begin{claim}
\label{claim2}
    If $-\ell\leq v\leq \ell$ then $S_v(X)=\frac{\pi e^{-X}i^v}{2}$.
\end{claim}
\indent The formula \cite[p. 925, 8.468]{BigBookIntegrals} implies that $S_0(X)=\pi e^{-X}/2$ as required. Moreover, by inspection of (\ref{sum-fun}) we have $S_{-v}(X)=(-1)^vS_v(X)$. Hence we may assume that $1\leq v\leq \ell$. It follows from (loc. cit.) that
\begin{equation}
\label{Magic-Sum}
S_v(X)=
\frac{i^{v}\cdot \sqrt{\pi}\cdot e^{-X}}{\sqrt{2}X^{v-1}}\sum_{k=0}^{\lfloor v/2\rfloor}\sum_{j=0}^{v-k-1}(-1)^k\binom{v}{2k}\frac{(v-k-1+j)!2^{k-j-1/2}\Gamma(k+1/2)}{j!(v-k-1-j)!}X^{v-1-k-j}. 
\end{equation}
If $k\in \Z_{\geq 0}$ then $\Gamma(k+1/2)=(2k)!\sqrt{\pi}\cdot 4^{-k}/k!$. Therefore (\ref{Magic-Sum}) is equal to the expression 
\begin{equation}
\label{to-be-reindexed}
\frac{i^v\pi e^{-X}}{2\cdot X^{v-1}} \sum_{k=0}^{\lfloor v/2\rfloor}\sum_{j=0}^{v-k-1}(-1)^k\cdot \frac{v!(v-k-1+j)!2^{-k-j}}{(v-2k)!(v-k-1-j)!j!k!}X^{v-1-k-j}
\end{equation}
Let $m = j+k$.  We obtain that
\begin{equation}\label{eqn:SvXmSum}
S_v(X) = \frac{i^v \pi e^{-X}}{2 X^{v-1}} \left(\sum_{0 \leq k \leq m \leq v-1, 2k \leq v} (-1)^k \frac{v! (v-2k-1+m)! 2^{-m} X^{v-1-m}}{(v-2k)! (v-1-m)! k! (m-k)!}\right).
\end{equation}
Taking $m=0$ in \eqref{eqn:SvXmSum} gives $ \frac{\pi e^{-X}i^v}{2}$.  Thus to prove Claim \ref{claim2}, we check that the terms with $m \geq 1$ vanish.
\begin{claim}\label{claim:bigmVanish}
If $m \geq 1$, then
\[
\sum_{0 \leq k \leq m, 2k \leq v} \frac{ (m-1 + v-2k)! (-1)^k}{(v-2k)! k! (m-k)!} = 0.
\]
\end{claim}
\begin{proof} The sum to be considered is
\[\frac{1}{m} \sum_{0 \leq k \leq m} \binom{m}{k} \binom{m-1+v-2k}{m-1} (-1)^k.\]
But now $F(k) :=\binom{m-1+v-2k}{m-1}$ is a polynomial in $k$ of degree at most $m-1$, so $\sum_{k}{(-1)^k \binom{m}{k} F(k)} =0$ \cite[p. 190, (5.42)]{MR1001562}.\end{proof}
The proof of Claim \ref{claim2} is now complete, which completes the proof of Proposition \ref{Archimedean-Integral}. \end{proof}

Before proving the next corollary we introduce coordinates on the maximal compact subgroup $K_H:=K\cap H(\R)$ of $H(\R)$ and describe how the identity component $K_H^0$ acts on the representation $\mathbf{V}_{\ell}$. Let 
$W_2^+(\R)=\R\linspan\{u_2, v_2\}$ and let $W_{n+1}^-(\R)$ denote the orthogonal complement of $W_2^+(\R)$ inside $V_{2,n+1}(\R)$. Then
\begin{equation}
\label{SO(2,n+1)-max-compact}
V_{2,n+1}(\R)=W_2^+(\R)\oplus W_{n+1}^-(\R)
\end{equation}
is a decomposition of $V_{2,n+1}(\R)$ into definite subspaces. The subgroup $K_H$ is the stabilizer in $H(\R)$ of the decomposition (\ref{SO(2,n+1)-max-compact}). So $K_H$ has identity component $K_H^0\simeq \SO(2)\times \SO(n+1)$. To describe the action of $K_H^0$ on $\mathbf{V}_{\ell}$ we first note that $\SO(n+1)\leq \SO(V^-)$ and thus $\SO(n+1)$ acts trivially on $\mathbf{V}_{\ell}$. It remains to describe the action of the subgroup
\[
\SO(2)=\{\exp(\theta u_2\wedge v_2) \colon \theta\in \R\}.
\]
Since $u_2\wedge v_2=-\frac{1}{2}(e^+-f^+)+\frac{1}{2}(e'^+-f'^+)$ and $\{x,y\}$ is a weight basis for the action of the $\mathfrak{sl}_2$-triple $\{e^+, h^+, f^+\}$ we have that 
\[\exp(\theta u_2\wedge v_2)\cdot (-ix+y)^{2\ell}=e^{-i\ell\theta}(-ix+y)^{2\ell}.\]

Theorem~\ref{thm:introFJgeneral} is now proved with the following corollary.  Recall the $K^0$-invariant bilinear pairing $\{\,,\,\}_{K^0}$ defined in Section~\ref{subsection-compact-subgroups}.
\begin{corollary}
\label{Main Corollary}
Suppose $\varphi\colon G(\Q)\backslash G(\A)\to \mathbf{V}_{\ell}$ is a weight $\ell$ quaternionic modular form on $G$. 
Let $\xi^{\varphi}$ denote the $\mathbf{V}_{\ell}$-valued automorphic form obtained by restricting the $\chi_{y_0}$-th Fourier coefficient of $\varphi$ to the subgroup $H(\A)\leq M_Q(\A)$ (see Section~\ref{Fourier Coefficients along NR}). Then the function 
\[
\mathrm{FJ}_{\varphi}\colon H(\A)\to \C, \qquad h\mapsto  \overline{\{\xi^{\varphi}(h), (-ix+y)^{2\ell}\}_{K^0}}
\] 
is the automorphic function corresponding to a weight $\ell$ holomorphic modular form $f_{\varphi}$ on $H$. 
\end{corollary}
\begin{proof}
For $h \in H(\R)$ and $T\in 2\pi V_{2,n}(\Q)$, define $W_T^{FJ}(h)$ as 
\[W_T^{FJ}(h) = \left\{\int_{\R}\mathcal{W}_{[y_0,T]}(\exp(s_{\infty}b_1\wedge b_{-2})h)ds_{\infty}, (-ix+y)^{2\ell}\right\}_{K^0}.\]

We claim that $\overline{W_T^{FJ}(hk)} = j(k,-iy_1)^{-\ell}\overline{W_T^{FJ}(h)}$ for all $h \in H(\R)^0$ and $k \in K_H^0$, and that the function $Z \mapsto  j(g,-iy_1)^{\ell} \overline{W_T^{FJ}(g)}$, where $Z = g \cdot i(-y_1)$, is a holomorphic function of $Z$.  This will establish the corollary, and also give the form of the Fourier expansion of the holomorphic modular form associated to $\xi^\varphi$.

Recall that $u_2 = \frac{1}{\sqrt{2}}(b_2 + b_{-2})$ and $v_2 = \frac{1}{\sqrt{2}}(b_4+b_{-4}) = \frac{1}{\sqrt{2}}y_1$.  Above it is calculated that 
\[\exp(\theta u_2 \wedge v_2) (-ix+y)^{2\ell} = e^{-i \ell \theta} (-ix+y)^{2\ell}\]
and thus 
\[W_T^{FJ}(h e^{\theta u_2 \wedge v_2}) = e^{- i \ell \theta} W_T^{FJ}(h).\]

To prove the first claim, we must compute that, if $k = e^{\theta u_2 \wedge v_2}$, then $j(k,-iy_1) = e^{-i\theta}$.  But an immediate check gives $u_2 \wedge v_2 \cdot( b_2 + iy_1 + b_{-2}) = i (b_2 + iy_1 + b_{-2})$, which implies that indeed $j(k,-iy_1) = e^{-i\theta}$.

For the second claim, suppose $g = \exp(v \wedge b_2) h$, where $h = \diag(t, u, t^{-1})$ and $v\in V_{1,n}$.  Then one calculates $g \cdot (-iy_1) = Z$, with $Z = v  +i tu \cdot (-y_1)$ and $j(g,-iy_1) = t^{-1}$.  Now by Proposition \ref{Archimedean-Integral}, there is a nonzero constant $\eta$ such that
\[j(g,-iy_1)^{\ell} \overline{W_{T}^{FJ}(g)} = t^{-\ell} e^{i(T,v)}\overline{W_T^{FJ}(h)} = \eta e^{i(T,v)} e^{t(T,u \cdot y_1)} = \eta e^{i(T,Z)}.\]
The corollary follows.
\end{proof}

Suppose now $n=2$ and $\varphi \in \mathcal{A}(G,\ell)$ is a quaternionic modular form on $G$ of weight $\ell \geq 1$.  Let $f_{\varphi}$ denote the holomorphic modular form on $H = \SO(V_{2,3})$, where recall that $V_{2,3}$ is the orthogonal complement of $y_0 = b_3+b_{-3}$ in $\Q\linspan(b_2, b_3, b_4, b_{-4},b_{-3},b_{-2})$.  We have that $V_{1,2}$ is the orthogonal complement to $y_0$ in $\Q\linspan(b_3, b_4, b_{-4},b_{-3})$.    The Fourier coefficients of $f_{\varphi}$ are indexed by $T \in V_{1,2}^\vee(\Q)$ (which can be identified with $V_{1,2}(\Q)$ via the bilinear form).  The map 
\[
V_{2,2}(\Q)/\Q y_0 \rightarrow V_{1,2}^\vee(\Q)
\]
given by $T \mapsto \{v \mapsto (T,v)\}$ is an isomorphism, so we can index the Fourier coefficients of $f_{\varphi}$ in terms of $T \in V_{2,2}(\Q)/\Q y_0$.

With this indexing, the $T$-th Fourier coefficient of $f_{\varphi}$, call it $a_{T}(f_{\varphi}): H(\A_f) \rightarrow \C$, takes the following form:
\[
\overline{a_{T}(f_{\varphi})(h_f)} = \int_{\A_f} a_{[y_0,T]}(\varphi,\exp(s b_1 \wedge b_{-2})h_f)\,ds.
\]
This follows immediately from Corollary \ref{cor-partial-FE-Xi}.  Suppose now that $\varphi$ is right-invariant by a compact open subgroup $U$.  Take $M \in \Z_{\geq 1}$ so that if $s \in M \widehat{\Z}$ then $\exp(s b_1 \wedge b_{-2}) \in U$.  The natural map
\[
\Q/M\Z \rightarrow \A_f/M\widehat{\Z}
\]
is an isomorphism.  Consequently
\begin{equation}\label{eqn:FCsMinv}
\overline{a_{T}(f_{\varphi})(1)} = M^{-1} \sum_{\mu \in \Q/M\Z} a_{[y_0,T+\mu y_0]}(\varphi, 1).
\end{equation}
All but finitely many terms in the sum are $0$.

In case $\varphi$ is of level one, i.e., $\varphi \in \mathcal{A}_{\Z}(G,\ell)$, this expression simplifies further. 
\begin{corollary}\label{cor:FJcoeffLevel1} Suppose $n=2$, and $\ell \geq 1$ is an integer.  Let $\varphi \in \mathcal{A}_{\Z}(G,\ell)$ be a quaternionic modular form on $G$ of weight $\ell$ and level one, and let $f_{\varphi}$ denote the holomorphic modular form on $H = \SO(V_{2,3})$ associated to $\varphi$ via the Fourier-Jacobi coefficient of Corollary \ref{Main Corollary}.  For the Fourier coefficients $a_{T}(f_\varphi)(1)$, one has the following:
	\begin{enumerate}
		\item If $T \in V_{1,2}^\vee(\Q)$ is not in the image of $V_{2,2}(\Z)$, then $a_{T}(f_\varphi)(1) = 0$.
		\item Suppose $T \in V_{2,2}(\Z)$. Then $\overline{a_{T}(f_{\varphi})(1)} = a_{[y_0,T]}(\varphi, 1)$. 
	\end{enumerate}
\end{corollary}
\begin{proof} This follows easily from \eqref{eqn:FCsMinv}.\end{proof}

\subsection{Fourier coefficients for $\Sp(4)$} In this subsection we specialize to the case $n=2$ so that $H=\SO(V_{2,3})$. Our first task is to describe a map $\Sp(4)\to H$ in explicit coordinates. 

Recall $y_0 = b_3 + b_{-3}$, $y_1 = b_4 + b_{-4}$, $V_{1,2} = \mathrm{Span}(b_3-b_{-3},b_4,b_{-4})$, and $$V_{2,3} = \Q b_2 + V_{1,2} + \Q b_{-2}.$$
Let $W_4$ be the standard representation of $\Sp(4)$, with symplectic basis $e_1, e_2, f_1,f_2$ and write $V_5$ for the kernel of the contraction map $\wedge^2 W_4 \rightarrow \Q$ given by the symplectic pairing.  The $\Sp(4)$-representation $\wedge^2 W_4$ is identified with $\mathrm{Span}(b_2,b_3,b_4,b_{-4},b_{-3},b_{-2})$ via
\begin{itemize}
	\item $b_2 = e_1 \wedge e_2$
	\item $b_3 = e_1 \wedge f_1$
	\item $b_4 = e_1 \wedge f_2$
	\item $b_{-4} = - e_2 \wedge f_1$
	\item $b_{-3} = e_2 \wedge f_2$
	\item $b_{-2} = - f_1 \wedge f_2$.
\end{itemize}
We put a bilinear form $(u_1,u_2)$ for $u_1,u_2\in \wedge^2 W_4$ by $u_1\wedge u_2=(u_1,u_2)e_1\wedge f_1\wedge e_2\wedge f_2$. Then the above identification of bases respects the quadratic forms on both sides.  Moreover, in this identification, $V_5 = \mathrm{Span}(b_2, b_3 - b_{-3}, b_4, b_{-4},b_{-2})$, which is $V_{2,3}$.

Because $\Sp(4)$ acts on $V_5$ preserving the quadratic form, and because the identification of $V_5$ with $V_{2,3}$ respects the quadratic forms on each, we obtain a map $\pi\colon \Sp(4) \rightarrow \SO(V_{2,3})$.  If $\varphi$ is an automorphic function on $\SO(V_{2,3})$, define $\varphi^*:=\varphi\circ \pi$.

Let
\[K_{\Sp(4)} = \left\{ \left(\begin{array}{cc} A & -B \\ B & A \end{array}\right): A + i B \in U(2)\right\}\]
 denote the standard maximal compact subgroup of $\Sp(4)(\R)$.  If $j_{\Sp(4)}(\gamma,Z)$ is the standard factor of automorphy, then $j_{\Sp(4)}(\,\cdot\,,i 1_2)$ is a character $K_{\Sp(4)} \rightarrow \C^\times$.  We claim that if $k\in K_{\Sp(4)}$ then $j_{\Sp(4)}(k,i 1_2) = j_{\SO(V_{2,3})}(\pi(k), -i y_1)$. Here $j_{\SO(V_{2,3})}$ is the factor of automorphy defined in Section~\ref{subsec:HMFS}. Indeed, to compute $ j_{\SO(V_{2,3})}(\pi(k), -i y_1)$, we simply let $\pi(k)$ act on 
 \[b_2 - i y_1 + b_{-2} = e_1 \wedge e_2 - f_1 \wedge f_2 - i (e_1 \wedge f_2 - e_2 \wedge f_1) = (e_1 - i f_1) \wedge (e_2 - i f_2).\]
 Now the claim is immediately verified.
 
  If $\varphi$ corresponds to a holomorphic modular form, then so does $\varphi^*$, and their Fourier coefficients are related as follows. Suppose $s = \mm{u}{v}{v}{w}$ and set $n(s) = \mm{1}{s}{}{1} \in \Sp(4)$.  Then one computes, under our identification of bases of $V_5$ with $V_{2,3}$, that
\begin{equation}
\label{eqn-subsec-FC1}
n(s) \cdot b_{-2} = \exp( b_2 \wedge (v(b_3-b_{-3}) - ub_4 - w b_{-4})) \cdot b_{-2}.
\end{equation}
Moreover, if $T = \mm{a}{b/2}{b/2}{c}$ then
\begin{equation}
\label{eqn-subsec-FC2}
(T,s) = au + bv + cw = (v(b_3-b_{-3}) - u b_4 - w b_{-4}, -(c b_4 + b b_3 + ab_{-4})).
\end{equation}
From combining (\ref{eqn-subsec-FC1}) and (\ref{eqn-subsec-FC2}) one sees that the $T$-th Fourier coefficient of $\varphi^*$ is the $ -(c b_4 + b b_3 + ab_{-4})$ Fourier coefficient of $\varphi$.  Here we note that elements of 
\begin{equation}
\label{eqn-quotient-lattice}
   \Z\linspan\{b_3,b_4,b_{-4},b_{-3}\}/\langle b_3 + b_{-3}\rangle 
\end{equation}
index the Fourier coefficients of $\varphi$.  And, for $v \in \Z\linspan\{b_3,b_4,b_{-4},b_{-3}\}$, we write the $v$-Fourier coefficient of $\varphi$ to mean the Fourier coefficient of $\varphi$ corresponding to the image of $v$ in the lattice \eqref{eqn-quotient-lattice}.

\subsection{Application of triality to Fourier coefficients}
In this subsection we prove Theorem \ref{thm:introSp4}.  To do so, we will use the following lemma.  Recall that $G' \rightarrow G$ denotes the spin group (cf. Section~\ref{subsec:groupnotation}).  If $\varphi'$ is a weight $\ell$ quaternionic modular form on $G'$, then $\varphi'$ has Fourier coefficients $a_{[T_1,T_2]}(\varphi'): G'(\A_f) \rightarrow \C$ for each $T_1, T_2 \in V_{2,2}$; see Remark \ref{FE-QMFS-on-Spin82}.
\begin{lemma}\label{lem:trialityOnFCs1} Suppose $\ell \geq 1$ is an integer and $\varphi' \in \mathcal{A}(G',\ell)$ is a quaternionic modular form on $G'$ of weight $\ell$.  There is a quaternionic modular form $\varphi_\sigma' \in \mathcal{A}(G',\ell)$ whose Fourier coefficients satisfy $a_{[T_1,T_2]}(\varphi_\sigma',1) = a_{[T_1',T_2']}(\varphi',1)$ where if
	\[
	[T_1,T_2] = [\gamma_1 b_3-\beta_2 b_4 + \delta b_{-4}+\gamma_3 b_{-3}, -\beta_3 b_3 + \alpha b_4 -\gamma_2 b_{-4} - \beta_1 b_{-3}]
	\]
	then
	\[
	[T_1',T_2'] = [\gamma_2 b_3-\beta_3 b_4 + \delta b_{-4}+\gamma_1 b_{-3}, -\beta_1 b_3 + \alpha b_4 -\gamma_3 b_{-4} - \beta_2 b_{-3}].
	\]
	Here $\alpha, \delta, \beta_i,$ and $\gamma_j$ are rational numbers.  Moreover, $\varphi'$ is of level one if and only if $\varphi_\sigma'$ is of level one, and $\varphi'$ is cuspidal if and only if $\varphi_\sigma'$ is cuspidal.
\end{lemma}
\begin{proof} This will follow from Theorem \ref{thm:trialityQMF}.  To use this result, we only need to make explicit how the coordinates used to understand triality in Appendix \ref{sec:triality} are identified with the coordinates $[T_1,T_2] \in V_{2,2}^{\oplus 2}$.
	
This identification is obtained as follows.  Recall the split octonion algebra $\mathbb{O}$ of Section~\ref{subsec:octonions}.  Let $E = \Q \times \Q \times \Q$.  We make use of the Lie algebra isomorphism $\Phi\colon \mathfrak{g}_E  \xrightarrow{\sim} \wedge^2\mathbb{O}$ (see (\ref{triality-compatible-identification})). Applying the identification (\ref{idenitification-of-lists}), $\Phi$ maps the element
\[ y=\alpha' E_{12} + v_1 \otimes (\beta_1',\beta_2',\beta_3') + \delta_3 \otimes(\gamma_1',\gamma_2',\gamma_3') + \delta' E_{23}\in\mathfrak{g}_E\]
to the element 
\[ b_1 \wedge y_1' + b_2 \wedge y_2' = b_1 \wedge( \beta_3' b_3 - \alpha' b_4 + \gamma_2' b_{-4} + \beta_1' b_{-3}) + b_2 \wedge( \gamma_1' b_3 -\beta_2' b_4 + \delta' b_{-4} + \gamma_3' b_{-3}). \]
Likewise, with $T_1, T_2$ as in the statement of the lemma, set
\[
w = \alpha E_{12} + v_1 \otimes \beta + \delta_3 \otimes \gamma + \delta E_{23}.
\]
Then $\Phi(w) = b_1 \wedge T_1 + b_2 \wedge T_2$.  Finally, the commutator $[w,y]$ satisfies
\[
[w,y]  = ((T_1, y_1') + (T_2,y_2')) E_{13}.
\]
The lemma now follows from Theorem~\ref{thm:trialityQMF}.\end{proof}

We are now ready to prove Theorem~\ref{thm:introSp4} as the following corollary.
\begin{corollary} \label{cor:FJLevel1}
Suppose $\varphi$ is a quaternionic modular form of even weight $\ell$ on $G$ and level one.  Then there is a Siegel modular form $f(Z)$ on $\Sp(4)$ of weight $\ell$ and level one, such that for all $a,b,c\in \Z$, one has the identity of Fourier coefficients
	\begin{equation}
    \label{eqn-equality-of-FCs}
	    a_f\left(\left(\begin{array}{cc} a & b/2 \\ b/2 & c \end{array}\right)\right) = \overline{a_{[T_1',T_2']}(\varphi,1)}.
	\end{equation}
	Here $T_1' = -b b_4 + ab_3 + b_{-3}$ and $T_2' = -c b_4 - b_{-4}$. Moreover, if $\varphi$ is cuspidal, then so is $f$. 
\end{corollary}
\begin{proof} Let $\varphi'$ be the pullback of $\varphi$ to the spin group $G'$ and let $\varphi_\sigma'$ be the quaternionic modular form of Lemma~\ref{lem:trialityOnFCs1}.  Let $T_1 = y_0 =b_{3}+b_{-3}$ and $T_2 = -(c b_4 + b b_3 + ab_{-4})$.  Then in the notation of Lemma~\ref{lem:trialityOnFCs1}, $(\alpha, \beta, \gamma,\delta) = (-c,(0,0,b),(1,a,1),0)$ and so $T_1' =-b b_4 + ab_3 + b_{-3}$ and $T_2' = -c b_4 - b_{-4}$.
	
Let $f$ be the holomorphic modular form on $\Sp(4)$ obtained as the Fourier-Jacobi coefficient of $\varphi_\sigma'$.  Then $f$ is of level one since $\varphi_\sigma'$ is of level one.  For $a,b,c \in \Z$, let $a_f([a,b,c])$ denote the $\mm{a}{b/2}{b/2}{c}$ Fourier coefficient of $f$.  Then we have
\[
\overline{a_f([a,b,c])} = a_{[y_0, -(c b_4 + b b_3 + ab_{-4})]}(\varphi_\sigma',1) = a_{[T_1',T_2']}(\varphi,1).
\]
Here the first equality comes from applying \eqref{eqn-subsec-FC2} and Corollary~\ref{cor:FJcoeffLevel1} while the second equality comes from applying Lemma~\ref{lem:trialityOnFCs1}.

If $\varphi$ is cuspidal, then so is $\varphi_\sigma'$, and so $\varphi_\sigma'$ is bounded.  Thus its first Fourier-Jacobi coefficient is a bounded automorphic function on $\Sp(4)$ that corresponds to a holomorphic modular form.  This implies that $f$ is cuspidal, as a holomorphic modular form is cuspidal if and only if its associated automorphic function is bounded. \end{proof}

The pair $[T_1',T_2']$ from Corollary~\ref{cor:FJLevel1} is strongly primitive. Since $S([T_1',T_2']) = \left(\begin{smallmatrix} a & b/2 \\ b/2 & c \end{smallmatrix}\right)$, we obtain the following, which completes the proof of Theorem~\ref{thm:intoSKMS}.

 \begin{corollary}
 \label{cor:MS in SK}
 Suppose $\varphi$ is a cuspidal, quaternionic modular form of even weight $\ell \geq 16$ and level one that is in the Maass Spezialschar.  Then there is a holomorphic cuspidal modular form $f$ on $\Sp(4)$ so that $\varphi = \theta^*(f)$.  In particular, $\varphi$ is in the Saito-Kurokawa subspace.
 \end{corollary}
 \begin{proof} Let $f$ be the holomorphic Siegel modular form of weight $\ell$ and level one on $\Sp(4)$ constructed in Corollary~\ref{cor:FJLevel1}.  By Theorem~\ref{pollack4.1.1}, the Fourier coefficients of $\theta^*(f)$ agree with those of $\varphi$ for $[T_1',T_2']$ of the form $[-b b_4 + ab_3 + b_{-3},-c b_4 - b_{-4}]$.  Because both $\theta^*(f)$ and $\varphi$ are in the quaternionic Maass Spezialschar, it follows from Definition~\ref{definition-Quaternionic-Spezialschar} that the Fourier coefficients of $\varphi$ and $\theta^*(f)$ agree for all $[T_1,T_2]$.  Consequently, $\varphi = \theta^*(f)$ as desired.
 \end{proof}

\section{Hecke stability and theta lifts}\label{sec:Hecke}
Suppose $\varphi$ is a cuspidal quaternionic eigenform on $G$.  The purpose of this section is to give some equivalent conditions for $\varphi$ to be in the Saito-Kurokawa subspace. In particular, we prove Theorem~\ref{thm:TFAE}, which gives a mostly representation theoretic criterion for $\varphi$ to be a quaternionic Saito-Kurokawa lift.  One consequence of this theorem is Corollary~\ref{cor:QStheta}, which is our starting point for the work of Section~\ref{sec:periods} on periods.

The results we prove in this section are mostly derived from what is known about the theta correspondence between symplectic and orthogonal groups.  Because we work with quaternionic modular forms on a special orthogonal group, instead of an orthogonal group, we spend some effort relating representations and automorphic forms on these two groups.  

Throughout this section, we frequently go between automorphic representations and cuspidal automorphic forms that are eigenvectors for the spherical Hecke algebra at primes $p$.  We now explain this relationship. For a finite set of places $S$ of $\Q$, let $\A^S = \prod_{v \notin S}' \Q_v$ denote the adeles away from $S$.
\begin{lemma}\label{lem:eigenHecke} Suppose $H$ is a semisimple group over $\Q$, and $S$ is a finite set of places of $\Q$, containing the archimedean place, and for which $H$ is split at all $p \notin S$.  Let $K_p = H(\Z_p)$ for a model of $H$ over $\Z_p$.  Assume that $\varphi$ is a nonzero cuspidal automorphic form on $H$ which is $K_p$-invariant for all $p \notin S$ and is an eigenvector for the $K_p$-bi-invariant Hecke algebra $\mathcal{H}_p(H(\Q_p))$ for all such $p$. Let $V_{S}$ denote the subspace of the cuspidal automorphic forms on $H$ generated by the action of $H(\A^{S})$ on $\varphi$.  Then $V_{S}$ is an irreducible representation of $H(\A^S)$.
\end{lemma}
\begin{proof}Because $V_S$ is contained in the space of cusp forms, it decomposes into a finite direct sum of irreducible representations of $H(\A^S)$: $V_S = \bigoplus_{j=1}^{N} V_j$, with each $V_j$ an irreducible representation of $H(\A^S)$.  Let $v_0 \in V_S$ denote $\varphi$.  Write $v_0 = \sum_{j=1}^{N}{v_{0,j}}$ so that $v_{0,j} \in V_j$.  Note that each $v_{0,j}$ is nonzero, because $V_S$ is generated by $v_0$.  Let $K_S = \prod_{p \notin S}{K_p}$.  Since the sum is direct and $v_0$ is $K_S$-invariant, each $v_{0,j}$ is as well. It follows that $\dim(V_S^{K_S}) \geq N$.

On the other hand, suppose $v \in V_S$.  We will check that $\int_{K_S} k \cdot v\,dk$ is proportional to $v_0$, and thus $N=1$. Here $\,dk$ is the Haar measure on $K_S$ normalized so that $\mathrm{Vol}(K_S)=1$. 

To see this, by definition of $V_S$, we can write $v = \sum_{i=1}^{L} a_i(h_i \cdot v_0)$ for some positive integer $L$, some $a_i \in \C$ and some $h_i \in H(\A^S)$.  The double coset $K_S h_i K_S$ is a finite disjoint union $\bigsqcup_{j=1}^{L_i} t_{i,j}K_S$ with $t_{i,j} \in K_S h_i$. Thus
\begin{align*}
\int_{K_S} k \cdot (h_i \cdot v_0)\,dk &= \frac{1}{L_i} \sum_{j=1}^{L_i}\int_{K_S} k \cdot (t_{i,j} \cdot v_0)\,dk \\
&=\frac{1}{L_i}\int_{K_S} k \cdot \left(\sum_{j=1}^{L_i} t_{i,j} \cdot v_0\right)\,dk.
\end{align*}
But the sum $\sum_{j=1}^{L_i} t_{i,j} \cdot v_0$ is proportional to $v_0$, because $v_0$ is an eigenvector for the Hecke operator $K_S h_i K_S$.  This proves the lemma.
\end{proof}

\subsection{Restriction from $O(V)$}
Recall that $V$ is our split quadratic space of dimension $8$ with integral lattice $V(\Z)$ spanned by the $b_i$, and $G = \SO(V)$.  Set $G_1 = O(V)$, the orthogonal group of $V$.  We write $K_{1,\ell}=G_1(\Z_\ell)$ for the subgroup of $G_1(\Q_\ell)$ that stabilizes $V(\Z_{\ell})$ and $G_1(\widehat{\Z}) = \prod_{\ell}{G_1(\Z_{\ell})}$.   Set $k_0 \in G_1$ to be the element that fixes $b_i$ for $i \in \{1,2,3,-1,-2,-3\}$ and exchanges $b_4$ with $b_{-4}$.  We will at various times consider $k_0$ in $G_1(\Q_v)$ or $G_1(\Q)$ or $K_{1,\ell}$.  

We begin with a simple lemma.
\begin{lemma}\label{lem:GG1bijection} The natural map
	\[ G(\Q) \backslash G(\A)/G(\widehat{\Z}) \rightarrow G_1(\Q) \backslash G_1(\A)/G_1(\widehat{\Z})\]
is a bijection.
\end{lemma}
\begin{proof} For the surjectivity, if $g_1 \in G_1(\A)$, there exists $\gamma_1 \in G_1(\Q)$ so that $(\gamma_1)_\infty g_{1,\infty} \in G(\R)$, and thus there exists $k \in G_1(\widehat{\Z})$ so that $\gamma_1 g_1 k \in G(\A)$.  For the injectivity, if $g,g' \in G(\A)$, and $g' = \gamma g k$ for some $\gamma \in G_1(\Q)$ and $k \in G_1(\widehat{\Z})$, then comparing determinants at the archimedean place proves $\gamma \in G(\Q)$, and at all the finite places proves $k \in G(\widehat{\Z})$.
\end{proof}

Recall from Definition~\ref{defn:QMFs1} that $\mathcal{A}_{0,\Z}(G,\ell)$ denotes the weight $\ell$ quaternionic cuspidal modular forms that are right invariant under $G(\widehat{\Z})$.  We make an analogous definition of quaternionic modular forms on the orthogonal group $G_1$. Let $K^0$ denote the identity component of our fixed maximal compact subgroup of $G(\R)$; it is also the identity component of a maximal compact subgroup of $G_1(\R)$.  As usual, we let $\mathcal{A}(G_1)$ denote the space of smooth, moderate growth automorphic forms on $G_1$.

Recall from Definition \ref{defn:DiffOpDl} the differential operator $D_{\ell}$ and the notion of quaternionic functions on $G_1(\R)$.  
\begin{definition} For an integer $\ell \geq 1$, let $\mathcal{A}(G_1,\ell)$ denote the space of $\varphi \in \mathcal{A}(G_1) \otimes \mathbf{V}_{\ell}$ satisfying 
\begin{enumerate}
    \item $\varphi(gk) = k^{-1} \cdot \varphi(g)$ for all $k \in K^0$ and
    \item for every $g_f \in G_1(\A_f)$, the function $\varphi_{g_f}:G_1(\R) \rightarrow \mathbf{V}_{\ell}$ defined as $\varphi_{g_f}(g) = \varphi(g_f g)$ is quaternionic in the sense of Definition \ref{defn:DiffOpDl}.
\end{enumerate}
We denote by $\mathcal{A}_0(G_1,\ell)$ the subspace of $\mathcal{A}(G_1,\ell)$ consisting of the cuspidal automorphic forms, and write $\mathcal{A}_{0,\Z}(G_1,\ell)$ for those cuspidal quaternionic modular forms of weight $\ell$ that are right invariant under $G_1(\widehat{\Z})$.
\end{definition}

We have a restriction map $Res: \mathcal{A}_{0,\Z}(G_1,\ell) \rightarrow \mathcal{A}_{0,\Z}(G,\ell)$. We use Lemma \ref{lem:GG1bijection} to define a lifting map $L:\mathcal{A}_{0,\Z}(G,\ell) \rightarrow \mathcal{A}_{0,\Z}(G_1,\ell)$.
\begin{proposition}
\label{proposition-restriction-to-SO(V)}
Regarding the maps $L$ and $Res$, we have the following facts:
\begin{enumerate}
	\item The maps $L$ and $Res$ are inverse isomorphisms.
	\item If $\varphi$ is a Hecke eigenform on $G$, then $L(\varphi)$ is a Hecke eigenform on $G_1$. 
\end{enumerate}   
\end{proposition}
\begin{proof}  Part (1) of the proposition is immediate from Lemma~\ref{lem:GG1bijection}.  For part (2), suppose $T \in \H_{p}(G_1)$, the algebra of smooth, compactly supported functions on $G_1(\Q_p)$ that are bi-invariant under $G_1(\Z_p)$.  Let $T_1$ be the restriction of $T$ to $G(\Q_p)$, so that $T_1 \in \H_p(G)$, the Hecke algebra of $G(\Q_p)$.  Then $T(g) = T_1(g)$ if $g \in G(\Q_p)$ and $T(g) = T_1(gk_0)$ if $\det(g) = -1$. Now, because $L(\varphi)$ is invariant by $k_0$, we have $T L(\varphi) = T_1 L(\varphi)$ (normalizing measures appropriately) so $Res(TL(\varphi)) = T_1 \varphi = \lambda_1 \varphi$.  Applying $L$ gives the desired conclusion.
\end{proof}

Let $T$ denote the diagonal torus of $G$ or $G_1$.  Let $\chi$ be an unramified character of $T(\Q_p)$, $\pi_\chi$ the irreducible unramified subquotient of $\ind_{B(\Q_p)}^{G(\Q_p)}(\delta_B^{1/2} \chi)$, and $\pi_{1,\chi}$ the irreducible unramified subquotient of $\ind_{B_1(\Q_p)}^{G_1(\Q_p)}(\delta_B^{1/2} \chi)$, where $\delta_B$ is the modulus character for the Borel subgroup of $G$. 
 Here $B_1 \simeq B$ is the upper triangular subgroup of $G_1$.  Let the normalized spherical vectors in the inductions be $\phi$ and $\phi_1$.  Observe that the restriction of $\phi_1$ to $G$ is $\phi$. 

\begin{lemma}\label{lem:irred} Let the notation be as above.  Suppose $\chi(\diag(1,1,1,p,p^{-1},1,1,1)) =1$.  Then the restriction to $G(\Q_p)$ of $\pi_{1,\chi}$ is $\pi_\chi$, i.e., $\pi_{1,\chi}$ restricts irreducibly to $G(\Q_p)$.
\end{lemma}
\begin{proof}  For ease of notation, set $t_p = \diag(1,1,1,p,p^{-1},1,1,1)$.  Note that the element $k_0$ stabilizes the upper-triangular subgroup of $G_1$, and that the condition $\chi(t_p) = 1$ implies $\chi(k_0 b k_0^{-1}) = \chi(b)$ for all $b \in B(\Q_p)$.

Let $V_1$ be the space of $\pi_{1,\chi}$ and $V$ the space of $\pi_\chi$.  We write $I_1$ for the induction $\ind_{B_1(\Q_p)}^{G_1(\Q_p)}(\delta_B^{1/2}\chi)$ and $I$ for $\ind_{B(\Q_p)}^{G(\Q_p)}(\delta_B^{1/2}\chi)$.  Thus $V_1$ is the unique unramified subquotient of $I_1$ and $V$ is the unique unramified subquotient of $I$. 

We have a restriction map $I_1 \rightarrow I$, which is $G(\Q_p)$ equivariant. We let $I_1'$ be the subspace of $I_1$ given by those functions $f$ that satisfy $f(k_0g) = f(g)$ for all $g \in G_1(\Q_p)$.  Note that $I_1'$ contains the spherical vector $\phi_1$.  Indeed, the function $g \mapsto \phi_1(k_0 g)$ is right-invariant by $K_{1,p}$ and satisfies 
\[\phi_1(k_0 b g) = \chi(k_0 b k_0^{-1}) \phi_1(k_0 g) = \chi(b) \phi_1(k_0g)\]
for all $b \in B(\Q_p)$.  Thus $\phi_1(k_0 g) = \phi_1(g)$ for all $g \in G(\Q_p)$.   We claim that the restriction map defines an isomorphism of $G(\Q_p)$ representations $I_1' \rightarrow I$.  

To see surjectivity, suppose $f \in I$.  Define a function $f_1$ on $G_1(\Q_p)$ as $f_1(g) = f(g)$ if $g \in G(\Q_p)$, and $f_1(k_0g) = f(g)$ if again $g \in G(\Q_p)$.  One checks that $f_1 \in I_1'$, using that $\chi(t_p) =1$.  For injectivity, if $f_1(g) = 0$ for all $g \in G(\Q_p)$ and some $f_1 \in I_1'$, then $f_1(k_0 g) = 0$ so $f_1=0$.  Consequently, $I_1' \rightarrow I$ is an isomorphism. 
	
Let $U \subseteq I$ be the submodule generated by the spherical vector $\phi$, and let $U_1 \subseteq I_1'$ be the submodule generated by the spherical vector $\phi_1$.  Let $\mathrm{vol}(K_p)$ denote the volume of $K_p$ and $\mathrm{vol}(K_{1,p})$ denote the volume of $K_{1,p}$ for a fixed choice of Haar measure. 
 Define 
 \[p_{K_p}(v) = \frac{1}{\mathrm{vol}(K_{p})} \int_{K_p}{k \cdot v\,dk}\]
 and similarly for $p_{K_{1,p}}$, so that these are projections.  Now set $M = \{u \in U: p_{K_p}(gu) = 0 \,\forall g \in G(\Q_p)\}$ and $M_1 = \{u \in U_1: p_{K_{1,p}}(gu) = 0 \,\forall g \in G_1(\Q_p)\}$.   One has that $M$ is the maximal proper submodule of $U$ and $M_1$ is the maximal proper submodule of $U_1$. We have $V_1 = U_1/M_1$ and $V = U/M$.

If $f \in I$ or $I_1'$, then $p_{K_{1,p}}(g \cdot f)(h) = \mathrm{vol}(K_{1,p})^{-1}\int_{K_{1,p}}{f(hk_1g)\,dk_1}$ and $p_{K_p}(g \cdot f)(h) = \mathrm{vol}(K_{p})^{-1}\int_{K_p}{f(hkg)\,dk}$.  Now observe that if $f \in I_1'$, then $Res(p_{K_{1,p}}(f)) = p_{K_p}(Res(f))$.  Indeed, if $h = b k$, with $k \in K_p$ and $b \in B$, then 
\begin{align*}\mathrm{vol}(K_{1,p}) Res(p_{K_{1,p}}(f))(bk) &= \int_{K_{1,p}}{f(bk k_1)\,dk_1} = \chi(b) \int_{K_{1,p}}{f(k_1)\,dk_1} \\ &= 2\chi(b) \int_{K_p}{f(k)\,dk} = \mathrm{vol}(K_{1,p}) p_{K_p}(Res(f))(bk).\end{align*}

It follows that the restriction map induces a well-defined injection $V_1 = U_1/M_1 \rightarrow V = U/M$.  As $V$ is irreducible, this map is an isomorphism.
\end{proof}

We will now discuss the theta correspondence between $\Sp(4)$ and the orthogonal group $G_1$ of the quadratic space $V$.  If $\phi \in S(V(\A)^2)$ is a Schwartz-Bruhat function on $V(\A)^2$, $g \in G_1(\A)$ and $h \in \Sp(4)(\A)$, we set 
\[\theta_\phi(g,h) = \sum_{\xi \in V(\Q)^2}{(\omega_\psi((g,h))\phi)(\xi)},\]
the associated theta function on $G_1 \times \Sp(4)$.  If $f$ is a cuspidal automorphic form on $\Sp(4)$, we set
\[
\theta_\phi(f)(g) = \int_{[\Sp(4)]}{\theta_\phi(g,h) \overline{f(h)}\,dh},
\]
the theta lift of $f$ to $G_1$. It is an automorphic form on $G_1$.  Likewise, if $\varphi$ is a cuspidal automorphic form on $G_1$, we set
\[
\theta_\phi(\varphi)(h) = \int_{[G_1]}{\overline{\varphi(g)} \theta_\phi(g,h)\,dg},
\]
which is an automorphic form on $\Sp(4)$.  Note that, if $\langle \, , \, \rangle$ denotes the Petersson inner product, normalized to be conjugate linear in the first factor and linear in the second factor, one has the identity
\[
\langle \varphi, \theta_\phi(f)\rangle = \overline{\langle \theta_\phi(\varphi), f \rangle}.
\]

Recall that if $f$ is a holomorphic cuspidal Siegel modular form on $\Sp(4)$ of level one, we set $\theta^*(f)$ to be the quaternionic modular form from Theorem \ref{pollack4.1.1}.

\begin{proposition} For $\ell \geq 16$ even, the Saito-Kurokawa subspace of $\mathcal{A}_{0,\Z}(G,\ell)$ is Hecke stable.
\end{proposition}
\begin{proof} Let $f_1, \ldots, f_N$ be a Hecke eigenbasis of the space of holomorphic cuspidal Siegel modular forms on $\Sp(4)$ of weight $\ell$ and level one.  Let $F_j = L(\theta^*(f_j))$ be the theta lift of $f_j$ to $G_1$, so that $\theta^*(f_j) = Res(F_j)$.  Then each $F_j$ is nonzero, and so the automorphic theta module $\Pi_{F_j}=\Theta_{G_1}(\pi_{f_j}) \neq 0$.  It is proved in \cite[Chapter I, section 3]{rallisHoweDuality} that $\Pi_{F_j}$ consists of cusp forms.  In particular, $\Pi_{F_j}$ is contained in the square-integrable automorphic forms.  From \cite{GanTakeda16}, local Howe duality is known at every place.  The cuspidality of $\Pi_{F_j}$ together with local Howe duality implies $\Pi_{F_j}$ is irreducible; see the proof in \cite[Proposition 3.1]{ganAWSchapter} applied to symplectic orthogonal dual pairs.  Consequently, $F_j$ is a Hecke eigenform.  If the Satake parameters at a prime $p$ of $f_j$ are $\diag(\alpha_j,\beta_j,1,\beta_j^{-1},\alpha_j^{-1})$, it follows from \cite{Rallis1982}, as explained in \cite[Theorem 8.1(c)]{ganIHESnotes}, that the Satake parameters of $F_j$ at $p$ are $\diag(\alpha_j,\beta_j,p,1,1,p^{-1},\beta_j^{-1},\alpha_j^{-1})$. For these particular Satake parameters, i.e., with two $1$'s, the restriction from $G_1$ to $G$ remains irreducible by Lemma~\ref{lem:irred}.  Consequently, the finite part $\Pi_{F_j,f}$ restricts irreducibly to $G$, so $\theta^*(f_j) = Res(F_j)$ is again an eigenform.  Thus $\mathrm{SK}_{\ell}$ is spanned by eigenforms, so is Hecke stable.\end{proof}

We will need Lemma~\ref{lem:RepsAgree} (below) soon.  We set up this lemma now.  Recall from Definition~\ref{defn:QMF2} that for an integer $\ell \geq 4$, we let $\pi_{\ell}^0$ be the quaternionic discrete series representation of $G(\R)^0$ with minimal $K^0$-type $S^{2\ell}(\C^2) \boxtimes \mathbf{1} \boxtimes \mathbf{1}$, as a representation of $\SU(2) \times \SU(2) \times \SO(4)$. 

\begin{definition}\label{defn:QDSOrthog1}
We set $\pi_{\ell}$ to be the induction from $G(\R)^0$ to $G(\R)$ of $\pi_{\ell}^0$. Let $W_{\ell}$ denote the underlying vector space for the $(\g,K^0)$-module $\pi_{\ell,\mathrm{fin}}^0$.  As a $(\g,K)$-module, the space of $K$-finite vectors $\pi_{\ell,\mathrm{fin}}$ in $\pi_{\ell}$ is $W_{\ell}':=W_{\ell} \oplus W_{\ell}$. Fix $\mu \in K \setminus K^0$. If $k \in K^0$, then $k$ acts on $W_{\ell}'$ as $k (w_1, w_2) = (kw_1, \mu k \mu^{-1} w_2)$. The element $\mu$ acts on $W_{\ell}'$ as $\mu(w_1, w_2) = ( w_2, \mu^2 w_1)$. If $X \in \g$, then $X$ acts on $W_{\ell}'$ as $X (w_1, w_2) = (X w_1, \mathrm{Ad}(\mu)(X) w_2)$.

Recall that $k_0\in G_1(\Q)$ denotes the element of determinant $-1$ that swaps $b_4$ and $b_{-4}$, and acts as the identity on $\mathrm{Span}\{b_1, b_2, b_3, b_{-3},b_{-2}, b_{-1}\}$. The representation $\pi_{\ell}$ extends to $G_1(\R)$, because it is isomorphic to its $k_0$-conjugate.  There are two possible extensions, which we denote by $\Pi_{\ell}$ and $\Pi_\ell \otimes \det$.  Let $O(4)^{-}$ be the subgroup of $G_1(\R)$ that acts as the identity on $V^{+}$.  We fix the notation so that $O(4)^{-}$ acts as the identity on the minimal $K$-types of $\Pi_{\ell}$ (and via the determinant map on the minimal $K$-types of $\Pi_{\ell} \otimes \det$).
\end{definition}
\begin{remark} 
    Recall that $K^0\simeq \SO(V^+) \times \SO(V^-)$ and $[K\colon K^0]=2$. The element $\mu\in K\backslash K^0$ may be chosen so that it acts on $\so(V^+)$ by interchanging the factors $\sl_2'$ and $\sl_2^{\mathrm{dist}}$ (see Section~\ref{subsection-compact-subgroups}). Therefore, the $K^0$ representation $\mathbf{V}_{\ell}$ is not isomorphic to its conjugate under the action of $\mu$. Since $\mathbf{V}_{\ell}$ is the minimal K-type of $\pi_{\ell}^0$, it follows that $\pi_{\ell}^0$ is likewise not isomorphic to its conjugate under the action of the $\mu$. Hence, $\pi_{\ell}$ is irreducible as a representation of $G(\R)$.
    \indent 
\end{remark}
\begin{definition}We define $\mathcal{A}_{0,\Z}^{rep}(G,\ell)$ to be the space of $(\g,K^0)$ homomorphisms from $\pi_{\ell,\mathrm{fin}}^0$ to $\mathcal{A}_{0,\Z}(G)$.  We define $\mathcal{A}_{0,\Z}^{rep}(G_1,\ell)$ to be the space of $(\g,K)$ homomorphisms from $\Pi_{\ell,\mathrm{fin}}$ to $\mathcal{A}_{0,\Z}(G_1)$.  

If $\varphi \in \mathcal{A}(G)$, set $\iota(\varphi) \in \mathcal{A}(G)$ as $\iota(\varphi)(g) = \varphi(k_0 g k_0^{-1})$.  We say that $\varphi$ is in the \emph{plus subspace} if $\iota(\varphi) = \varphi$.  We remark that it follows immediately from the definition of the theta lift $\theta^*$ that any $\varphi = \theta^*(f)$ is in the plus subspace.
\end{definition}
\begin{lemma}\label{lem:RepsAgree} If $\varphi \in \mathcal{A}_{0,\Z}^{rep}(G,\ell)$ is in the plus subspace, then $L(\varphi)$ is in $\mathcal{A}_{0,\Z}^{rep}(G_1,\ell)$.
\end{lemma}
\begin{proof}  Suppose $F$ is a $(\g,K^0)$-module homomorphism $W_{\ell} \rightarrow \mathcal{A}_{0,\Z}(G)$.  Let $\phi \in W_{\ell}'$, so that $\phi = (w_1, w_2)$ in the notation above.  Define $F': W_{\ell}' \rightarrow \mathcal{A}(G)$ as $F'(\phi) = F(w_1) + \mu^{-1} F(w_2)$.  One checks that $F'$ is a $(\g,K)$ homomorphism.

Define now $\Psi: \Pi_{\ell} \rightarrow \mathcal{A}(G_1)$ as $\Psi(\phi) = L(F'(\phi))$.  This is $(\g,K)$ equivariant.  Observe that $k_0^{-1} \circ \Psi \circ k_0$ and $\Psi$ are both $(\g,K)$-equivariant maps from $\Pi_{\ell}$ to $\mathcal{A}(G_1)$.  If $\tau$ denotes the minimal $K$-type of $\pi_{\ell}^0$ and $F(\tau)$ lands in the plus space, then one checks that $k_0^{-1} \circ \Psi \circ k_0$ and $\Psi$ agree on $\tau$.  By irreducibility of $\Pi_{\ell}$, they then agree on $\Pi_{\ell}$.  This proves the lemma.
\end{proof}

\subsection{Lifts back to $\Sp(4)$}

We begin with a preliminary result.  Recall the notation $\Pi_{\ell}$ from Definition~\ref{defn:QDSOrthog1}.  

By virtue of our fixed additive character $\psi: \Q \backslash \A \rightarrow \C^\times$, the characters 
\[\chi: N_P(\Q) \backslash N_P(\A) \rightarrow \C^\times\]
can be identified with the $\Q$-vector space $N_P^{ab}(\Q)^\vee$, the linear dual of the abelianization of $N_P(\Q)$.  This vector space is prehomogeneous for the adjoint action of $M_P(\Q)$.  We say that a character $\chi$ is \emph{non-degenerate} if it is in the unique Zariski open orbit, and \emph{degenerate} otherwise.

\begin{lemma}\label{lem:degenFC1cusp} Suppose $\Pi=\Pi_f \otimes \Pi_\infty$ is a cuspidal automorphic representation on $G_1$, with $\Pi_\infty$ a quaternionic discrete series $\Pi_{\ell}$ and $\Pi_f$ the irreducible admissible representation of $G_1(\A_f)$.  Let $\varphi$ be a cusp form in the space of $\Pi$, and $\chi$ a degenerate character of $N_P$, the unipotent radical of the Heisenberg parabolic subgroup.  Then the Fourier coefficient $\varphi_\chi(g) \equiv 0$.\end{lemma}
\begin{proof} Let $\alpha: \Pi \rightarrow \mathcal{A}_0(G_1)$ be the given embedding.  If $\xi \in \mathcal{A}_0(G_1)$, let
\[
\xi_\chi(g) = \int_{[N_P]}\chi^{-1}(n)\xi(ng)\,dn
\]
denote the $\chi$-Fourier coefficient.  Let now $L_\chi$ denote the $\C$-valued linear form on $\mathcal{A}_0(G_1)$ given by $L_\chi(\xi) = \xi_\chi(1)$.  

Suppose $\varphi = \alpha(v_f \otimes v_\infty)$.  Let $\{v_j\}$ be a basis of $\mathbf{V}_\ell \subseteq \Pi_\infty$, and $\{v_j^\vee\}$ the dual basis.  Consider the function $F_\chi: G_1(\R) \rightarrow \mathbf{V}_\ell$ defined as $F_\chi(g) = \sum_j{L_\chi(g \alpha(v_f \otimes v_j)) v_j^\vee}$.  Then $F_\chi(g)$ is a generalized Whittaker function of type $\chi$, and is a bounded function of $g$ because $\alpha$ lands in cusp forms.  Consequently, $F_\chi(g) \equiv 0$ on $G_1(\R)^0$ by \cite[Proposition 10.0.1(3)]{pollackQDS}.
	
Let $\iota = \diag(\mm{}{1}{1}{},1,\mm{}{1}{1}{})$ and $k_0$ be as usual.  Then the elements $1,\iota,k_0, \iota k_0$ are in the standard Levi part of the Heisenberg parabolic subgroup, and meet every connected component of $G_1(\R)$.  If $x$ is one of these representatives, then note that $F_\chi(xg): G_1(\R)^0 \rightarrow \mathbf{V}_\ell$ is again a bounded generalized Whittaker function for a degenerate character (namely, $\chi \cdot x$).  Thus $F_\chi(xg) \equiv 0$.  Consequently, $F_\chi(g)$ is $0$ on all of $G_1(\R)$, so $\varphi_\chi(g)$ must also be identically $0$.	
\end{proof}

From Lemma~\ref{lem:degenFC1cusp}, we now immediately obtain the following proposition.  Recall that we write $V = U + V_{2,n} + U^\vee$.
\begin{proposition}\label{prop:liftIsCusp} Let $\Pi = \Pi_f \otimes \Pi_\infty$ be a cuspidal automorphic representation on $G_1$ with $\Pi_\infty$ a quaternionic discrete series. For a vector $v \in V_{2,n}$, define $N_v = \exp( U \wedge (U^{\perp} \cap v^{\perp}))$ so that $N_v$ is a subgroup of $N_P$. One has the following facts.
\begin{enumerate}
	\item Suppose $\varphi \in \Pi$.  Then the constant term of $\varphi$ over $N_v$ is identically $0$.
	\item The theta lift of $\Pi$ to $\SL(2)$ is $0$.
	\item The theta lift of $\Pi$ to $\Sp(4)$ is cuspidal.
\end{enumerate}
\end{proposition}
\begin{proof} For the first part, the constant term of $\varphi$ over $N_v$ can be expanded in terms of degenerate Fourier coefficients of $\varphi$, so it is $0$ by Lemma~\ref{lem:degenFC1cusp}.  For the second part, a standard calculation with the definition of the Schr\"odinger model of the Weil representation immediately shows that the Fourier coefficients of the theta lifts $\theta_\phi(\varphi)$ to $\SL(2)$ factor through periods of $\varphi$ over subgroups $S_v = \{g \in G_1: gv = v\}$.  These subgroups contain the $N_v$'s, so the periods vanish.  The third part follows from the second by the Rallis tower property \cite[Theorem I.1.1]{rallisHoweDuality}, or directly from the first part by a similar calculation with theta lifts and periods.
\end{proof}

We will shortly use the following direct corollary of results of Yamana \cite{Yamana14}.  Recall that $\mathcal{A}_0(\Sp(4))$ denotes the space of cuspidal automorphic forms on $\Sp(4)$.
\begin{theorem}\label{thm:thetaNE}Suppose $\Pi_1,\Pi_2$ are automorphic cuspidal representations of $G_1$, which are isomorphic (but not necessarily equal in the space of automorphic forms).  
\begin{enumerate}
	\item One has that $\Theta_{\Sp(4)}(\Pi_1) \neq 0$ if and only if $\Theta_{\Sp(4)}(\Pi_2) \neq 0$. 
	\item Suppose $\Theta_{\Sp(4)}(\Pi)$ is contained in the space of cuspidal automorphic forms on $\Sp(4)$ for every automorphic representation $\Pi$ isomorphic to $\Pi_1$.  Let $\mathcal{A}_0(G_1)[\Pi_1]$ denote the $\Pi_1$-isotypic subspace of the cusp forms on $G_1$.  If there exists $f \in \mathcal{A}_0(\Sp(4))$ and test data $\phi$ so that $\theta_\phi(f) \in \mathcal{A}_0(G_1)[\Pi_1]$ is nonzero, then every element of this space is in the image of the theta lift from $\Sp(4)$.
\end{enumerate}	
   \end{theorem}
\begin{proof} For the first statement, note that Yamana proves that these theta lifts are nonzero precisely if the local theta lifts $\Theta(\Pi_{i,v})$ are nonzero for every place $v$ of $\Q$, and the $L$-function $L(\Pi_i,\mathrm{Std},s)$ has a pole at $s=2$.  These latter conditions only depend on the isomorphism type of the $\Pi_i$.
	
For the second statement, let $B \subseteq \mathcal{A}_0(G_1)[\Pi_1]$ be the space of the theta lifts from $\Sp(4)$ and let $C$ be its orthogonal complement in $\mathcal{A}_0(G_1)[\Pi_1]$.  Then $B$ is a $G_1(\A)$-representation, and thus so is $C$. Suppose for the sake of contradiction that $C$ is nonzero. Let $\Pi_2 \subseteq C$.  By the first part, $\Theta_{\Sp(4)}(\Pi_2) \neq 0$.  Let $f' \in \Theta_{\Sp(4)}(\Pi_2)$ be a nonzero cusp form in this space.  Then $f' =\theta_\phi(\varphi_2)$ for some $\varphi_2 \in C$.  One has
\[\langle \varphi_2, \theta_\phi(f') \rangle  = \overline{\langle \theta_\phi(\varphi_2),f' \rangle} \neq 0. \]
Moreover, it follows from \cite[Proposition 3.1]{ganAWSchapter} and local Howe duality \cite{GanTakeda16} that $\theta_\phi(f') \in \mathcal{A}_0(G_1)[\Pi_1]$.  This contradicts the statement that $\Pi_2$ is orthogonal to $B$, proving the claim. 
\end{proof}

We now have the following result, which follows from Theorem~\ref{thm:thetaNE}.
\begin{theorem}\label{thm:TFAE} Suppose $\varphi \in \mathcal{A}_{0,\Z}^{rep}(G,\ell)$ is a nonzero Hecke eigenform in the plus subspace. The following conditions are equivalent:
\begin{enumerate}
	\item $\varphi \in \mathrm{SK}_{\ell}$, i.e., there exists a nonzero holomorphic Siegel modular form $f$ of weight $\ell$ and level one so that $\varphi = \theta^*(f)$ is in $\mathcal{A}_{0,\Z}(G,\ell)$;
	\item $\Theta_{\Sp(4)}(\Pi_{L(\varphi)}) \neq 0$;
\end{enumerate}
\end{theorem}
\begin{proof} It is clear that (1) implies (2): Indeed, if $\varphi = \theta^*(f)$, then there is test data $\phi$ so that $L(\varphi) = \theta_\phi(f) \in \mathcal{A}_0(G_1)$ is nonzero. Since $\langle \theta_\phi(f), \theta_\phi(f) \rangle \neq 0$, we see that 
\[\int_{[G_1] \times [\Sp(4)]}{\overline{\theta_\phi(f)(g)} \theta_\phi(g,h)\overline{f(h)}\,dg\,dh} \neq 0,\]
so $\Theta_{\Sp(4)}(\Pi_{L(\varphi)}) \neq 0$.

We now prove that (2) implies (1).  Let $\pi'$ be the irreducible cuspidal automorphic representation of $\Sp(4)$ which is $\Theta_{\Sp(4)}(\Pi_{L(\varphi)})$.  It is cuspidal by Proposition~\ref{prop:liftIsCusp} and irreducible by global Howe duality \cite{GanTakeda16}, see also \cite[Proposition 3.1]{ganAWSchapter}.  By \cite[Proposition 3.1]{ganAWSchapter}, $\pi'$ is generated by a holomorphic Siegel modular form $f'$ of weight $\ell$ and level one.  It follows that $\Theta_{G_1}(\pi')$ has nonzero inner product with some element of $\Pi_{L(\varphi)}$, and that $\Theta_{G_1}(\pi')$ is isomorphic to $\Pi_{L(\varphi)}$.  By Theorem~\ref{thm:thetaNE}, there exists a cuspidal automorphic representation $\pi$ on $\Sp(4)$ of level one for which $\Theta_{G_1}(\pi) = \Pi_{L(\varphi)}$.  Indeed, there exists a cuspidal automorphic representation $\pi_1$ on $\Sp(4)$ and $f_1 \in \pi_1$ so that $\theta_\phi(f_1) \in \Pi_{L(\varphi)}$.  Because $\Pi_{L(\varphi)}$ is irreducible, $\Theta_{G_1}(\pi_1) \supseteq \Pi_{L(\varphi)}$.  But $\Theta_{G_1}(\pi_1)$ itself is irreducible, so $\Theta_{G_1}(\pi_1) = \Pi_{L(\varphi)}$.  One sees that $\pi_1$ must be unramified at every finite place and holomorphic discrete series at infinity, so we can take $\pi = \pi_1$.

Finally, let $f \in \pi$ be the Siegel modular form of level one.  Then $\theta^*(f)$ is nonzero, but also in $\Pi_{\varphi}$, so we must have $\theta^*(f) = \varphi$ as desired.
\end{proof}

For a half-integral $2\times 2$ symmetric matrix $S$, let $P_{S,\ell}$ be the weight $\ell$ holomorphic Poincar\'e series on $\Sp(4)$ associated to $S$, and let $Q_{S,\ell}$ be the associated Poincar\'e lift; see Proposition~\ref{Proposition-Explicit-Poincare-Lift}.
\begin{corollary}\label{cor:QStheta}Suppose $\varphi \in \mathcal{A}_{0,\Z}^{rep}(G,\ell)$ is an eigenform in the plus subspace, for $\ell \geq 16$ even.  Then $\varphi \in \mathrm{SK}_{\ell}$ if and only if $\langle \varphi, Q_{S,\ell} \rangle \neq 0$ for some $S$.
\end{corollary}
\begin{proof}  If $\varphi \in \mathrm{SK}_{\ell}$, then clearly such an $S$ exists: We have $\varphi = \theta^*(f)$, and $f = \sum_{j}{\alpha_j P_{S_j,\ell}}$, so that $\varphi = \sum_{j}{\beta_j Q_{S_j,\ell}}$, and hence $\langle \varphi, Q_{S_j,\ell} \rangle \neq 0$ for some $j$.

Conversely, assume $\langle \varphi, Q_{S,\ell} \rangle \neq 0$ for some $S$.  This inner product is equal to $\langle L(\varphi),L(Q_{S,\ell}) \rangle_{G_1}$, and $L(Q_{S,\ell})$ is a theta lift from $\Sp(4)$.  Thus $\Theta_{\Sp(4)}(\Pi_{L(\varphi)}) \neq 0$, so the corollary follows from Theorem~\ref{thm:TFAE}.\end{proof}

In the next section, we will study the inner products  $\langle \varphi, Q_{S,\ell} \rangle$ using periods of $\varphi$.

\section{Periods}\label{sec:periods}
The purpose of this section is to prove Theorem~\ref{thm:introPeriod}, restated below as Corollary~\ref{cor:part 6 of main thm}.  By Corollary~\ref{cor:QStheta}, we can characterize the elements of the Saito-Kurokawa subspace in terms of inner products with the $Q_{S,\ell}$.  Thus to prove Theorem~\ref{thm:introPeriod} it remains to relate the inner products $\langle \varphi,Q_{S,\ell} \rangle$ to periods of $\varphi$.  That is what we do in this section.

Recall from Section~\ref{sec:notation} that we let $V$ denote the underlying quadratic space for the group $G$. Suppose $v_1, v_2 \in V$.  We write $H_{v_1,v_2}$ for the subgroup of $G$ fixing the subspace spanned by $v_1$ and $v_2$.  If $v_1,v_2\in V$ are non-isotropic vectors, and $Y_{v_1,v_2}$ denotes the orthogonal complement of $\Q\linspan\{v_1,v_2\}$ in $V$, then the map $h\mapsto h\rvert_{Y_{v_1,v_2}}$ gives an isomorphism 
\begin{equation}
\label{eqn-isomorphism-type-H_{v1,v2}}
H_{v_1,v_2}\simeq \SO(Y_{v_1,v_2}). 
\end{equation}
Moreover, if $v_1,v_2 \in V(\Q)$ or $V(\R)$, then $H_{v_1,v_2}$ is an algebraic group over $\Q$ or $\R$ respectively.  Finally, if $v_1, v_2 \in L = V(\Z)$, then we write $H_{v_1,v_2}(\Z)$ for the subgroup of $G(\Z)$ that fixes $v_1,v_2$.

Given $v_1, v_2 \in V(\R)$ spanning a positive-definite two-plane, recall from \eqref{eqn:Bv1v2Def} the function $B_{[v_1,v_2]}: G(\R) \rightarrow \mathbf{V}_{\ell}$.  Recall that we let $V^+ \subseteq V(\R)$ denote a certain positive-definite subspace of dimension four; see Section~\ref{subsection-compact-subgroups}.  That the denominator in the definition of $B_{v_1, v_2}$ is nonzero follows from the following lemma, and the fact that if $v_1, v_2$ span a positive definite two-plane, then the projections of $v_1$ and $v_2$ onto $V^+$ still span a two-plane.
\begin{lemma}\label{lem:proju} Suppose $u_1, u_2 \in V^{+}$ span a two-plane.  Then $\mathrm{pr}_{K}(u_1 \wedge u_2)$ is nonzero.
\end{lemma}
\begin{proof}  This is clear for the basis elements $u_1, u_2$ of Section~\ref{The underlying quadratic space}.  It now follows in general by $\SO(V^+)$ equivariance of the projection map.
\end{proof}

Suppose now that $T$ is a fixed half-integral, positive definite, symmetric matrix.  Set $X_{T} = \{v_1,v_2 \in L: S(v_1, v_2) = T\}.$  If $\ell \geq 16$ is even, recall we set 
\[Q_{T,\ell}(g) = \sum_{(v_1,v_2) \in X_{T}}{B_{[v_1,v_2],\ell}(g)}.\]
Similarly, if $v_1, v_2 \in L$ span a positive-definite two-plane, and $\ell \geq 16$ is even, we set
\[Q_{v_1,v_2;\ell}(g) = \sum_{\gamma \in H_{v_1,v_2}(\Z)\backslash G(\Z)}{B_{[v_1,v_2],\ell}(\gamma g)}.\]

Suppose $D = -4\det(T)$ is odd and square-free.  By \cite{Saha2013}, the Poincar\'e series $P_T$ with these $T$ span the space of cusp forms of weight $\ell$ and level one on $\Sp(4)$. By Theorem~\ref{thm:AITSO8}, for such $T$, $Q_{T,\ell} = Q_{v_1,v_2;\ell}$ for any $v_1, v_2 \in L$ with $S(v_1, v_2) = T$.  Thus we have:
\begin{corollary} Suppose $\ell \geq 16$ is even, and $\varphi$ is a cuspidal quaternionic modular form on $G$ of weight $\ell$ and level one. Suppose moreover that $\varphi$ is an eigenform in the plus subspace. Then $\varphi \in \mathrm{SK}_{\ell}$ if and only if $\langle \varphi, Q_{v_1,v_2;\ell}\rangle \neq 0$ for some $v_1, v_2 \in L$ with $-4\det(S(v_1,v_2))$ odd and square-free.\end{corollary}

Our objective for the rest of this section is to reinterpret the inner product $\langle \varphi, Q_{v_1,v_2;\ell}\rangle$ as a period of $\varphi$ over $H_{v_1,v_2}$.  The inner product $\langle \varphi, Q_{v_1,v_2;\ell}\rangle$ is defined adelically. It can be interpreted as an integral over the real points of $G$, by the following lemma.

\begin{lemma} The canonical map
	\[G(\Z)\backslash G(\R) \rightarrow G(\Q) \backslash G(\A)/ G(\widehat{\mathbf Z})\]
is a bijection.
\end{lemma}
\begin{proof} The map is clearly an injection.  For the surjectivity, we must prove that $G(\A_f) = G(\Q) G(\widehat{\Z})$.  Thus suppose $g \in G(\A_f)$.  By the Iwasawa decomposition for the Siegel parabolic subgroup $NM$ of $G$, we can write $g = u m k$ with $u \in N(\A_f)$, $m \in M(\A_f) \simeq \GL(4,\A_f)$ and $k \in G(\widehat{\Z})$.  We can write $m = \gamma k_1$ for some $\gamma \in M(\Q) \subseteq G(\Q)$ and $k_1 \in G(\widehat{\Z})$.  Thus $g = \gamma (\gamma^{-1} u \gamma) k_1 k$.  But now $\gamma^{-1} u \gamma = \mu k_2$ for some $\mu \in N(\Q)$ and $k_2 \in G(\widehat{\Z})$.  The lemma follows.
\end{proof}

To set up our result on the inner product $\langle \varphi, Q_{v_1, v_2;\ell}\rangle$, we need an additional lemma.
\begin{lemma}\label{lem:Cartan} Suppose $v_1, v_2 \in V(\R)$ span a positive-definite two-plane.  Then the image of $K \cap H_{v_1,v_2}(\R)$ in the long root $\SU_2/\mu_2$ is a nontrivial torus.
\end{lemma}
\begin{proof}  Suppose $u_1, u_2$ span the orthogonal complement in $V^+$ of the projection of $v_1, v_2$ to $V^+$.  Then the projection of $K \cap H_{v_1,v_2}(\R)$ in the long root $\SU(2)/\mu_2$ is the projection of the  $\exp(t u_1 \wedge u_2)$ with $t \in \R$.  Now the lemma follows by Lemma~\ref{lem:proju}.
\end{proof}

We write $K_{v_1,v_2} = K \cap H_{v_1,v_2}(\R)$.  It follows from Lemma~\ref{lem:Cartan} that $K_{v_1,v_2}$ stabilizes a unique line in $\mathbf{V}_{\ell}$.  Suppose $v' \in \mathbf{V}_{\ell}$ spans this line.  Observe that $B_{[v_1,v_2]}(1)$ is in the same line, so that $\langle v', B_{[v_1,v_2]}(1) \rangle \neq 0$.

If $v_1, v_2 \in L$, and $\varphi$ is a cuspidal quaternionic modular form on $G$ of level one, we define
\[P_{v_1,v_2}(\varphi) = \frac{1}{\langle v', B_{[v_1,v_2]}(1)\rangle} \int_{H_{v_1,v_2}(\Z)\backslash H_{v_1,v_2}(\R)}{\langle v', \varphi(h)\rangle \,dh}.\]
In terms of the vector-valued period $\mathcal{P}$ defined in the introduction, note that $P_{v_1,v_2}(\varphi) = \frac{ \langle v',\mathcal{P}_{v_1,v_2}(\varphi)\rangle}{\langle v', B_{[v_1,v_2]}(1) \rangle}$.  Observe that $\mathcal{P}_{v_1,v_2}(\varphi)$ is invariant by $K_{v_1,v_2}$, so by Lemma~\ref{lem:Cartan}, $P_{v_1,v_2}(\varphi) \neq 0$ if and only if $\mathcal{P}_{v_1,v_2}(\varphi) \neq 0$.

\begin{theorem}\label{thm:PeriodQ} Suppose $\ell \geq 22$.  There is a nonzero constant $C_{v_1,v_2}$ so that $\langle \varphi, Q_{v_1,v_2} \rangle = C_{v_1,v_2} P_{v_1,v_2}(\varphi)$ for all cuspidal quaternionic modular forms $\varphi$ of weight $\ell$ and level one.
\end{theorem}

The condition on the weight $\ell \geq 22$ in the statement of Theorem~\ref{thm:PeriodQ} comes from a convergence criterion studied in Appendix~\ref{sec:finiteness_appendix}; it can likely be weakened.  Before proving Theorem~\ref{thm:PeriodQ}, we require one more lemma.  

Say that $v \in \mathbf{V}_{\ell}$ is \emph{completely degenerate} if $v$ is in the $\SL(2,\C)$ orbit of a highest weight line of $\mathbf{V}_{\ell}$. Additionally, say that a smooth, $K$-equivariant function $b\colon G(\R) \to \mathbf{V}_{\ell}$ is \textit{quaternionic} if $D_{\ell}b\equiv 0$. Here $D_{\ell}$ is the differential operator appearing in Definition~\ref{defn:QMFs1}. 
\begin{lemma}\label{lem:aF} Suppose $F: G(\R) \rightarrow \mathbf{V}_\ell$ is a quaternionic function so that $F(g)$ is never completely degenerate.  Suppose also that $a: G(\R) \rightarrow \C$ is a smooth function with $a F$ quaternionic.  Then $a$ is constant.
\end{lemma}
\begin{remark}\label{rem:compDeg} Suppose $v \in \mathbf{V}_{\ell}$.  Consider the condition:
\begin{itemize} 
	\item $\langle v, u \rangle = 0$ with $u \in V_2$ implies $u = 0$, where $\langle v, u \rangle \in S^{2\ell-1}(V_2)$ is the contraction.  \end{itemize}
The element $v$ satisfies this condition precisely when $v$ is not completely degenerate.  Indeed, by acting by $\SL(2,\C)$, it suffices to consider the case $u =x$.  Then $\langle v, x \rangle = 0$ implies $v$ is in $\C x^{2\ell}$.   We also remark that if $\langle v, v \rangle \neq 0$, then $v$ is not completely degenerate.
\end{remark}

\begin{proof}[Proof of Lemma~\ref{lem:aF}] First observe that, because both $a F$ and $F$ are $K$-equivariant, and $F$ is never $0$, we obtain that $a$ is $K$-invariant.  Now, let $\p = V_2 \otimes W$, $\{w_\alpha\}$ be a basis of $W$, and $\{w_\alpha^\vee\}$ the dual basis of $W^\vee$.  Then
\[D(aF) = a D(F) + \sum_{\alpha}{ (x \otimes w_\alpha)(a) \langle F,y \rangle w_\alpha^\vee - (y \otimes w_\alpha)(a) \langle F,x \rangle w_\alpha^\vee}.\]
We have $D(F) = 0$, and the $w_\alpha^\vee$ are linearly independent.  Thus we obtain 
\[(x \otimes w_\alpha)(a) \langle F,y \rangle w_\alpha^\vee - (y \otimes w_\alpha)(a) \langle F,x \rangle w_\alpha^\vee = 0 \]
for all $\alpha$.  But by Remark~\ref{rem:compDeg}, $\langle F, y\rangle$ and $\langle F,x \rangle$ are linearly independent in $S^{2\ell-1}(V_2)$.  Consequently $(x\otimes w_\alpha)(a) = 0$ and $(y\otimes w_\alpha)(a) =0 $ for all $\alpha$.  In other words, $Xa = 0$ for all $X \in \p$.  One concludes that $a$ is constant.
\end{proof}

\begin{proof}[Proof of Theorem~\ref{thm:PeriodQ}]
Set
\[B_\varphi(g) = \int_{H_{v_1,v_2}(\Z)\backslash H_{v_1,v_2}(\R)}{\varphi(hg)\,dh}.\]
Then $B_\varphi(g)$ is quaternionic and left $H_{v_1,v_2}(\R)$-invariant.  Indeed, to see that $B_\varphi(g)$ is quaternionic, we simply need to justify differentiation under the integral over the domain $H_{v_1,v_2}(\Z)\backslash H_{v_1,v_2}(\R).$  To do this, fix $X \in \p$.  Then we are interested in proving the equality
\[ \frac{d}{dt} \int_{H_{v_1,v_2}(\Z)\backslash H_{v_1,v_2}(\R)}{ \varphi(hg e^{tX})\,dh} = \int_{H_{v_1,v_2}(\Z)\backslash H_{v_1,v_2}(\R)}{\frac{d}{dt} \varphi(hg e^{tX})\,dh}\]
at $t=0$. One now justifies the exchange simply by the boundedness of cusp forms.

Now, we claim that there is a constant $D_1$ so that $B_\varphi(g) = D_1 B_{[v_1,v_2]}(g)$.  Indeed, if $k \in K \cap g^{-1} H_{v_1,v_2}(\R) g = K_{g^{-1}v_1,g^{-1}v_2}$, then
\[k^{-1} \cdot B_\varphi(g) = B_\varphi(gk) = B_\varphi(g)\]
because $B_\varphi$ is left $H_{v_1,v_2}(\R)$-invariant.  Here $k^{-1} \cdot B_\varphi(g)$ denotes the action of $k^{-1} \in K^0$ on $B_{\varphi}(g) \in \mathbf{V}_{\ell}$.  Thus $B_\varphi(g)$ and $B_{[v_1,v_2]}(g)$ lie on the same line in $\mathbf{V}_{\ell}$ for every $g \in G(\R)$.  Because $\langle B_{[v_1,v_2]}(g), B_{[v_1,v_2]}(g) \rangle \neq 0$ for all $g$, we can apply Lemma~\ref{lem:aF} to deduce that $B_\varphi(g) = D_1B_{[v_1,v_2]}(g)$ for some constant $D_1$.

Now, the constant $D_1$ is determined by evaluating both sides at $g=1$ and pairing with $v'$, so we obtain $D_1 = P_{v_1,v_2}(\varphi)$.

Now one has
\begin{align*}\langle \varphi, Q_{v_1,v_2} \rangle &= \int_{H_{v_1,v_2}(\Z)\backslash G(\R)}{ \langle B_{[v_1,v_2]}(g), \varphi(g) \rangle \,dg} \\ &= \int_{H_{v_1,v_2}(\R)\backslash G(\R)}{\langle B_{[v_1,v_2]}(g), B_\varphi(g) \rangle \,dg} \\ &= P_{v_1,v_2}(\varphi) \int_{H_{v_1,v_2}(\R)\backslash G(\R)}{\langle B_{[v_1,v_2]}(g), B_{[v_1,v_2]}(g) \rangle \,dg}.\end{align*}

To justify the unfolding of the integral, we must prove that the first integral converges absolutely.  We have 
\[|\langle B_{[v_1,v_2]}(g), \varphi(g)\rangle|  \leq || B_{[v_1,v_2]}(g)|| \cdot ||\varphi(g)|| \leq C || B_{[v_1,v_2]}(g)||.\]
Thus
\begin{align*} \int_{H_{v_1,v_2}(\Z)\backslash G(\R)}{ |\langle B_{[v_1,v_2]}(g), \varphi(g)\rangle |\,dg} &\leq C \int_{H_{v_1,v_2}(\Z)\backslash G(\R)}{ ||B_{[v_1,v_2]}(g)||\,dg} \\ &= C' \int_{H_{v_1,v_2}(\R)\backslash G(\R)}{ ||B_{[v_1,v_2]}(g)||\,dg}\end{align*}
Since $H_{v_1,v_2}(\R)$ is semisimple (see \eqref{eqn-isomorphism-type-H_{v1,v2}}), $H_{v_1,v_2}(\Z)\backslash H_{v_1,v_2}(\R)$ has finite volume \cite[Theorem~7.8]{MR147566} and $B_{[v_1,v_2]}(g)$ is left-invariant by $H_{v_1,v_2}(\R)$.  Now the result, for $\ell \geq 22$, follows from Theorem~\ref{thm:finitenessInt}.
\end{proof}

Putting everything together, we have obtained the following result.
\begin{corollary} 
\label{cor:part 6 of main thm}
Suppose $\ell \geq 22$ is even.  Suppose $\varphi$ is a cuspidal quaternionic eigenform on $G$ of weight $\ell$ and level one in the plus subspace.  Then $\varphi$ is a Saito-Kurokawa lift if and only if there exists $v_1, v_2 \in L$ with $S(v_1,v_2) > 0$ and $D:=-4\det(S(v_1,v_2))$ odd and square-free, so that the period $P_{v_1,v_2}(\varphi) \neq 0$.
\end{corollary}

\appendix
\section{Triality}\label{sec:triality}
The purpose of this section is to work out facts regarding triality on $D_4$, and how it interacts with the notion of quaternionic modular forms. This is an important ingredient in the proof of Theorem~\ref{thm:introSp4}.

To define triality on $D_4$, recall the trilinear form $(x,y,z) = \tr_\O(x(yz))$ on the octonions $\O$.  See Section~\ref{subsec:octonions} for notation regarding the octonions.  The group $G'=\Spin(\O)$ is defined as the set of triples $(g_1, g_2, g_3) \in \SO(\O)^3$ that satisfy 
\[(g_1 x, g_2y, g_3 z) = (x,y,z) \text{ for all } x,y,z \in \O.\]
The association $g \mapsto g_1$ gives a map of groups $G' \rightarrow G$, which induces an isomorphism on Lie algebras.

Permutation of the $g_j$ induces an $S_3$ action on $G'$ as follows.  For $g \in \SO(\O)$, let $g^* = c \circ g \circ c$, where $c: \O \rightarrow \O$ is the conjugation on $\mathbb{O}$.  If $\sigma \in S_3$ has $\mathrm{sgn}(\sigma) =1$, we define $\sigma (g_1, g_2, g_3) = (g_{\sigma^{-1}(1)},g_{\sigma^{-1}(2)},g_{\sigma^{-1}(3)})$.  For $\sigma \in S_3$ with $\mathrm{sgn}(\sigma) =-1$, we define $\sigma (g_1, g_2, g_3) = (g_{\sigma^{-1}(1)}^*,g_{\sigma^{-1}(2)}^*,g_{\sigma^{-1}(3)}^*)$.  One verifies easily that this maps $G' \rightarrow G'$ and is an $S_3$ action.

In this section, we understand how this triality interacts with the notion of quaternionic modular forms on $G'$.  Let $P'$  be the Heisenberg parabolic subgroup of $G'$, defined as the inverse image of $P$ in $G'$, and similarly define $M',N', K'$, $\p'$.  Let $W$ denote the set of automorphic characters of $N'$, i.e., the set of characters $N'(\Q)\backslash N'(\A) \rightarrow \C^\times$.  
\begin{theorem}\label{thm:trialityQMF} Suppose $\sigma \in S_3$.
\begin{enumerate}
	\item One has $\sigma(M') = M'$, $\sigma(N') = N'$, $\sigma(K') = K'$, $\sigma(\p') = \p'$.  Moreover, if $v' \in \mathbf{V}_{\ell}$, and $k \in K'$, then $\sigma(k) \cdot v' = k \cdot v'$.
	\item The action of $S_3$ on $N'$ preserves the center $Z'$ of $N'$, and induces an action on the set of automorphic characters $W$.
	\item If $\varphi$ is a weight $\ell$ quaternionic modular form on $G'$, then $\varphi_\sigma(g) :=\varphi(\sigma^{-1}(g))$ is a weight $\ell$ quaternionic modular form on $G'$ and the Fourier coefficients $a_{\varphi_\sigma}$ satisfy $a_{\varphi_\sigma}(w)(g_f) = a_\varphi(\sigma^{-1}(w))(\sigma^{-1}(g_f))$ for all $w \in W$.
	\end{enumerate}
\end{theorem}

To prove the theorem, we will make explicit the action of triality on $\g' = \Lie(G') \simeq \Lie(G) \simeq \wedge^2 V$.

The Lie algebra $\g' = \Lie(G')$ of $G'$ is the set of triples $(X_1,X_2,X_3) \in \Lie(\SO(\O))^3$ that satisfy 
\[(X_1 x, y,z) + (x,X_2 y,z) + (x,y,X_3z) = 0\]
for all $x,y,z \in \O$.  We denote by $\g = \Lie(\SO(\O))$ the Lie algebra of $G = \SO(\O)$.

We have the following well-known lemma.
\begin{lemma} For $j=1,2,3$, the map $g \mapsto g_j$ induces an isomorphism of Lie algebras $\g' \rightarrow \g$. \end{lemma}

For $x \in \O$, denote $\ell_x: \O \rightarrow \O$ left multiplication by $x$ and $r_x: \O \rightarrow \O$ right multiplication by $x$.

\begin{proposition} If $u,v \in \O$, then
	\[(u\wedge v, \frac{1}{2}(\ell_{u^*}\ell_v - \ell_{v^*}\ell_u), \frac{1}{2}(r_{u^*}r_v-r_{v^*}r_u))\]
is a triality triple, i.e., in the Lie algebra of $\Spin(\O)$.
\end{proposition}
\begin{proof}This follows from \cite[Theorem 3.5.5]{springerVeldkamp}.
\end{proof}

We would like to explicitly calculate the triality triples for a basis of elements of $\wedge^2 \mathbb{O}$.  This can be done using the following lemmas.  Let $V_7 \subseteq \O$ denote the space of trace $0$ elements.  The Lie algebra of $G_2$ embeds in $\wedge^2 \O$ as the kernel of the map $\wedge^2 V_7 \rightarrow V_7$ given by $u\wedge v \mapsto \mathrm{Im}(uv)$; see \cite[section 2.2]{pollackG2}.  If $X \in \mathrm{Lie}(G_2)$, then $(X,X,X)$ is a triality triple.  This gives the following lemma.
\begin{lemma}\label{lem:G2Lie1} Suppose $u,v \in V_7$ and $\mathrm{Im}(uv) = 0$.  Then $(u\wedge v, u\wedge v, u\wedge v)$ is a triality triple.
\end{lemma}

We can bootstrap off of Lemma~\ref{lem:G2Lie1} to obtain the following.
\begin{lemma}\label{lem:trialityIsot} Suppose $W_1, W_2, W_3$ are two-dimensional isotropic subspaces of $\O$ and $W_i \cdot W_{i+1} = 0$.  Let $u_j, v_j$ be a basis of $W_j$.  Then there are nonzero $\alpha_1, \alpha_2, \alpha_3$ so that $(\alpha_1 u_1 \wedge v_1, \alpha_2 u_2 \wedge v_2, \alpha_3 u_3 \wedge v_3)$ is in $\Lie(\Spin(\O))$.
\end{lemma}
\begin{proof} First note that if $(X,Y,Z) \in \Lie(\Spin(\O))$ and $(g_1,g_2,g_3) \in \Spin(\O)$ then 
	\[(g_1X, g_2Y,g_3Z) \in \Lie(\Spin(\O)).\]
 This follows immediately from the definitions.  Let $u,v$ be as in Lemma~\ref{lem:G2Lie1}; then we know $(u\wedge v, u \wedge v, u \wedge v)$ is a triality triple.  Now, if $W_1, W_2, W_3$ are as in the statement of this lemma, there exists $(g_1,g_2,g_3) \in \Spin(\O)$ so that $g_i \Span\{u,v\} = W_i$; this follows from, for example, Lemma 2.2(3) and Lemma 2.3(3) of \cite{PWZ}.  The lemma follows.
\end{proof}

Observe that if $(X,Y,Z)$ is a triality triple, then $X+Y+Z \in \Lie(G_2)$.  Because we already know how the Lie algebra of $G_2$ embeds in $\wedge^2 \O$, we can frequently use this fact together with Lemma~\ref{lem:trialityIsot} to compute many triality triples.  We do this now.

\begin{proposition} We have the following triality triples:
\begin{enumerate}
	\item $(\epsilon_1 \wedge e_{j}, e_{j+1}^* \wedge e_{j-1}^*, - \epsilon_2 \wedge e_j)$ for $j \in \Z/3\Z$;
	\item $(\epsilon_1 \wedge e_{j}^*, - \epsilon_2 \wedge e_j^*, e_{j+1} \wedge e_{j-1})$ for $j \in \Z/3\Z$.
\end{enumerate}
\end{proposition}
\begin{proof} One computes that $W_1 = \Span\{\epsilon_1, e_1\}, W_2 = \Span\{e_2^*,e_3^*\}$, $W_3 = \Span\{\epsilon_2,e_1\}$ satisfy $W_i \cdot W_{i+1} = 0$.  Because $(\epsilon_1-\epsilon_2) \wedge e_1 + e_2^* \wedge e_3^*$ is in $\Lie(G_2)$, the first point follows in the case $j=1$.  The other cases follow from the $j=1$ case by applying the $\SL(3) \subseteq G_2$ action.
	
The second point is similar.
\end{proof}

We now set up the main theorem of this section. Let $F$ be a field of characteristic $0$, and  set $E = F \times F \times F$ with its usual $S_3$-action, given by permuting the factors.  We think of $E$ as a cubic norm structure.  So, the norm on $E$ is $N_E(z_1,z_2,z_3) = z_1 z_2 z_3$ and the adjoint on $E$ is $(z_1,z_2,z_3)^\# = (z_2 z_3, z_3 z_1,z_1 z_2)$. For $z,z' \in E$ one sets $z \times z' = (z+z')^\#-z^\#-(z')^\#$.  

Let $\g_E$ be the associated Lie algebra, as in \cite[section 4.2]{pollackQDS}.  We have
\[ \g_E = (\sl_3 \oplus E^0) \oplus V_3 \otimes E \oplus V_3^\vee \otimes E^\vee.\]
Here $V_3$ is the standard representation of $\sl_3$ and $V_3^\vee$ is the dual representation.  Moreover, $E^0$ is the trace $0$ subspace of $E$, and $E^0$ acts on $E$ via multiplication $\lambda \cdot x = \lambda x$, while $E^0$ acts on $E^\vee = E$ as $\lambda \cdot \gamma = - \lambda \gamma$.  The $S_3$ action on $E$ induces an $S_3$ action on $\g_E$.

We will recall the Lie bracket on $\g_E$.  Before doing so, for $u \in E^0$, let $\Psi_u \in E^0 \subseteq \g_E$ be the map $E \rightarrow E$ defined as $\Psi_u(z) = uz$.

For $X \in E$ and $\gamma \in E^\vee$, there is a map $\Phi_{\gamma,X}: E \rightarrow E$ defined as
\[\Phi_{\gamma,X}(Z) = - \gamma \times (X \times Z) + (\gamma,Z)X + (\gamma,X)Z.\]
Here $(\gamma,Z)$ and $(\gamma,X)$ denotes the canonical pairing between $E$ and the dual $E^\vee$.  One sets $\Phi_{\gamma,X} - \frac{2}{3}(X,\gamma)1_E =: \Phi'_{\gamma,X}.$ The definition of the Lie bracket on $\g_E$ from \cite{pollackQDS} uses the $\Phi'_{\gamma,X}$.

We require the following lemma, which relates the $\Phi'_{\gamma,X}$ with the $\Psi_u$.  
\begin{lemma}\label{lem:PhigammaX} Suppose $x \in E$, $\gamma \in E^\vee$.  Then $\Phi'_{\gamma,x} = \Psi_{2u}$ with $u =  x \gamma - \frac{1}{3}(x,\gamma) 1_E$.
\end{lemma}
\begin{proof} This is a direct computation:
\[\Phi'_{\gamma,x}(z) = -\gamma \times (x \times z) + (\gamma,z)x + \frac{1}{3}(x,\gamma) z.\]
Now $(\gamma,z)x = (x_1(\gamma_1 z_1+\gamma_2 z_2 + \gamma_3z_3),\ldots)$, $(x,\gamma)z = (z_1(x_1 \gamma_1 + x_2 \gamma_2 + x_3 \gamma_3),\ldots)$, and one has
\begin{align*} \gamma \times (x \times z) &= \gamma \times (x_2 z_3 + x_3 z_2, x_1 z_3 + x_3 z_1, x_1 z_2 + x_2 z_1) \\ &= ((x_1z_3+x_3z_1)\gamma_3 + (x_1 z_2 + x_2z_1)\gamma_2,\ldots).\end{align*}
Combining the above and using the symmetry between the indices gives the claim.
\end{proof}

We now explicate the Lie bracket on $\g_E$: Take $\phi_3 \in \sl_3$, $\phi_J \in E^0$, $v, v' \in V_3$, $\delta, \delta' \in V_3^\vee$, $X, X' \in E$ and $\gamma, \gamma' \in E^\vee$.  Then
\begin{align*} [\phi_3, v \otimes X + \delta \otimes \gamma] &= \phi_3(v) \otimes X + \phi_3(\delta) \otimes \gamma. \\ [\phi_J, v \otimes X+ \delta \otimes \gamma] &= v \otimes \phi_J(X) + \delta \otimes \phi_J(\gamma) \\ [v \otimes X,v' \otimes X'] &= (v \wedge v') \otimes (X \times X') \\ [\delta \otimes \gamma, \delta' \otimes \gamma'] &= (\delta \wedge \delta') \otimes (\gamma \times \gamma') \\ [\delta \otimes \gamma, v \otimes X] &= (X,\gamma) v \otimes \delta + \delta(v) \Phi_{\gamma,X} - \delta(v)(X,\gamma) \\ &= (X,\gamma) \left(v \otimes \delta -\frac{1}{3}\delta(v)\right) + \delta(v) \left(\Phi_{\gamma,X} - \frac{2}{3}(X,\gamma)\right).\end{align*}
Note that $v \otimes \delta -\frac{1}{3}\delta(v) \in \sl_3$ and $\Phi_{\gamma,X} - \frac{2}{3}(X,\gamma) = \Phi'_{\gamma,X} \in E^0$.  

Furthermore, the action of $\sl_3$ and $E^0$ on $V_3^\vee$ and $E^\vee$ is determined by the equalities $(\phi_3(v),\delta) + (v,\phi_3(\delta))=0$ and $(\phi_J(X),\gamma) + (X,\phi_J(\gamma)) = 0$.

For $j \neq k$, let $E_{jk} \in \sl_3$ be the matrix with a $1$ in the $(j,k)$ position and $0$'s elsewhere.  Let $v_1, v_2, v_3$ be the standard basis of $V_3$ and $\delta_1, \delta_2,\delta_3$ the dual basis of $V_3^\vee$. Denote by 
\begin{equation}
\label{triality-compatible-identification}
\Phi\colon \g_E \rightarrow \wedge^2 \O    
\end{equation} the linear isomorphism given as follows:
\begin{enumerate}
	\item $\Phi(E_{jk}) = e_k^* \wedge e_j$;
	\item $\Phi( v_j \otimes (x_1,x_2,x_3)) = x_1 \epsilon_1 \wedge e_j + x_2 e_{j+1}^* \wedge e_{j-1}^* + x_3 (-\epsilon_2 \wedge e_j)$;
	\item $\Phi( \delta_j \otimes (\gamma_1,\gamma_2,\gamma_3))= \gamma_1(-\epsilon_2 \wedge e_j^*) + \gamma_2 e_{j+1} \wedge e_{j-1} + \gamma_3 \epsilon_1 \wedge e_j^*$;
	\item on $E^0$, $\Phi$ is: $\Phi(\Psi_{2u}) = (u_1-u_3) \epsilon_1 \wedge \epsilon_2 + u_2 (e_1 \wedge e_1^* + e_2 \wedge e_2^* + e_3 \wedge e_3^*)$.
\end{enumerate}

\begin{theorem}\label{thm:spin8O} The linear map $\Phi$ is a Lie algebra isomorphism, respecting the $S_3$ actions.
\end{theorem}

\begin{proof}[Proof of Theorem~\ref{thm:spin8O}] There are lots of brackets to check.
\begin{enumerate}
	\item Suppose $\ell \neq k$.  Then $[\Phi(E_{\ell,k}),\Phi(v_{j} \otimes x)] = [e_k^*\wedge e_{\ell},  x_1 \epsilon_1 \wedge e_j + x_2 e_{j+1}^* \wedge e_{j-1}^* + x_3 (-\epsilon_2 \wedge e_j)] = 0$ if $j \neq k$ and is equal to $\Phi(v_{\ell} \otimes x)$ if $j = k$.
	\item Suppose $\ell \neq k$.  Then $[\Phi(E_{\ell,k}),\Phi(\delta_{j} \otimes \gamma)] = [e_k^*\wedge e_{\ell},  \gamma_1(-\epsilon_2 \wedge e_j^*) + \gamma_2 e_{j+1} \wedge e_{j-1} + \gamma_3 \epsilon_1 \wedge e_j^*] = 0$ if $\ell \neq j$ and is equal to $\Phi(-\delta_{k} \otimes \gamma)$ if $\ell = j$.  
	\item Suppose $\alpha_1+\alpha_2+\alpha_3 = 0$, and $h = \alpha_1 E_{11} + \alpha_2 E_{22} + \alpha_3 E_{33} \in \h_{\SL_3}$, the diagonal Cartan of $\sl_3$.  Then $[\Phi(h), \epsilon_1 \wedge e_j] = \alpha_j \epsilon_1 \wedge e_j$.  Using the $S_3$ action on $\wedge^2\O$, and that we already know $h \in \Lie(G_2) \subseteq (\wedge^2 \O)^{S_3}$, we obtain $[\Phi(h), \Phi(v_j \otimes x)] = \Phi(\alpha_j v_j \otimes x)$.  Similarly, we obtain $[\Phi(h), \Phi(\delta_j \otimes x)] = \Phi(-\alpha_j \delta_j \otimes x)$. 
	\item Note that $\epsilon_1 \wedge e_j, e_{j+1}^*\wedge e_{j-1}^*,$ and $-\epsilon_2 \wedge e_j$ all commute.  Thus $[\Phi(v_j \otimes x),\Phi(v_{j}\otimes x')] = 0$.  Similarly, $[\Phi(\delta_j \otimes \gamma), \Phi(\delta_j \otimes \gamma')] = 0$.
	\item We compute $[\Phi(v_j \otimes x),\Phi(v_{j+1}\otimes x')]$.  One has
	\begin{enumerate}
		\item $[\epsilon_1\wedge e_j,\epsilon_1 \wedge e_{j+1}] = 0$;
		\item $[\epsilon_1\wedge e_j,e_{j-1}^* \wedge e_{j}^*] = \epsilon_1 \wedge e_{j-1}^*$;
		\item $[\epsilon_1\wedge e_j,-\epsilon_2 \wedge e_{j+1}] = e_{j} \wedge e_{j+1}$;
	\end{enumerate}
	Thus $[\Phi(v_j \otimes (x_1,0,0)),\Phi(v_{j+1}\otimes x')] = \Phi(\delta_{j-1}\otimes (0,x_1 x_3',x_1 x_2'))$. Using the $S_3$ action on $\wedge^2 \O$, one deduces that $[\Phi(v_j \otimes x),\Phi(v_{j+1}\otimes x')] = \Phi(\delta_{j-1} \otimes (x \times x')).$
	\item We compute $[\Phi(\delta_j \otimes \gamma),\Phi(\delta_{j+1}\otimes \gamma')]$.  One has
	\begin{enumerate}
		\item $[-\epsilon_2\wedge e_j^*,-\epsilon_2 \wedge e_{j+1}^*] = 0$;
		\item $[-\epsilon_2\wedge e_j^*,e_{j-1} \wedge e_{j}] = -\epsilon_2 \wedge e_{j-1}$;
		\item $[-\epsilon_2\wedge e_j^*,\epsilon_1 \wedge e_{j+1}^*] = e_j^* \wedge e_{j+1}^*$;
	\end{enumerate}
	Thus $[\Phi(\delta_j \otimes (\gamma_1,0,0)),\Phi(\delta_{j+1}\otimes \gamma')] = \Phi(v_{j-1}\otimes (0,\gamma_1 \gamma_3',\gamma_1 \gamma_2'))$. Using the $S_3$ action on $\wedge^2 \O$, one deduces that $[\Phi(\delta_j \otimes \gamma),\Phi(\delta_{j+1}\otimes \gamma')] = \Phi(v_{j-1} \otimes (\gamma \times \gamma')).$
	\item $[\Phi(\Psi_{2u}),\Phi(v_j \otimes x)] = \Phi(v_j \otimes (2ux))$.  Indeed, we have
	\begin{align*} [\Phi(\Psi_{2u}),\Phi(v_j \otimes x)] &= [(u_1-u_3) \epsilon_1 \wedge \epsilon_2 + u_2 (e_1 \wedge e_1^* + e_2 \wedge e_2^* + e_3 \wedge e_3^*), \\ &\,\,\,\,\,\,\, x_1 \epsilon_1 \wedge e_j + x_2 e_{j+1}^* \wedge e_{j-1}^* + x_3 (-\epsilon_2 \wedge e_j)] \\ &= (u_1-u_3)(x_1 \epsilon_1 \wedge e_j + x_3 \epsilon_2 \wedge e_j) \\ &\,\,\,\,\,+ u_2(-x_1 \epsilon_1 \wedge e_j + 2x_2  e_{j+1}^* \wedge e_{j-1}^* + x_3 \epsilon_2 \wedge e_j) \\ &= 2\Phi(v_j \otimes ux).
	\end{align*}
	\item Similarly, one computes $[\Phi(\Psi_{2u}),\Phi(\delta_j \otimes \gamma)] = \Phi(\delta_j \otimes (-2u\gamma))$.
	\item It is immediately checked that $[\Phi(\phi),\Phi(\Psi_u)] = 0$ if $\phi \in \sl_3$ and $[\Phi(\Psi_u),\Phi(\Psi_{u'})] = 0$.
	\item We are left to compute $[\Phi(\delta \otimes \gamma),\Phi(v\otimes x)]$.  First observe that, by explicit computation,
	\[[e_{j+1}\wedge e_{j-1}, e_{k+1}^* \wedge e_{k-1}^*] = - e_k \wedge e_j^* + \delta_{jk}(e_1 \wedge e_1^* + e_2 \wedge e_2^* + e_3 \wedge e_3^*).\]
	Using this, one obtains
	\begin{align*} [\Phi(\delta_j \otimes \gamma),\Phi(v_k \otimes x)] &= [\gamma_1(-\epsilon_2 \wedge e_j^*) + \gamma_2 e_{j+1} \wedge e_{j-1} + \gamma_3 \epsilon_1 \wedge e_j^*,\\ &\,\,\,\,\,\,\,\,\, x_1 \epsilon_1 \wedge e_k + x_2 e_{k+1}^* \wedge e_{k-1}^* + x_3 (-\epsilon_2 \wedge e_k)] \\
	&=\gamma_1 x_1( e_j^* \wedge e_k + \delta_{jk}\epsilon_1 \wedge \epsilon_2) \\ &\,\,\,+ \gamma_2 x_2(- e_k \wedge e_j^* + \delta_{jk}(e_1 \wedge e_1^* + e_2 \wedge e_2^* + e_3 \wedge e_3^*))\\ &\,\,\,+ \gamma_3 x_3(e_j^* \wedge e_k - \delta_{jk} \epsilon_1 \wedge \epsilon_2)
	\\ &= (\gamma,x) e_j^* \wedge e_k + \delta_{jk}((\gamma_1x_1-\gamma_3 x_3)\epsilon_1 \wedge \epsilon_2 \\ &\,\,\,\,\,+\gamma_2 x_2(e_1 \wedge e_1^* + e_2 \wedge e_2^* + e_3 \wedge e_3^*))
	\end{align*}
	Set $I = e_1^* \wedge e_1 + e_2^* \wedge e_2 + e_3^* \wedge e_3$.  Then the above is
	\[(\gamma,x)(e_j^* \wedge e_k - \delta_{jk}I/3)+\delta_{jk}(\gamma_1 x_1 - \gamma_3 x_3) \epsilon_1 \wedge \epsilon_2 + \delta_{jk}(\gamma_2 x_2 - (\gamma,x)/3)(-I).\]
	But 
	\[[\delta_j \otimes \gamma,v_k \otimes x] = (\gamma,x)(E_{kj} - \delta_{jk}1_3/3) + \Phi_{\gamma, x}' = (\gamma,x)(E_{kj} - \delta_{jk}1_3/3) + \Psi_{2u}\]
	with $u = x\gamma - \frac{1}{3}(x,\gamma)1_E$.
	This finishes the proof that $\Phi$ is a Lie algebra homomorphism.
\end{enumerate}
To see that $\Phi$ is $S_3$ invariant: we have already checked this on everything but the $\Phi(\Psi_{2u})$.  But if $\sigma \in S_3$ and $u \in E^0$, we have
\begin{align*}\sigma(\Phi(\Psi_{2u})) &= \sigma(\Phi([\delta_1 \otimes u,v_1 \otimes 1_E])) = \sigma([\Phi(\delta_1 \otimes u),\Phi(v_1 \otimes 1_E)]) \\ &= [\sigma(\Phi(\delta_1 \otimes u)),\sigma(\Phi(v_1 \otimes 1_E))] = [\Phi(\delta_1 \otimes \sigma(u)),\Phi(v_1 \otimes 1_E)] \\&= \Phi(\Psi_{2\sigma(u)}).\end{align*}
This finishes the proof of the theorem.
\end{proof}

There is a Cartan involution $\Theta$ on $\g_E$ from \cite{pollackQDS} given as follows:
\begin{enumerate}
\item on $\sl_3$, it is $\Theta(X) = -X^t$
\item on $E^0$, it is given as $\Theta(\Phi_{\gamma,X}') = - \Phi'_{\iota(X),\iota(\gamma)}$.  Thus $\Theta(\Psi_{2u}) = - \Psi_{2u}$.
\item on $V_3 \otimes E$ it is $\Theta(v \otimes x) = \iota(v) \otimes \iota(x)$, and
\item on $V_3^\vee \otimes E^\vee$ it is $\Theta(\delta \otimes \gamma) = \iota(\delta) \otimes \iota(\gamma)$.
\end{enumerate}
Note that this Cartan involution commutes with the $S_3$ action on $\g_E$.

We also have a Cartan involution on $\wedge^2 \O$ given as $\Theta(u \wedge v) = \iota(u) \wedge \iota(v)$, where $\iota(b_j ) = b_{-j}$.
\begin{corollary}\label{cor:trialityCartan} The map $\Phi: \g_E \simeq \wedge^2\O$ respects the Cartan involution on each side.  In particular, the Cartan involution on $\wedge^2\O$ commutes with the $S_3$ action.\end{corollary}
\begin{proof} The proof is an immediate check from the above formulas.\end{proof}

We need one additional lemma.  Recall that $W$ denotes the space of automorphic characters on $N'$.  Let $\mathcal{W}_\chi: G' \rightarrow \mathbf{V}_{\ell}$ denote the generalized Whittaker function specified in \cite[Equations (1) and (2)]{pollackE8}.
\begin{lemma}\label{lem:trialityGWF} Suppose $\sigma \in S_3$ and $\chi \in W$. The generalized Whittaker function $\mathcal{W}_\chi$ on $G'$ satisfies $\mathcal{W}_{\sigma^{-1}(\chi)}(\sigma^{-1}(g)) = \mathcal{W}_\chi(g)$.
\end{lemma}
\begin{proof} Fix $\chi \in W$.  Define $W_1: G'(\R) \rightarrow \mathbf{V}_{\ell}$ as $W_1(g) = \mathcal{W}_{\sigma^{-1}(\chi)}(\sigma^{-1}(g))$.  We wish to prove that $W_1 = \mathcal{W}_\chi$ as functions on $G'(\R)$.  We will use the uniqueness result \cite[Theorem~1.2.1]{pollackQDS} or equivalently Theorem~\ref{Thm 1.2.1 Aaron Paper} (see Remark~\ref{rmk-FE-QMF-SPIN8}).
	
We have $W_1(gk) = k^{-1} W_1(g)$ by Theorem \ref{thm:trialityQMF} part (1) (which we have already proved).  Likewise, because $\sigma(\p') = \p'$, one has $D_{\ell}W_1 \equiv 0$.  If $n \in N'(\R)$, then
\[
W_1(n g) = \mathcal{W}_{\sigma^{-1}(\chi)}(\sigma^{-1}(n) \sigma^{-1}(g)) = \chi(n) \mathcal{W}_{\sigma^{-1}(\chi)}(\sigma^{-1}(g)) = \chi(n) W_1(g).
\]
It follows by uniqueness that $W_1 = \lambda \mathcal{W}_\chi$ for some $\lambda \in \C$.  However, if $g=1$, then from the explicit formula for $\mathcal{W}_\chi$ and $\mathcal{W}_{\sigma^{-1}(\chi)}$, one sees $W_1(1) = \mathcal{W}_\chi(1) \neq 0$.  This proves the lemma.
\end{proof}

\begin{proof}[Proof of Theorem~\ref{thm:trialityQMF}] Everything is now straightforward, using Theorem~\ref{thm:spin8O}, Corollary~\ref{cor:trialityCartan}, and Lemma~\ref{lem:trialityGWF}.\end{proof}

\section{An integral orbit problem}\label{sec:orbits}
In this section, we study an integral orbit problem that is needed for the work on periods in Section~\ref{sec:periods}.  More specifically, we prove Theorem~\ref{thm:AITSO8} below which is then used in the proof of Theorem~\ref{thm:introPeriod}.  

We set up this theorem now.  Suppose $L = \Z^{2n}$ is the standard split lattice inside of the split $2n$-dimensional quadratic space $V$.  Let $G(\Z)$ denote the stabilizer of $L$ inside of $G(\Q)$, where $G$ is the special orthogonal group of $V$.  Set $(v_1,v_2) = q(v_1+v_2)-q(v_1)-q(v_2)$ the split bilinear form on $V$.  Recall that if $T_1, T_2 \in V$, we set 

\[S(T_1,T_2) = \frac{1}{2} \left(\begin{array}{cc} (T_1,T_1) & (T_1,T_2) \\ (T_2,T_1) & (T_2,T_2)\end{array}\right).\]
	
Note that if $T_1,T_2 \in L$, then $S(T_1,T_2)$ is a half-integral symmetric matrix.  For a general half-integral symmetric matrix $T$, set
\[X_T = \{(T_1,T_2) \in L^2: S(T_1,T_2) = T\}.\]

\begin{theorem}\label{thm:AITSO8} Suppose $n \geq 4$.  Let $T$ be a half-integral symmetric matrix, and assume $D = - 4 \det(T)$ is odd and square-free.  Then $G(\Z)$ acts transitively on $X_T$.
\end{theorem}

Define a bilinear form on $\wedge^2 V$ as 
\[(x_1 \wedge x_2, y_1 \wedge y_2) = (x_1,y_2)(x_2,y_1) - (x_1,y_1)(x_2,y_2).\]
Note that $(T_1 \wedge T_2, T_1 \wedge T_2) = -4\det(S(T_1,T_2)).$  To prove Theorem~\ref{thm:AITSO8} we will use the following lemmas.

\begin{lemma}\label{lem:AIT1} Suppose $n \geq 3$, and  suppose $v \in L$ is primitive, with $q(v) = a$.  Then there exists $g \in G(\Z)$ so that $gv = a b_1 + b_{-1}$.
\end{lemma}
\begin{proof} Throughout the proof, we set $X = \mathrm{Span}(b_1,\ldots,b_n)$ and $Y = \mathrm{Span}(b_{-1}, \ldots, b_{-n})$.
Suppose $v = u + w$, with $u \in X$ and $w \in Y$.  Using the $\GL(n,\Z)$ action from the Levi of the Siegel parabolic subgroup, we can assume $u = a' b_1$ for an integer $a'$.  If $a' = 0$, the lemma follows by using the $\GL(n,\Z)$ action on $w$.  Suppose then that $a' \neq 0$.  We use the $\GL(n-1,\Z) \subseteq \GL(n,\Z)$ that fixes $b_1, b_{-1}$ to move $w$ to an element of the form $w = w_1 b_{-1} + w_2 b_{-2}$.  
	
Observe that $\gcd(w_1,w_2,a') = 1$.  Now, we can apply a unipotent element in the unipotent group opposite to the Siegel parabolic subgroup to move $w$ to $w' = w_1 b_{-1} + w_{2} b_{-2} + a' b_{-3}$.  Thus $w'$ is primitive, so we can use the $\GL(n,\Z)$ action to move $w'$ to $b_{-1}$.  Hence we have found a $g \in G(\Z)$ so that $g v = u' + b_{-1}$ with $u' \in X$.  We can now finish the proof by using elements in the unipotent radical of the Siegel parabolic subgroup to move $u'$ to an element of the form $a b_1$.
\end{proof}

\begin{lemma}\label{lem:AIT2} Suppose $n \geq 4$, and suppose $v_1,v_2 \in L$ with $D:=-4\det(S(v_1,v_2))$ odd and square-free.  Then there exists $u_1, u_2 \in L$ with $\mathrm{Span}(u_1,u_2)$ isotropic and $(v_1 \wedge v_2, u_1 \wedge u_2) = 1$.
\end{lemma}
\begin{proof} One first verifies that there exists $g \in G(\Z)$ so that $g v_1 = a b_1 + b_{-1}$ and $g v_2 = r b_1 + s b_{-1} + m(\beta b_2 + b_{-2})$.  Indeed, we first apply Lemma~\ref{lem:AIT1} to move $v_1$ into the specified form.  We then apply Lemma~\ref{lem:AIT1} again for the subgroup of $G(\Z)$ stabilizing $b_1, b_{-1}$ to move $v_2$ into the specified form.
	
Note that if we set $\alpha = a s -r$, then we have $D = \alpha^2 - 4 m^2 a \beta$.
	
Now, if $v_1, v_2$ are in the above specialized form, then there exists $u_1, u_2$ spanning an isotropic two-plane so that $(v_1 \wedge v_2, u_1 \wedge u_2) = 1$.  Indeed, because $D = \alpha^2 - 4 m^2 a \beta$,  the integers $\alpha$ and $m$ are relatively prime, so there exists integers $x,y$ so that $\alpha x- m y =1$.  Now one verifies $(v_1 \wedge v_2, u_1 \wedge u_2) = 1$, where $u_1 = b_1 + b_3$ and $u_2 = x b_{-1} + yb_2 - x b_{-3}$.
\end{proof}

\begin{lemma}\label{lem:AIT3} Suppose $n \geq 3$, $u_1, u_2$ are isotropic and $u_1 \wedge u_2$ is primitive in $\wedge^2 L$.  Then there exists $g \in G(\Z)$ so that $g u_1 = b_1$ and $g u_2 = b_2$.
\end{lemma}
\begin{proof} Recall that $P$ is the parabolic subgroup stabilizing $\mathrm{Span}(b_1,b_2)$.  One can leverage the Iwasawa decomposition to check that $G(\Q) = P(\Q) G(\Z)$.  Now, by Witt's theorem, there exists $g = pk \in G(\Q)$ so that $g u_1 = b_1$ and $g u_2 = b_2$.  Here $p \in P(\Q)$ and $k \in G(\Z)$.  Thus $ku_1, ku_2 \in \mathrm{Span}(b_1,b_2)$ and $(ku_1) \wedge (ku_2)$ is primitive.  One now finishes the proof by using an element in $\GL(2,\Z)$ in the Levi subgroup of the Heisenberg parabolic subgroup.
\end{proof}

\begin{proof}[Proof of Theorem~\ref{thm:AITSO8}]
By Lemma~\ref{lem:AIT2} and Lemma~\ref{lem:AIT3}, we can assume $(T_1 \wedge T_2, b_1 \wedge b_2) = 1$.   Say $T_j = x_j + y_j$ with $x_j \in X$, $y_j \in Y$.

Observe that $1 = (b_1 \wedge b_2, y_1 \wedge y_2)$.   Thus $y_1 \wedge y_2$ is primitive in $\wedge^2 Y$, so we may use the action of $\GL(n,\Z)$ from a Levi factor of the Siegel parabolic subgroup to assume $y_1 = b_{-1}$, $y_2 = b_{-2}$.

Suppose now 
\[S(T_1,T_2) = \left(\begin{array}{cc} a & b/2 \\ b/2 & c \end{array}\right).\]
Let $N$ denote the unipotent radical of the Siegel parabolic subgroup of $G$.  Then there exists $n \in N(\Z)$ so that $n T_1 = a b_1 + b_{-1}$ and $n T_2 = b b_1 + c b_2 + b_{-2}$.  This claim follows easily by writing down the action of $N(\Z)$.

The theorem is proved.
\end{proof}

\section{The finiteness of an integral} \label{sec:finiteness_appendix}
In this section, we prove the finiteness of an integral that arises in the work on periods in Section~\ref{sec:periods}.  For ease of notation, set $H = H_{v_1,v_2}$.  We prove the following theorem.  This theorem is used in the proof of Theorem~\ref{thm:PeriodQ}, which is then used in the proof of Theorem~\ref{thm:introPeriod}.

\begin{theorem}\label{thm:finitenessInt} Suppose $\ell \geq 22$.  Then the integral
	\[\int_{H(\R)\backslash G(\R)}{ ||B_{[v_1,v_2]}(g)|| \,dg}\]
is finite.
\end{theorem}
Observe that one can identify the quotient space $H(\R)\backslash G(\R)$ with pairs of vectors $(v_1',v_2') \in V^2$ for which $S(v_1',v_2') = S(v_1,v_2)$.  To prove the finiteness of the integral, we work to put the invariant measure $dg$ in these coordinates. 

\textbf{Notation}: Throughout this section, we use the notation $A \approx B$ to mean that there is a nonzero constant $\beta$ so that $A = \beta B$.  

We begin by understanding some differential forms on a vector space.  Suppose $V = \R^n$, with coordinates $(z_1,z_2, \ldots, z_n)$.  We will later take $n=8$, but for now, we work more generally.

\begin{proposition}\label{propC2} We have the following facts concerning forms on $V$.
\begin{enumerate}
	\item $\omega_{[1]} := d(z_1^2 + \cdots + z_n^2) = \sum_{j}{2 z_j dz_j}$ is $\SO(n)$-invariant.
	\item Set $\omega_{[n]} = dz_1 \wedge \cdots \wedge dz_n$ and $\eta_j = (-1)^{j-1} dz_1 \wedge \cdots \wedge \widehat{dz_j} \wedge \cdots \wedge dz_n$, so that $dz_j \wedge \eta_j = \omega_{[n]}$.  Let $\omega_{[n-1]}:= \sum_{j}{z_j \eta_j}$. Then $\omega_{[n-1]}$ is $\SL(n,\R)$ invariant.  In fact, if $g \in \GL(n,\R)$, then $g^* \omega_{[n-1]} = \det(g) \omega_{[n-1]}$.
	\item $\omega_{[1]} \wedge \omega_{[n-1]} = 2(\sum_{j}{z_j^2}) \omega_{[n]}$.
\end{enumerate}
\end{proposition}
\begin{proof} One interprets $\eta_j \in \wedge^{n-1}(V) \simeq \det(V) V^\vee$.  Then, if $\delta_j$ is dual to $z_j$, we have that $\sum_{j}{z_j \otimes \delta_j}$ is $\GL(n,\R)$-invariant.  The proposition follows.
\end{proof}

We construct invariant differential forms on $V^2$.  We use the variables $(w_1, \ldots, w_n)$ for the coordinates on the first copy of $V$, and the $z_j$'s as coordinates on the second copy of $V$.  Set $q_{11} = \sum_{j}{w_j^2}$, $q_{12} = \sum_{j}{w_j z_j}$, and $q_{22} = \sum_j{z_j^2}$.  To distinguish between differential forms defined using the coordinates $w$ from forms defined using $z$, we use superscripts.  So, we let $\omega_{[n-1]}^z$ denote the form $\omega_{[n-1]}$ defined in Proposition~\ref{propC2}, and let $\omega_{[n-1]}^w = \sum_j{w_j \eta_j^w}$, where $\eta_j^w= (-1)^{j-1} dw_1 \wedge \cdots \wedge \widehat{dw_j} \wedge \cdots \wedge dw_n$.

Now, if $j < k$, set $\eta_{j,k}^z = (-1)^{j+k-1} dz_1 \wedge \cdots \wedge \widehat{dz_j} \wedge \cdots \wedge \widehat{dz_k} \wedge \cdots \wedge dz_n$, so that $dz_j \wedge dz_k \wedge \eta_{j,k}^z = \omega_{[n]}^z$.  Similarly define $\eta_{j,k}^w$.  Now let
\[\omega_{[n-2]}^{w,z} = \sum_{1 \leq j < k \leq n}{(w_j z_k - w_k z_j) \eta_{j,k}^{w}}\]
and $\omega_{[2n-3]}^{w,z} = \omega_{[n-2]}^{w,z} \wedge \omega_{[n-1]}^{z}.$

\begin{proposition} We have the following facts concerning differential forms on $V^2$.
\begin{enumerate}
	\item The $q_{ij}$ are $\SO(n)$-invariant.
	\item The $(n-2)$-form $\omega_{[n-2]}^{w,z}$ is $\SL(n,\R)$-invariant.
\end{enumerate}
\end{proposition}
\begin{proof} We leave this to the reader.
\end{proof}

Set $\omega_{[n-1]}^{w,z'} = \sum_{j}{z_j \eta_j^w}$.  This is again an $\SL(n,\R)$-invariant differential form.  Let $\det(Q) = q_{11}q_{22}-q_{12}^2$.
\begin{proposition} One has the following equalities:
\begin{enumerate}
	\item $2q_{22}\det(Q) \omega_{[n-1]}^w \wedge \omega_{[n-1]}^{z} = (-q_{11}q_{22}dq_{11} + 2q_{11}q_{22} dq_{12}-q_{11}q_{12} dq_{22}) \wedge \omega_{[2n-3]}^{w,z}$.
	\item $2q_{22}\det(Q) \omega_{[n-1]}^{w,z'} \wedge \omega_{[n-1]}^{z} = (-q_{22}^2dq_{11} + 2q_{12}q_{22} dq_{12}-q_{12}^2 dq_{22}) \wedge \omega_{[2n-3]}^{w,z}$.
 \item $2q_{22}\det(Q) \omega_{[n-2]}^{w,z} \wedge \omega_{[n]}^{z} = (-1)^n \det(Q) dq_{22} \wedge \omega_{[2n-3]}^{w,z}$.
\end{enumerate}
\end{proposition}
\begin{proof} One computes $dq_{ij} \wedge \omega_{[2n-3]}^{w,z}$ in terms of $\omega_{[n-1]}^w \wedge \omega_{[n-1]}^{z}, \omega_{[n-1]}^{w,z'} \wedge \omega_{[n-1]}^{z}, \omega_{[n-2]}^{w,z} \wedge \omega_{[n]}^{z}$, and then inverts the expression.
\end{proof}

\begin{lemma}\label{lem:dqsW} One has
\begin{enumerate}
	\item $4q_{22}\det(Q) \omega_{[n]}^w \wedge \omega_{[n]}^z = (-1)^n d q_{11} \wedge d q_{12} \wedge d q_{22} \wedge \omega_{[2n-3]}^{w,z}.$
\end{enumerate}
\end{lemma}
\begin{proof} From the previous proposition, $(-1)^n d q_{22} \wedge \omega_{[2n-3]}^{w,z} = 2 q_{22} \omega_{[n-2]}^{w,z} \wedge \omega_{[n]}^{z}$.  Now, wedging with $dq_{11} \wedge dq_{12}$, we get
\[dq_{11} \wedge dq_{12} \wedge \omega_{[n-2]}^{w,z} \wedge \omega_{[n]}^{z} = 2\left(\sum_{j',k'}{w_{j'} z_{k'} dw_{j'} \wedge dw_{k'}}\right) \wedge \left(\sum_{j,k}{(w_j z_k - w_k z_j) \eta_{j,k}^{w}}\right) \wedge \omega_{[n]}^{z}.\]
It is now a straightforward computation.
\end{proof}

Suppose $V = V_a^{+} \oplus V_b^{-}$ is a quadratic space of signature $(a,b)$, so that $a+b=n$.  In our application, $a=b=4$.  We let $u_1,\ldots, u_a$, $v_1,\ldots, v_a$ be the coordinates on the two copies of $V_a^{+}$ in $V^2$, and $x_1,\ldots,x_b$, $y_1,\ldots,y_b$ be the coordinates on the two copies of $V_b^{-}$ in $V^2$. Let $r_{11} = \sum_{j}{x_j^2}$, $r_{12} = \sum_{j}{x_j y_j}$ and $r_{22} = \sum_{j}{y_j^2}$.  Similarly, let $t_{11} = \sum_{j}{u_j^2}$, $t_{12} = \sum_{j}{u_j v_j}$, and $t_{22} = \sum_{j}{v_j^2}$.  Finally, set $s_{ij} = t_{ij}-r_{ij}$.  Let $T,R,S$ be the two-by-two matrices with entries $t_{ij}, r_{ij}, s_{ij}$, so that $S = T-R$.  We set 
\[V^2_S = \{(v_1',v_2') \in V^2: ((v_i',v_j')) = S\}\]
to be the collection of vectors with Gram matrix $S$.

\begin{theorem}\label{thm:V2Sdiff} Suppose $S$ is positive-definite.  Then there is a nonzero cubic homogeneous polynomial $P(R,T)$ in the variables $r_{ij}$ and $t_{ij}$ so that $4 r_{22} t_{22}\det(R)\det(T) \omega_{[2n-3]}^{w,z}$ has the same restriction to $V^2_S$ as $P(R,T) dr_{11} \wedge dr_{12} \wedge dr_{22} \wedge \omega_{[2a-3]}^{u,v} \wedge \omega_{[2b-3]}^{x,y}$.
\end{theorem}
\begin{proof} We have $\omega_{[2n-3]}^{w,z}= \omega_{[n-2]}^{w,z} \wedge \omega_{[n-1]}^{z}$.  Now, 
\[ \omega_{[n-1]}^z = \omega_{[a-1]}^v \wedge \omega_{[b]}^y + (-1)^a\omega_{[a]}^{v} \wedge \omega_{[b-1]}^y.\]
Also
\[
\omega_{[n-2]}^{w,z} = \omega_{[a-2]}^{u,v} \wedge \omega_{[b]}^{x} + \omega_{[a]}^{u} \wedge \omega_{[b-2]}^{x,y} + (-1)^{a-1}\sum_{1 \leq j \leq a, 1 \leq k \leq b}{(u_j y_k- v_j x_k) \eta_j^{u} \wedge \eta_k^{x}}.
\]

Observe that $d t_{ij} = dr_{ij}$ when restricted to $V^2_S$.  We can use formulas above to put our $2n-3$ form in the desired shape.  For example,
\begin{align*}\omega_{[a-2]}^{u,v} \wedge \omega_{[a]}^v \wedge \omega_{[b]}^{x} \wedge \omega_{[b-1]}^{y}
&\approx \frac{1}{t_{22}} d t_{22} \wedge \omega_{[2a-3]}^{u,v} \wedge \omega_{[b]}^{x} \wedge \omega_{[b-1]}^{y}\\
&\approx \frac{1}{t_{22}} \omega_{[2a-3]}^{u,v} \wedge \omega_{[b]}^{x} \wedge (r_{22}\omega_{[b]}^{y})\\
&\approx \frac{r_{22} \det(T)}{r_{22}t_{22}\det(R)\det(T)} dr_{11} \wedge dr_{12} \wedge dr_{22} \wedge \omega_{[2a-3]}^{u,v} \wedge \omega_{[2b-3]}^{x,y}.\end{align*}

As another example,
\begin{align*}
\sum_{j,k}{u_j v_k \eta_j^u \wedge \eta_k^x} \wedge \omega_{[a-1]}^{v} \wedge \omega_{[b]}^{y} \\
\approx \omega_{[a-1]}^{u} \wedge \omega_{[a-1]}^{v} \wedge \omega_{[b-1]}^{x,y'} \wedge \omega_{[b]}^{y}.
\end{align*}
Then $\omega_{[a-1]}^{u} \wedge \omega_{[a-1]}^{v}  = \frac{1}{t_{22}\det(T)} \alpha(t) \wedge \omega_{[2a-3]}^{u,v}$, where $\alpha(t)$ is a linear combination of the $dt_{ij}$ with quadratic coefficients.  Additionally, one has that
\[\omega_{[b-1]}^{x,y'} \wedge \omega_{[b]}^{y} \approx \frac{1}{r_{22}} dr_{22} \wedge \omega_{[b-1]}^{x,y'} \wedge \omega_{[b-1]}^{y} \approx \frac{1}{r_{22}\det(R)} \beta(r) \wedge \omega_{[2b-3]}^{x,y}\]
where $\beta(r)$ is a 2-form in the $dr_{ij}$ with linear coefficients.

The other terms are similar to one of the above two examples.  This completes the proof, except for the non-vanishing of $P(R,T)$.  However, it is easy to see using Lemma~\ref{lem:dqsW} that $\omega_{[2n-3]}^{w,z}$ has nonzero restriction to $V^2_S$, so the theorem follows.
\end{proof}

To go further, we parameterize $V^2_S$ explicitly in terms of the Borel subgroup $B(\R)$ of $\GL(2,\R)$ and compact sets.  First, let $B'(\R)$ be the subgroup of $B(\R)$ with positive diagonal entries.  Next, denote $A_1 = \{(a_1,a_2) \in V_a^2: T(a_1,a_2) = 1_2\}$. Here $T(a_1,a_2)$ is the $2\times 2$ matrix with entries the inner products $(a_i,a_j)$.  Similarly, denote $B_1 = \{(b_1,b_2) \in V_b^2: R(b_1,b_2) = 1_2\}$. Finally, set $V^2_{S,o}$ to be the open subset of $V^2_S$ consisting of pairs $((u,x),(v,y))$ with $\det(R(x,y)) \neq 0$.  We define a diffeomorphism $\Psi: B'(\R) \times A_1 \times B_1 \rightarrow V^2_{S,o}$ as follows.

Given $b \in B'(\R)$, let $r(b) = b b^t$, $t(b) = S + b b^t$, $r_2(b) = r(b)^{1/2}$, and $t_{2}(b) = t(b)^{1/2}$.  Here the square roots are the unique symmetric positive definite ones.  Then
\[\Psi(b,(a_1,a_2),(b_1,b_2)) = (r_2(b) (a_1,a_2)^t, t_2(b) (b_1,b_2)^t).\]

We require the following lemma.
\begin{lemma}\label{lem:sigmaA1} Suppose $g \in \GL(2,\R)$.  Let $T = g g^t$, and let $\sigma_{A}$ be the pullback of $\omega_{[2a-3]}^{u,v}$ to $A_1$.  Then the pullback of $g^* \omega_{[2a-3]}^{u,v} = \omega_{[2a-3]}^{g (u,v)^t}$ to $A_1$ is $t_{22} \det(T)^{a-1/2} \sigma_A$.
\end{lemma}
\begin{proof}  The action of $g = (g_{ij})$ on the variables is $u \mapsto g_{11}u + g_{12}v$ and $v \mapsto g_{21}u + g_{22} v$.	
	
We write $g = bk$.  First observe that $k^* \omega_{[2a-3]}^{u,v}$ pulls back to $\sigma_A$.  Indeed, the action of $\SO(2)$ preserves $A_1$, and so $k^* \sigma_A$ is another $\SO(a)$-invariant $(2a-3)$-form on $A_1$ (which has dimension $2a-3$), so $k^* \sigma_A$ is proportional to $\sigma_A$.  This defines a continuous homomorphism $\SO(2) \rightarrow \R^\times$, which therefore must be trivial.  Thus $k^* \sigma_A = \sigma_A$.

For the pullback of $b^* \omega_{[2a-3]}^{u,v}$, it is easy to compute explicitly.  If $b = (b_{ij})$, then $u \mapsto b_{11}u + b_{12} v$ and $v \mapsto b_{22}v$.  Thus $u_j \mapsto b_{11}u_j + b_{12}v_j$ and $v_j \mapsto b_{22}v_j$.  We obtain that $b^*\omega_{[2a-3]}^{u,v} = b_{22}^2 \det(b)^{a-1} \omega_{[2a-3]}^{u,v} + \epsilon$, where $\epsilon$ is divisible by $\omega_{[a]}^{v}$.  But, because $v_1^2 + \cdots + v_a^2 = 1$ is fixed on $A_1$, $\omega_{[a]}^{v}$ pulls back to $0$.

Combining the above, and the fact that the $k$ action commutes with pullback, we obtain the lemma.
\end{proof}

One computes that if $R = b b^t$, then
\[ dr_{11} \wedge d r_{12} \wedge d r_{22} = 4b_{22} \det(b) d b_{11} \wedge d b_{12} \wedge d b_{22}.\]
We can now compute $\Psi^* \omega_{[2n-3]}^{w,z}$.
\begin{theorem} \label{thm:Measure1} One has $\Psi^* \omega_{[2n-3]}^{w,z} = \frac{1}{4}P(R,T)\det(T)^{a-3/2} \det(R)^{b-3/2} d r_{11} \wedge d r_{12} \wedge d r_{22} \wedge \sigma_A \wedge \sigma_B$, where $P(R,T)$ is a cubic homogeneous polynomial in the variables $r_{ij}, t_{ij}$.
\end{theorem}
\begin{proof} One has $\omega_{[2a-3]}^{t_2(b)(u,v)^t}$ is $t_2(b)^* \omega_{[2a-3]}$ plus terms involving the $db_{ij}$.  Similarly for $\omega_{[2b-3]}^{r_2(b)(x,y)^t}$.  Thus the theorem follows from Lemma~\ref{lem:sigmaA1} and Theorem~\ref{thm:V2Sdiff}.
\end{proof}

\begin{lemma}\label{lem:posDmeasure} There is a positive constant $D$, that only depends upon $S$, so that $|r_{ij}| \leq D \det(T)$ and $|t_{ij}| \leq D \det(T)$.  Consequently, there is a positive constant $C$ so that $|P(R,T)| \leq C \det(T)^3$.
\end{lemma}
\begin{proof} We have $r_{11}, r_{22} \leq \tr(R) \leq \tr(T)$, and $t_{11}, t_{22} \leq \tr(T)$.  Additionally, $r_{12}^2 \leq r_{11}r_{22}$, so $|r_{12}| \leq \tr(R)/2 \leq \tr(T)$ (by AM-GM) and similarly $|t_{12}| \leq \tr(T)$.  So it suffices to prove that $\tr(T) \leq D \det(T)$ for some $D > 0$.  To do this, let $\epsilon > 0$ be the minimum of the quadratic form $v^t S v$ on the unit circle in the plane.  Let $\lambda_1, \lambda_2$ be the eigenvalues of $T$, with corresponding unit eigenvectors $e_1, e_2$.  Then since $T \geq S$,  $\lambda_j = e_j^t T e_j \geq e_j^t S e_j \geq \epsilon$.  Consequently, 
\[\tr(T) = \lambda_1 + \lambda_2 = \lambda_1 \lambda_2 (\lambda_1^{-1} + \lambda_2^{-1}) \leq 2 \epsilon^{-1} \det(T).\]
This proves the lemma.
\end{proof}

\begin{proof}[Proof of Theorem~\ref{thm:finitenessInt}]
We will prove that if $\ell > 2n+5$, then the integral of $||B_{[v_1,v_2]}(g)||$ over $H(\R)\backslash G(\R)/K$ is finite.	
	
We have $\langle B_{[v_1,v_2]}(g), B_{[v_1,v_2]}(g) \rangle \approx \det(T)^{-(\ell+1)}$, so $||B_{[v_1,v_2]}(g)|| \approx \det(T)^{-(\ell+1)/2}$.  By Theorem~\ref{thm:Measure1} and Lemma~\ref{lem:posDmeasure}, the measure $dg$ on $H(\R)\backslash G(\R)/K$ is bounded by a positive constant times $\det(T)^{n} d t_{11} \wedge d t_{12} \wedge d t_{22}$.  We therefore must bound 
\[\int_{T \geq S}{\det(T)^{n-(1+\ell)/2}  d t_{11} \wedge d t_{12} \wedge d t_{22}}.\]

Now $d^* T:= \det(T)^{-3/2} dT$, where $dT = |d t_{11} \wedge d t_{12} \wedge dt_{22}|$, is a $\GL(2,\R)$ invariant measure on the positive definite cone.  We therefore must bound $\int_{T \geq S}{ \det(T)^{n+1-\ell/2} d^*T}$.  Making a change of variables, it suffices to bound $\int_{T \geq 1}{\det(T)^{n+1-\ell/2} d^* T} = \int_{T \geq 1}{\det(T)^{n-(1+\ell)/2} d T}$.  

We write $T = 1 + b b^t$, so that
\[ \det(T) = 1 + \tr(b b^t) + \det(b)^2 = 1 + b_{12}^2 + b_{11}^2 + b_{22}^2 + b_{11}^2 b_{22}^2 \geq 1 + b_{11}^2 + b_{12}^2 + b_{22}^2.\]
We thus must bound
\[ \int_{b_{11}, b_{22} \geq 0, b_{12} \in \R}{ (1+b_{11}^2+b_{22}^2 + b_{12}^2)^{n-(1+\ell)/2} b_{11} b_{22}^2 d b_{11} \wedge d b_{12} \wedge d b_{22}}.\]
Let $r = (b_{11}^2+b_{22}^2+b_{12}^2)^{1/2}$.  As $r \rightarrow \infty$, the integrand decays as $r^{-\ell +2n+2}$.  Hence if $\ell - 2n-2 > 3$, the integral converges.  This proves the theorem.
\end{proof}

\bibliography{nsfANT2020new}
\bibliographystyle{amsalpha}
\end{document}